\documentclass[fleqn,11pt,oneside]{article}

\usepackage{soul}
\usepackage{hyperref}
\usepackage[final]{graphicx}
\usepackage{bbm}  %
\usepackage{bm}  %
\usepackage{amsmath,amsthm,amssymb,amscd}
\usepackage{caption,subcaption}
\usepackage{booktabs,multirow}
\usepackage{setspace} 
\usepackage[top=1.1in, bottom=1.1in, left=1.6in, right=1.1in]{geometry}
\usepackage[inline,final]{showlabels}
\usepackage{dashbox}
\usepackage{microtype}
\usepackage[medium]{titlesec}

\newtheorem{theorem}{Theorem}[section]
\newtheorem{lemma}[theorem]{Lemma}
\newtheorem{corollary}[theorem]{Corollary}
\newtheorem{definition1}{Definition}[section]
\newtheorem{observe}{Observation}[section]
\newtheorem{remark1}[observe]{Remark}
\newtheorem{example1}{Example}[section]
\newtheorem{aside1}[observe]{Aside}

\newenvironment{definition}[1][]{\begin{definition1}[#1] \rm}{\end{definition1}}
\newenvironment{observation}{\begin{observe} \rm}{\end{observe}}
\newenvironment{remark}{\begin{remark1} \rm}{\end{remark1}}

\def\qed{\hfill$\blacksquare$\\} \renewenvironment{proof}{\noindent {\bf 
Proof.}}{\qed}

\newif\ifshowboxes \showboxestrue

\providecommand{\e}[1]{\ensuremath{\times 10^{#1}}}

\newcommand{\spn}{\mathop{\mathrm{span}}}

\renewcommand\Re{\operatorname{Re}}

\newcommand{\inner}[2]{\ensuremath{ {\langle #1,#2 \rangle} }}
\newcommand{\ai}{\text{Ai}}
\renewcommand{\d}{\,\mathrm{d}}
\newcommand{\dd}[2]{\frac{d #1}{d #2}}
\newcommand{\ddd}[2]{\frac{d^2 #1}{d #2 ^2}}
\newcommand{\pp}[2]{\frac{\partial #1}{\partial #2}}

\newcommand{\norm}[1]{\ensuremath{ {\lVert #1 \rVert} }}

\newcommand{\bbnorm}[1]{\ensuremath{ {\biggl\lVert #1 \biggr\rVert} }}

\newcommand{\abs}[1]{\ensuremath{ {\lvert #1 \rvert} }}

\def\C{\mathbbm{C}}
\def\R{\mathbbm{R}}
\def\N{\mathbbm{N}}
\def\1{\mathbbm{1}}

\def\G{\mathcal{G}}
\def\O{\mathcal{O}}
\def\T{\mathcal{T}}
\def\L{\mathcal{L}}

\def\cA{\mathcal{A}}
\def\cK{\mathcal{K}}
\usepackage{color} %
\usepackage{xcolor}
\usepackage{xspace}
\renewcommand{\tilde}{\widetilde}

\begin{document}

\begin{center}
    \begin{minipage}[t]{6.0in}

The distributions of the $k$-th largest level at the soft edge scaling limit
of Gaussian ensembles are some of the most important distributions in random
matrix theory, and their numerical evaluation is 
a subject of great practical importance. One numerical
method for evaluating the distributions uses the fact that they
can be represented as Fredholm determinants involving the
so-called Airy integral operator. When the spectrum of the integral operator
is computed by discretizing it directly, the eigenvalues are known to at
most absolute precision. Remarkably, the Airy integral operator is an
example of a so-called bispectral operator, which admits a commuting
differential operator that shares the same eigenfunctions. In this
paper, we develop an efficient numerical algorithm for evaluating the
eigendecomposition of the Airy integral operator to full relative precision,
using the eigendecomposition of the commuting differential operator. This
allows us to rapidly evaluate the distributions of the $k$-th largest level
to full relative precision rapidly everywhere, except in the left tail, where
they are computed to absolute precision. In addition, we characterize the
eigenfunctions of the Airy integral operator, and describe their extremal
properties in relation to an uncertainty principle involving the Airy
transform. We observe that the Airy integral operator is fairly universal,
and we describe a separate application to Airy beams in optics. 

 \noindent
{\bf Keywords:}
{\it Airy integral operator; Eigendecomposition; Random Matrix Theory;
Gaussian ensembles; Finite-energy Airy beam; Propagation-invariant optical fields;
Bispectral operator}

   \vspace{ -100.0in}
 
 \thispagestyle{empty}

   \end{minipage}
 \end{center}
 
 \vspace{ 3.00in}
 \vspace{ 0.80in}
 
 \begin{center}
   \begin{minipage}[t]{4.4in}
     \begin{center}
 
 \textbf{On the Evaluation of the Eigendecomposition of the Airy Integral Operator} \\
 
   \vspace{ 0.50in}
 
 Zewen Shen$\mbox{}^{\dagger\, \star}$ and
 Kirill Serkh$\mbox{}^{\ddagger\, \diamond}$  \\
               v4, updated June 17, 2022
 
     \end{center}
   \vspace{ -100.0in}
   \end{minipage}
 \end{center}
 
 \vspace{ 2.00in}

 \vfill
 
 \noindent 
 $\mbox{}^{\diamond}$  This author's work was supported in part by the NSERC
 Discovery Grants RGPIN-2020-06022 and DGECR-2020-00356.
 \\

 \vspace{2mm}
 
 \noindent
 $\mbox{}^{\dagger}$ Dept.~of Computer Science, University of Toronto,
 Toronto, ON M5S 2E4\\
 \noindent
 $\mbox{}^{\ddagger}$ Dept.~of Math. and Computer Science, University of Toronto,
 Toronto, ON M5S 2E4 \\
 
 \vspace{2mm}
 \noindent 
 $\mbox{}^{\star}$  Corresponding author
 \\

 \vfill
 \eject
\tableofcontents

\section{Introduction}
Recently, random matrix theory (RMT) has become one of the most exciting fields in
probability theory, and has been applied to problems in physics \cite{phy},
high-dimensional statistics \cite{highds}, wireless communications
\cite{wire}, finance \cite{fin}, etc. The Tracy-Widom distributions, or, more
generally, the distributions of the $k$-th largest level at the soft edge
scaling limit of Gaussian ensembles, are some of the most important
distributions in RMT, and their numerical evaluation
is a subject of great practical importance (see \cite{rmtintro,rao}
for friendly introductions to RMT, and see \cite{bornrmt} for an overview of the
numerical aspects of RMT).
There are generally two ways of calculating the distributions to high
accuracy numerically: one, using the Painlev\'e representation of the
distribution to reduce the calculation to solving a nonlinear ordinary
differential equation (ODE) numerically \cite{dieng}, and the other, using
the determinantal representation of the distribution to reduce the
calculation to an eigenproblem involving an integral operator
\cite{bornrmt}. 

In the celebrated work \cite{tw}, the Tracy-Widom distribution for the
Gaussian unitary ensemble (GUE) was shown to be representable as an integral
of a solution to a certain nonlinear ODE called the Painlev\'e II equation.
This nonlinear ODE can be solved to relative accuracy numerically, but
achieving relative accuracy is extremely expensive, since it generally
requires multi-precision arithmetic \cite{table}. In addition, the extension
of the ODE approach to the computation of the $k$-th largest level at the
soft edge scaling limit of Gaussian ensembles is not straightforward, as it
requires deep analytic knowledge for deriving connection formulas
\cite{bornrmt,dieng}. 

On the other hand, the method based on the Fredholm determinantal
representation uses the fact that the cumulative distribution function (CDF)
of the $k$-th largest level at the soft edge scaling limit of the Gaussian
unitary ensemble can be written in the following form:
    \begin{align}
      F_2(k;s)= \sum_{j=0}^{k-1} \frac{(-1)^j}{j!} \pp{^j}{z^j} \det\big(I - z
      \cK|_{L^2[s,\infty)}\big) \Bigr|_{z=1},
    \end{align}
where $\cK|_{L^2[s,\infty)}$ denotes the integral operator on $L^2[s,\infty)$
with kernel
\begin{align}
    K_{Ai}(x,y)=\int_s^\infty \ai(x+z-s)\ai(z+y-s) \d z,
      \label{f2kai}
\end{align}
where $\ai(x)$ is the Airy function of the first kind (see
\cite{tw,forrester} for the derivations). We also note that there exist
similar Fredholm determinantal representations for the cases of the Gaussian
orthogonal ensemble (GOE) and Gaussian sympletic ensemble (GSE) (see Section
\ref{sec:klevel}). The cumulative distribution
function and the
probability density function (PDF) of the distribution can be computed using
the eigendecomposition of the so-called Airy integral operator $\T_s$, where
$\T_s[f](x)=\int_0^\infty \ai(x+y+s)f(y) \d y$ for $x\geq 0$. This is because 
$\cK|_{L^2[s,\infty)} = \G_s^2$, where $\G_s[f](x)=\int_s^\infty
\ai(x+y-s)f(y)  \d y$ for $x\geq s$, and $\T_s$ shares the same eigenvalues and
eigenfunctions (up to a translation) with $\G_s$. If the
eigenvalues of the integral operator $\T_s$ are computed directly, they can
be known only to absolute precision, since $\T_s$ is a compact integral
operator. Furthermore, the number of degrees of freedom required to
discretize $\T_s$ increases when the kernel is oscillatory (as $s\to
-\infty$). 

In this paper, we present a new method for computing the
eigendecomposition of the Airy integral operator $\T_s$, which
solves an open problem in random matrix theory (see, for example, Open
Problem~6 in~\cite{deift}). It exploits the remarkable fact that the Airy
integral operator admits a commuting differential operator, which shares the
same eigenfunctions (see, for example, \cite{tw,karoui}). In our method, we
compute the spectrum and the eigenfunctions of the differential operator by
computing the eigenvalues and eigenvectors of a banded eigenproblem. Since
the eigenproblem is banded, the eigendecomposition can be done very quickly
in $\O(n^2)$ operations, and the eigenvalues and eigenvectors can be
computed to entry-wise full relative precision.  Finally, we use the
computed eigenfunctions to recover the spectrum of the Airy integral operator
$\T_s$, also to full relative precision.

As a direct application, our method computes the distributions of the $k$-th
largest level at the soft edge scaling limit of Gaussian ensembles to full
relative precision rapidly everywhere, except in the left tail (the left tail
is computed to absolute precision). We note that several other integral
operators admitting commuting differential operators have been studied
numerically from the same point of view as this paper (see, for
example, \cite{prolb,royt4})

Integral operators like $\T_s$, which admit commuting differential
operators, are known as bispectral operators (see, for example,
\cite{bispectral}). One famous example of a bispectral operator is the
truncated Fourier transform, which was investigated by Slepian and his
collaborators in the 60's \cite{slepian}; its eigenfunctions are known as
prolate spheroidal wavefunctions. We note that, unlike prolates, the
eigenfunctions of the operator $\T_s$ are relatively unexamined: ``The
behavior of the eigenfunctions, a problem of great practical interest,
presents a serious numerical challenge'' (see Open Problem 6
in~\cite{deift}); ``In the case of the Airy kernel, the differential
equation did not receive much attention and its solutions are not known''
(see Section 24.2 in~\cite{mehta}). In this paper, we also characterize
these previously unstudied eigenfunctions, and describe their extremal
properties in relation to an uncertainty principle involving the Airy
transform.

Finally, we note that the Airy integral operator $\T_s$ is rather universal.
For example, in Section \ref{optics}, we describe an application to optics.
In that section, we use the eigenfunctions of the Airy integral operator to
compute finite-energy Airy beams that are optimal, in the sense that they
maximally concentrate energy near the main lobes in their initial
profiles, while also remaining diffraction-free over the longest possible
distances.

\section{Mathematical and Numerical Preliminaries}
In this section, we introduce the necessary mathematical and numerical
preliminaries.

\subsection{Airy function of the first kind}
The Airy function of the first kind is the solution to the differential equation
    \begin{align}
\ddd{f}{x}-xf=0,\label{aiode}
    \end{align}
for all $x\in \R$, that decays for large $x$. It can also be written in an
integral representation
    \begin{align}
\ai(x)=\frac{1}{\pi}\int_0^\infty \cos\Bigl(\frac{t^3}{3}+xt\Bigr)\d t.
    \end{align}
\begin{remark}
One can extend the definition of $\ai(x)$ to the complex plane and show that
it is an entire function.
\end{remark}
\begin{remark}\label{asymp}
As $x\to+\infty$, 
    \begin{align}
\ai(x)\sim \frac{e^{-\frac{2}{3}x^{3/2}}}{2\pi^{1/2}x^{1/4}}.
    \end{align}
\end{remark}

\subsection{The Airy Integral Operator}
In this section, we give the definition and properties of the Airy integral
operator.

\subsubsection{The Airy integral operator $\T_c$ and its associated integral
operator $\G_c$}\label{aiintch}

In this subsection, we define the Airy integral operator, including its
eigenvalues and eigenfunctions. Its associated integral operator is
introduced as well.
\begin{definition}
  \label{gctcdef}
Given a real number $c$, let $\T_c\colon L^2[0,\infty)\to L^2[0,\infty)$
denote the Airy integral operator defined by
    \begin{align}
\T_c[f](x)=\int_0^\infty \ai(x+y+c)f(y)  \d y,\quad x\geq 0.\label{Tc}
    \end{align}
Let $\G_c\colon L^2[c,\infty)\to L^2[c,\infty)$ denote
the associated Airy integral operator defined by
    \begin{align}
\G_c[f](x)=\int_c^\infty \ai(x+y-c)f(y)  \d y,\quad x\geq c.
      \label{Gc}
    \end{align}
\end{definition}
Obviously, $\G_c$ and $\T_c$ are both compact and self-adjoint.

\medskip

We denote eigenvalues of $\T_c$ by
$\lambda_{0,c},\lambda_{1,c},\dots,\lambda_{n,c},\dots$, ordered so that
$\abs{\lambda_{j-1,c}}\geq \abs{\lambda_{j,c}}$ for all $j\in \N^+$. For each
non-negative integer $j$, let $\psi_{j,c}$ denote the $(j+1)$-th
eigenfunction of $\T_c$, so that
    \begin{align}
\hspace{-1em}\lambda_{j,c}\psi_{j,c}(x) =& \int_0^\infty
\ai(x+y+c)\psi_{j,c}(y)  \d y,\quad x\in[0,\infty).\label{Teigfun}
    \end{align}
In this paper, we normalize the eigenfunctions such that
$\norm{\psi_{j,c}}_2=1$ for any real number $c$
and non-negative integer $j$. Since the eigenfunctions are real, this
condition only specifies the eigenfunctions up to multiplication by $-1$. We
thus require that $\psi_{j,c}(0) > 0$ (we show in Theorem~\ref{thm:rec} in
Appendix~\ref{sec:misc} that $\psi_{j,c}(0)\ne 0$).

\medskip

Note that $\T_c$ and $\G_c$ share the same eigenvalues and eigenfunctions up
to a translation, i.e.  $\lambda_{j,c}, \psi_{j,c}(x)$ is an
eigenpair of the operator $\T_c$, and $\lambda_{j,c}, \psi_{j,c}(x-c)$ is an
eigenpair of the operator $\G_c$ (see Theorem \ref{thm:similarity}). Note
that the operator $\T_c$ is more convenient to work with than the operator
$\G_c$, as its domain is invariant under change of $c$. Therefore, we
will mainly focus on the study of the Airy integral operator $\T_c$ in this
paper.

\begin{remark}
For simplicity, we will use $\lambda_j$ and $\psi_j$ to denote the eigenvalue
and the eigenfunction when there is no ambiguity.
\end{remark}

\subsubsection{Properties and connection to the Airy transform}

\begin{definition}
Let $\cA\colon L^2(\R) \to L^2(\R)$ denote the integral transform defined by 
the formula
  \begin{align}
\cA[\phi](x) = \int_{-\infty}^\infty \ai(x+y)\phi(y) \d y.
  \end{align}
In a mild abuse of terminology, we call $\cA$ the Airy transform. Note that 
the standard Airy transform of $\phi$ is defined as $\int_{-\infty}^\infty
\ai(x-y)\phi(y) \d y$, which can be written as $\cA\circ R$, where $R$
denotes the reflection operator. 
\end{definition}
It is well-known that $\cA$ is unitary, and that $\cA^2 = I$, where $I$ is
the identity operator (see, for example,~\cite{vallee}). To introduce the connection
between the Airy transform $\cA$, the so-called Airy kernel integral
operator $\cK$ (see formula (\ref{f2kai})), and the two integral operators
$\T_c, \G_c$ defined in Section \ref{aiintch}, we first define the following
operators.

\begin{definition}\label{def:F}
Given real numbers $a$ and $b$, let $F_{a,b} \colon
L^2(\R) \to L^2(\R)$ be the operator defined by the formula
  \begin{align}
F_{a,b}[\phi](x) = \1_{[b,\infty)}(x) \cA[ \1_{[a,\infty)}(y) \phi(y)](x),\label{for:F}
  \end{align}
where $\1_X$ denotes the indicator function associated with the set $X$. 
Let $\tilde F_c$ be a synonym for the operator $F_{0,c}$. 
\end{definition}

The operator $F_{a,b}$
represents a truncation or ``band-limiting'' to the half line $[a,\infty)$, 
followed by an Airy transform, followed by another truncation to the half
line $[b,\infty)$.  Clearly, $F_{a,b}^* = F_{b,a}$.

\begin{definition}\label{def:P}
Given a real number $c$, let $P_c \colon
L^2(\R) \to L^2(\R)$ denote the projection operator defined by the formula
  \begin{align}
P_c[\phi](x) = \1_{[c,\infty)}(x) \phi(x).
  \end{align}
\end{definition}
It's easy to see that $\tilde F_c = P_c \cA P_0$, and that $\tilde
F_{-\infty} = \cA P_0$. 
\begin{definition}\label{def:T}
Given a real number $c$, let $T_c\colon L^2(\R) \to L^2(\R)$ denote the
translation operator defined by the formula
  \begin{align}
T_c[\phi](x) = \phi(x-c).
  \end{align}
\end{definition}
Below, we define the integral operator $\cK$.
\begin{definition}\label{def:f2kai}
Let $\cK\colon L^2(\R) \to L^2(\R)$ denote the integral operator with kernel
\begin{align}
    K_{\ai}(x,y) = \int_0^\infty \ai(x+z)\ai(y+z) \d z.
    \label{twkern}
\end{align}
\end{definition}
Clearly, $\cK = \tilde F_{-\infty} \tilde F_{-\infty}^*$.
Moreover, by a change of variables, the kernel~(\ref{twkern}) can be rewritten as
  \begin{align}
K_{\ai}(x,y) = \int_c^\infty \ai(x+z-c)\ai(y+z-c) \d z,
  \end{align}
from which we see that $\cK|_{L^2[c,\infty)}$ (equivalently, $P_c \cK P_c$
or $\tilde F_c {\tilde F_c}^*$) is equal to the square of the associated
Airy integral operator $\G_c$ defined in Definition~\ref{gctcdef}. Thus, the
eigenfunctions and eigenvalues of $\cK|_{L^2[c,\infty)}$ are given by
$\psi_{j,c}(x-c)$ and $\lambda_{j,c}^2$ (see formula (\ref{Teigfun})),
respectively.

The following theorem, proved in Lemma~2 of~\cite{tw}, states that the eigenvalues
$\lambda_{j,c}$ approach one in absolute value as $c\to-\infty$.
\begin{theorem}\label{thm:eig1}
For each $j$, $\lambda_{j,c}^2 \to 1$ as $c\to -\infty$, where $\lambda_{j,c}$ is 
the $(j+1)$-th eigenvalue of the Airy integral operator $\T_c$.
\end{theorem}
\begin{proof}
We first show that $\cK$ is a projection operator. Since $\tilde
F_{-\infty} = \cA P_0 $, we have that $\cK = \tilde F_{-\infty} \tilde
F_{-\infty}^* = \cA P_0 \cA$, so $\cK^2 = \cA P_0 \cA^2 P_0 \cA$.  Recalling
that $\cA^2 = I$, it follows that $\cK^2 = \cA P_0 \cA = \cK$.  Since $\cK$
is a projection, its spectrum takes values in the set $\{0,1\}$, and since
$\cK$ has an infinite dimensional range, it has infinitely many eigenvalues
equal to~$1$.  The operator $P_c \cK P_c$ converges to $\cK$ as $c\to
-\infty$, so it follows that, for each $j$, $\lambda_{j,c}^2 \to 1$ as $c\to
-\infty$.
\end{proof}

In the next theorem, we show that the Airy integral operator $\T_c$ is related to its
associated integral operator $\G_c$ by a similarity transformation.
\begin{theorem}\label{thm:similarity}
The Airy integral operator $\T_c$ is similar to its associated integral
operator $\G_c$. Furthermore, if $\lambda_{j,c}$ and $\psi_{j,c}$
are eigenvalues and eigenfunctions of $\T_c$, then $\lambda_{j,c}$ and
$\psi_{j,c}(x-c)$ are eigenvalues and eigenfunctions of $\G_c$.
\end{theorem}
\begin{proof}
We observe that $\T_c = T_{-c}
F_{0,c}$,  so $\T_c = T_{-c} (T_c F_{c,0} T_c) = T_{-c} \G_c T_c$. Since
$T_c^* = T_{-c}$ and $T_c T_c^* = I$, we see that $\T_c$ is related to
$\G_c$ by a similarity transformation. The statement about the
eigenfunctions and eigenvalues follows immediately.
\end{proof}

Finally, we characterize the relation between the eigenfunction
$\psi_{j,c}$ and its Airy transform $\cA[\psi_{j,c}]$.

\begin{theorem}\label{thm:anacont}
For any real $c$, there exists an analytic continuation of the
eigenfunction $\psi_{j,c}$ of the Airy integral operator with parameter $c$,
which we denote by $\tilde\psi_{j,c}$. Furthermore,
  \begin{align}
\tilde\psi_{j,c}(x)=\frac{1}{\lambda_{j,c}}\cA[\psi_{j,c}](x+c),
  \end{align}
for all $x\in \R$, where $\lambda_{j,c}$ is the corresponding eigenvalue of
$\psi_{j,c}$.
\end{theorem}

\begin{proof}
The existence of the analytic continuation $\psi_{j,c}$ is given by formula
(\ref{Teigfun}) and the fact that the Airy function is analytic and decays
superexponentially. Note that
  \begin{align}
\T_{c} = T_{-c}
F_{0,c} = T_{-c} P_c \cA = P_0 T_{-c} \cA,\label{for:TcPTA_new}
  \end{align}
so after applying both sides of (\ref{for:TcPTA_new}) to $\psi_{j,c}$, we get
  \begin{align}
\lambda_{j,c}\psi_{j,c}(x)=P_0\cA[\psi_{j,c}](x+c),
  \end{align}
from which it follows that
  \begin{align}
\tilde\psi_{j,c}(x)=\frac{1}{\lambda_{j,c}}\cA[\psi_{j,c}](x+c),
  \end{align}
for all $x\in \R$.
\end{proof}

\subsubsection{Commuting differential operator}\label{commutechap}
\begin{definition}
Given a real number $c$, let $\L_c\colon L^2[0,\infty)\to L^2[0,\infty)$ denote
the Sturm-Liouville operator defined by
    \begin{align}
\L_c[f](x)=-\dd{}{x} \Bigl(x \dd{}{x}f \Bigr) + x(x+c)f.\label{Lc}
    \end{align}
\end{definition}
Obviously, $\L_c$ is self-adjoint (more specifically, it's a singular
Sturm-Louville operator with singular points $x=0$ and $x=\infty$). It has
been shown in \cite{tw} that $\L_c$ commutes with the Airy integral operator
$\T_c$, and their eigenvalues have multiplicity one. Thus, $\L_c$ and $\T_c$
share the same set of eigenfunctions. The following theorem formalizes this
statement (see \cite{tw,karoui}).
\begin{theorem}
For any real number $c$, there exists a strictly increasing sequence of positive real
numbers $\chi_{0,c},\chi_{1,c},\dots$ such that, for each $m\geq 0$, the
differential equation
    \begin{align}
    \label{eigode}
\dd{}{x} \Bigl(x \dd{}{x}\psi_{m,c} \Bigr) - (x^2+cx-\chi_{m,c})\psi_{m,c}=0
    \end{align}
has a unique solution $\psi_{m,c}$ that is continuous on the half-closed
interval $[0,\infty)$.  For each $m\geq 0$, the function $\psi_{m,c}$ is
exactly the $(m+1)$-th eigenfunction of the integral operator $\T_{c}$.
\end{theorem}

\begin{remark}
The equation (\ref{eigode}) can also be written as 
    \begin{align}
\L_c[\psi_{m,c}]=\chi_{m,c}\psi_{m,c}.\label{eigode2}
    \end{align}
\end{remark}

\begin{remark}
The numerical evaluation of high-order eigenfunctions via the discretization
of the Airy integral operator $\T_c$ is highly inaccurate due to its exponentially
decaying eigenvalues. However, the Sturm-Liouville operator $\L_c$
has a growing and well-separated spectrum, which is numerically much more
tractable. Therefore, $\L_c$ is the principal analytical tool for
computing the eigenvalues and eigenfunctions of $\T_c$ to relative accuracy.
\end{remark}

\subsection{Laguerre polynomials}
\label{laguerrech}
The Laguerre polynomials, denoted by $L_n\colon[0,\infty)\to \R$, are defined
by the following three-term recurrence relation for any $k\geq 1$ (see
\cite{stegun}):
    \begin{align}
L_{k+1}(x)=\frac{(2k+1-x)L_k(x)-kL_{k-1}(x)}{k+1},\label{lrec}
    \end{align}
with the initial conditions
    \begin{align}
L_0(x)=1,\quad L_1(x)=1-x.\label{lini}
    \end{align}
The polynomials defined by the formulas (\ref{lrec}) and (\ref{lini}) are an
orthonormal basis in the Hilbert space induced by the inner product
$\inner{f}{g}=\int_0^\infty e^{-x} f(x)g(x) \d x$, i.e.,
    \begin{align}
\inner{L_n}{L_m}=\int_0^\infty e^{-x}L_n(x)L_m(x) \d x=\delta_{n,m}.\label{lort}
    \end{align}
In addition, the Laguerre polynomials are solutions of Laguerre's equation
    \begin{align}
xf''+(1-x)f'+nf=0.    
    \end{align}

We find it useful to use the scaled Laguerre functions defined below.
\begin{definition}
Given a positive real number $a$, the scaled Laguerre functions, denoted by
$h_n^a\colon[0,\infty)\to \R$, are defined by
    \begin{align}
h_n^a(x)=\sqrt{a}e^{-ax/2}L_n(ax).\label{slag}
    \end{align}
\end{definition}

\begin{remark}
The scaled Laguerre functions $h_n^a(x)$ are an orthonormal basis in
$L^2[0,\infty)$, i.e.,
    \begin{align}
\int_0^\infty h_n^a(x)h_m^a(x) \d x=\delta_{n,m}.\label{hort}
    \end{align}
\end{remark}

The following two theorems directly follow from the results for Laguerre
polynomials in, for example, \cite{stegun}.
\begin{theorem}
Given a positive real number $a$ and a non-negative integer $n$,
    \begin{align}
\hspace{-2em}xh_n^a(x)=&\
\frac{1}{a}\bigl(-nh_{n-1}^a(x)+(2n+1)h_n^a(x)-(n+1)h_{n+1}^a(x)\bigr),\label{slagx}\\
\hspace{-2em}x^2h_n^a(x)=&\
\frac{1}{a^2}\bigl(n(n-1)h_{n-2}^a(x)-4n^2h_{n-1}^a(x)+(6n^2+6n+2)h_{n}^a(x)\notag\\
&-4(1+n)^2h_{n+1}^a(x)+(n+1)(n+2)h_{n+2}^a(x)\bigr).\label{slagxx}
    \end{align}
\end{theorem}

\begin{theorem}
Given a positive real number $a$ and a non-negative integer $n$,
    \begin{align}
\frac{d}{dx}h_n^{a}=& -\frac{a}{2}h_n^a-a\sum_{k=0}^{n-1}h_k^a,\label{dslag}\\
\frac{d^2}{dx^2}h_n^{a}=&\
\frac{a^2}{4}h_n^a+a^2\sum_{k=0}^{n-1}(n-k)h_k^a.\label{d2slag}
    \end{align}
\end{theorem}

The following corollary is a direct result of (\ref{dslag}).
\begin{corollary}\label{spcdiff}
Given a positive real number $a$ and a non-negative integer $n$,
    \begin{align}
\frac{d}{dx}h_n^{a}-\frac{d}{dx}h_{n-1}^{a}=-\frac{a}{2}h_n^a-\frac{a}{2}h_{n-1}^a.
    \end{align}
\end{corollary}

\begin{observation}
The scaled Laguerre functions $h_n^a(x)$ are solutions of the following
ODE on the interval $[0,\infty)$:
    \begin{align}
\frac{d}{dx}\Bigl(x\frac{d}{dx}h_n^a\Bigr)-\frac{a}{4}(ax-4n-2)h_n^a=0.\label{slagode}
    \end{align}

\end{observation}

The following theorem, proven (in a slightly different form) in
\cite{xiang}, describes the decaying property of the expansion coefficients
in the Laguerre polynomial basis.
\begin{theorem}\label{decayx}
Suppose $f\in C^{k}[0,\infty)$ where $k\geq 1$, and $f$ satisfies
\begin{align}
    \lim_{x\to \infty} e^{-x/2}x^{j+1}f^{(j)}(x)=0, \\
    V=\sqrt{\int_0^\infty x^{k+1}e^{-x}(f^{(k+1)}(x))^2\d x}<\infty,
\end{align}
for $j=0,1,\dots,k$. Suppose further that $a_n=\int_0^\infty
e^{-x}f(x)L_n(x)\d x$. Then, for $n>k$,
\begin{align}
    |a_n|\leq \frac{V}{\sqrt{n(n-1)\dots(n-k)}} =
    \O\Bigl(\frac{1}{n^{(k+1)/2}}\Bigr),
\end{align}
and
\begin{align}
    \norm{f(x)-\sum_{n=0}^N a_n L_n(x)} \to 0,
\end{align}
as $N\to\infty$, where $\norm{\cdot}$ represents the $L^2[0,\infty)$ norm with the weight
function $e^{-x}$.
\end{theorem}

The following corollary extends the theorem above to the case where the
Laguerre polynomials are replaced by scaled Laguerre functions. 
\begin{corollary}\label{decaycorr}
Suppose that $a\in \R$ and $a>0$. Suppose further that $g\in
C^{k}[0,\infty)$ for some $k\geq 1$, and define
$f(x)=\frac{1}{\sqrt{a}}e^{x/2} g(x/a)$. Assume finally that $f$ satisfies
\begin{align}
    \lim_{x\to \infty} e^{-x/2}x^{j+1}f^{(j)}(x)=&\,0, \label{decayreq1}\\
    V=\sqrt{\int_0^\infty x^{k+1}e^{-x}(f^{(k+1)}(x))^2\d
    x}<&\,\infty,\label{decayreq2}
\end{align}
for $j=0,1,\dots,k$, and let $b_n=\int_0^\infty g(x)h_n^a(x) \d x$. Then,
for $n>k$,
\begin{align}
    |b_n|\leq
    \frac{V}{\sqrt{n(n-1)\dots(n-k)}}=\O\Bigl(\frac{1}{n^{(k+1)/2}}\Bigr),\label{decay2}
\end{align}
and
\begin{align}
    \norm{g(x)-\sum_{n=0}^N b_n h_n^a(x)} \to 0,\label{decay3}
\end{align}
as $N\to\infty$, where $\norm{\cdot}$ represents the $L^2[0,\infty)$ norm with the weight
function $1$.
\end{corollary}
\begin{proof}
By definition,
\begin{align}
    |b_n|=&\,\Big|\int_0^\infty h_n^a(x) g(x) \d x\Big| \notag\\
    =&\, \Big|\int_0^\infty
    \sqrt{a}e^{-ax/2}L_n(ax) g(x) \d x\Big|\notag\\
    =&\,\Big|\int_0^\infty \frac{1}{\sqrt{a}}e^{-x/2}L_n(x) g(x/a) \d
    x\Big|\notag\\
    =&\,\Big|\int_0^\infty e^{-x} L_n(x)f(x)\d x\Big|\notag\\
    =&\,|a_n|\notag\\
    \leq&\, \frac{V}{\sqrt{n(n-1)\dots(n-k)}},\label{anbn}
\end{align}
where $a_n$ is defined in the same way as in Theorem \ref{decayx}. Thus,
(\ref{decay2}) is proved.

To prove (\ref{decay3}), note that 
\begin{align}
    \norm{g(x)-\sum_{n=0}^N b_n h_n^a(x)}\,^2 = \,&\int_{0}^\infty
    \Big(g(x)-\sum_{n=0}^N b_n \sqrt{a}e^{-ax/2}L_{n}(ax)\Big)^2 \d
    x\notag\\
    =&\, \int_{0}^\infty \Big(g\big(\frac{y}{a}\big)-\sum_{n=0}^N
    b_n\sqrt{a}e^{-y/2}L_{n}(y)\Big)^2 \frac{1}{a} \d y\notag\\
    =&\,\int_0^\infty
    \Big(\frac{1}{\sqrt{a}}g\big(\frac{y}{a}\big)-\sum_{n=0}^N
    b_ne^{-y/2}L_{n}(y)\Big)^2 \d y\notag\\
    =&\, \int_0^\infty e^{-y}\Big(f(y)-\sum_{n=0}^N b_nL_n(y)\Big)^2 \d y\notag\\
    \leq &\,\int_0^\infty \Big(f(y)-\sum_{n=0}^N b_nL_n(y)\Big)^2 \d y \notag\\
    =&\,\norm{f(x)-\sum_{j=0}^n a_n L_n(x)}\,^2 \to 0,
\end{align}
as $N\to\infty$, where the last equality holds by combining Theorem \ref{decayx} and the fact
that $b_n=a_n$ (see formula (\ref{anbn})).
\end{proof}

\subsection{Numerical tools for five-diagonal matrices}
\subsubsection{Eigensolver}
A five-diagonal matrix can be reduced to a tridiagonal
matrix using the algorithm in \cite{five}. Once it is in tridiagonal form, a
standard Q-R (or Q-L) algorithm can then be used to solve for all of its
eigenvalues to absolute precision. 
\begin{remark}
The time complexity of the reduction and Q-R algorithm are both $\O(n^2)$
for a five-diagonal matrix of size $n\times n$.
\end{remark}

\subsubsection{Shifted inverse power method}\label{ipm}
Suppose that $A$ is an $N\times N$ real matrix, for some positive integer
$N$, and suppose that its eigenvalues are distinct.
Let $\sigma_1<\sigma_2<\dots<\sigma_N$ denote the eigenvalues of $A$.
The shifted inverse power method iteratively finds the eigenvalue $\sigma_k$
and the corresponding eigenvector $v_k\in\R^N$, provided an approximation
$\lambda$ to $\sigma_k$ is given, and that
    \begin{align}
\abs{\lambda-\sigma_k}<\max \{\abs{\lambda-\sigma_j}\colon j\neq k\}.
    \end{align}
Each shifted inverse power iteration solves the linear system
    \begin{align}
(A-\lambda_j I)x=w_{j},
    \end{align}
where $\lambda_j$ and $w_j\in\R^n$ are the approximations to $\sigma_k$ and
$v_k$, respectively, after $j$ iterations; the number $\lambda_j$ is usually
referred to as the ``shift''.
The approximations $\lambda_{j+1}$ and $w_{j+1}\in\R^N$ are evaluated via the formulas
    \begin{align}
w_{j+1}=\frac{x}{\norm{x}},\quad \lambda_{j+1}=w_{j+1}^T A w_{j+1}
    \end{align}
(see, for example, \cite{prolb,tref} for more details). 

In this paper, we note that we use the phrase ``inverse power method'' to refer to
the unshifted inverse power method.
\begin{remark}\label{ipmr}
The shifted inverse power method converges cubically in the vicinity of the
solution, and each iteration requires $\O(n)$ operations for a tridiagonal
or a five-diagonal matrix (see \cite{prolb,tref}). 
\end{remark}

\section{Analytical Apparatus}
In this section, we first introduce several analytical results which we will use
to develop the numerical algorithm of this paper. We then characterize
the Airy integral operator's previously unstudied eigenfunctions, and
describe their extremal properties in relation to an uncertainty principle
involving the Airy transform (see Sections \ref{sec:uncert}, \ref{sec:extremal},
\ref{sec:qual}).

Recall that we denote the eigenfunctions of the eigenfunctions of the operators $\T_c$
and $\L_c$ by $\psi_{n,c}$ (see Sections \ref{aiintch}, \ref{commutechap}),
and represent them in the basis of scaled Laguerre functions $h_k^a$ (see
Section \ref{laguerrech}). We denote the eigenvalues of the Airy
integral operator $\T_c$ by $\lambda_{n,c}$.
\subsection{The commuting differential operator in the basis of scaled Laguerre functions}

\begin{theorem}\label{thm:Lh}
For any positive real number $a$, real number $c$, and non-negative integer $k$,
  \begin{align}
&\L_c[h_k^a](x)=\frac{1}{4a^2}\Bigl(4k(k-1)h_{k-2}^a(x)\notag \\
&+ k(a^3-4ac-16k)h_{k-1}^a(x)\notag\\
&+ (8+a^3+4ac+24k+2a^3k+8ack+24k^2)h_{k}^a(x) \notag\\
&+ (k+1)\bigl(a^3-4ac-16(k+1)\bigr)h_{k+1}^a(x)\notag\\
&+ 4(k+1)(k+2)h_{k+2}^a(x) \Bigr),\label{Lh}
  \end{align}
for $x\in [0,\infty)$.
\end{theorem}

\begin{proof}
By definition,
  \begin{align}
\L_c[h_k^a](x) = -\dd{}{x} \Bigl(x \dd{}{x}h_k^a(x) \Bigr) +
x(x+c)h_k^a(x).\label{slageig1}
  \end{align}
By applying (\ref{slagode}), terms involving derivatives of $h_n^a(x)$ on
the right side of (\ref{slageig1}) disappear.
Finally, we reduce the remaining $xh_k^a(x), x^2h_k^a(x)$ terms to
$h_k^a(x)$ via (\ref{slagx}), (\ref{slagxx}).
\end{proof}

\begin{remark}
Although $h_{k-2}^a(x), h_{k-1}^a(x)$ may be undefined when $k=0,1$, the
theorem still holds, since the coefficients of $h_{k-2}^a(x),
h_{k-1}^a(x)$ in (\ref{Lh}) will be zero in that case.
\end{remark}

\subsection{Decay of the expansion coefficients of the eigenfunctions}
\begin{theorem}\label{decaypf}
Suppose that $a,c\in \R$ and $a>0$. Suppose further that
$\beta_k^{(m)}=\int_0^\infty \psi_{m,c}(x)h_k^a(x)\d x$ for $k=0,1,\dots $ .
Then, $|\beta_k^{(m)}|$ decays super-algebraically as $k$ goes to infinity.
\end{theorem}
\begin{proof}
Using the integral representation of $\psi_{m,c}$,
    \begin{align}
|\beta_{k}^{(m)}|=&\, \frac{1}{|\lambda_m|}\bigg|\int_0^\infty
\psi_{m,c}(y)\Big(\int_0^\infty \ai(y+x+c)h_k^a(x)\d x\Big)\d
y\bigg|\notag\\
\leq&\,
\frac{1}{|\lambda_m|}\norm{\psi_{m,c}(y)}_{L^2[0,\infty)}\bbnorm{\int_0^\infty
\ai(y+x+c)h_k^a(x)\d x}_{L^2[0,\infty)}\notag\\
=&\, \frac{1}{|\lambda_m|}\bbnorm{\int_0^\infty \ai(y+x+c)h_k^a(x)\d
x}_{L^2[0,\infty)},
    \end{align}
by the Cauchy-Schwartz inequality and the fact that
$\norm{\psi_{m,c}(y)}_{L^2[0,\infty)}=1$.

Define $g(x)=\ai(y+x+c)$ for some constants $y\geq 0,c\in \R$. By Remark
\ref{asymp}, for any real number $a>0$, it's clear that
$f(x)=\frac{1}{\sqrt{a}}e^{x/2} g(x/a)$ satisfies the conditions
(\ref{decayreq1}), (\ref{decayreq2}) in Corollary \ref{decaycorr}. As $g$ is
analytic, we have that $g\in C^{p}[0,\infty)$ for any non-negative integer
$p$. Therefore, by Corollary \ref{decaycorr}, $|\beta_k^{(m)}|$ decays
super-algebraically as $k$ goes to infinity.
\end{proof}

\subsection{Recurrence relation involving the Airy integral operator acting
on scaled Laguerre functions of different orders}
\begin{theorem}
Given a positive real number $a$, a real number $s$, and a non-negative
integer $n$, define 
    \begin{align}
H_n^a:=\int_0^\infty \ai\left(y+s\right) h_n^a(y) \d y=\sqrt{a}\int_0^\infty
\ai\left(y+s\right)e^{-\frac{ay}{2}}L_n(ay) \d y.\label{Hna}
    \end{align}
Then 
    \begin{align}
\hspace*{-3em}(n-1)H_{n-2}^a&-(4n-1+as-\frac{1}{4}a^3)H_{n-1}^a+(6n+3+2as+\frac{1}{2}a^3)H_n^a\notag\\
\hspace*{-3em}&-(4n+5+as-\frac{1}{4}a^3)H_{n+1}^a+(n+2)H_{n+2}^a=0,\label{fiveterm}
    \end{align}
for $n=1,2,\dots$ . We note that $H_n^a$ depends on the variable $s$, but we
omit this dependency on $s$ in our notation where the meaning is
clear.\label{recur}
\end{theorem}

\begin{proof}
By combining the recurrence relation for Laguerre polynomials (see
(\ref{lrec})) and the definition of the Airy function (see (\ref{aiode})),
we have
    \begin{align}
\hspace{-3em}H_{n+1}^a=&\int_0^\infty\ai(y+s)h_{n+1}^a(y)\d y\notag\\        
\hspace{-3em}=&\int_0^\infty\ai(y+s)\frac{(2n+1-ay)h_n^a(y)-nh_{n-1}^a(y)}{n+1}\d
y\notag\\
\hspace{-3em}=&\
\frac{2n+1}{n+1}H_n^a-\frac{n}{n+1}H_{n-1}^a-\frac{a}{n+1}\int_0^\infty y
\ai(y+s)h_n^a(y)d y\notag\\
\hspace{-3em}=&\
\frac{2n+1+as}{n+1}H_n^a-\frac{n}{n+1}H_{n-1}^a-\frac{a}{n+1}\int_0^\infty
\ai''(y+s)h_n^a(y)\d y,\label{Tna}
    \end{align}
for any non-negative integer $n$. By applying integration by parts twice to
the last term in (\ref{Tna}), we get
    \begin{align}
\hspace{-0em}\int_0^\infty \ai''(y+s)h_n^a(y)\d
y=&-\sqrt{a}\ai'(s)-a\sqrt{a}(\frac{1}{2}+n)\ai(s)\notag\\
&+\int_0^\infty \ai(y+s)({h_n^{a}}(y))''\d y.\label{aiprimeh}
    \end{align}
By (\ref{d2slag}), the last term in (\ref{aiprimeh}) becomes
    \begin{align}
\hspace{-4em}\int_0^\infty \ai(y+s)({h_n^{a}}(y))''\d y=&\ a^2\int_0^\infty
\ai(y+s) \Bigl(\frac{1}{4}h_n^a(y)+\sum_{k=0}^{n-1}(n-k)h_k^a(y)\Bigr)\d
y\notag\\
\hspace{-4em}=&\ a^2\bigl(\frac{1}{4}H_n^a +
\sum_{k=0}^{n-1}(n-k)H_k^a\bigr).\label{aihprime}
    \end{align}
Thus, by multiplying both sides of (\ref{Tna}) by $n+1$, and combining
(\ref{aiprimeh}), (\ref{aihprime}), we have
    \begin{align}
\hspace{-3em}nH_{n-1}^a-(2n+1+as-\frac{1}{4}a^3)H_n^a+(n+1)H_{n+1}^a+a^3\sum_{k=0}^{n-1}(n-k)H_k^a\notag
\\
\hspace{-3em}=\ a\sqrt{a}\Bigl(\ai'(s)+a(\frac{1}{2}+n)\ai(s)\Bigr),\label{aihid1}
    \end{align}
for $n=0,1,2,\dots$ .

We can write (\ref{aihid1}) equivalently as 
\begin{align}
    \hspace{-3em}(n-1)H_{n-2}^a-(2n-1+as-\frac{1}{4}a^3)H_{n-1}^a+nH_{n}^a+a^3\sum_{k=0}^{n-2}(n-1-k)H_k^a\notag
    \\
    \hspace{-3em}=\
    a\sqrt{a}\Bigl(\ai'(s)+a(-\frac{1}{2}+n)\ai(s)\Bigr),\label{aihid2}
\end{align}
for $n=1,2,3,\dots$, or
\begin{align}
    \hspace{-3em}(n+1)H_{n}^a-(2n+3+as-\frac{1}{4}a^3)&H_{n+1}^a+(n+2)H_{n+2}^a+a^3\sum_{k=0}^{n}(n+1-k)H_k^a\notag
    \\
    \hspace{-5em}&=\
    a\sqrt{a}\Bigl(\ai'(s)+a(\frac{3}{2}+n)\ai(s)\Bigr),\label{aihid3}
\end{align}
for $n=-1,0,1,\dots$ .

Finally, noticing that 
    \begin{align}
\sum_{k=0}^{n-2}(n-1-k)H_k^a-2\sum_{k=0}^{n-1}(n-k)H_k^a+\sum_{k=0}^{n}(n+1-k)H_k^a=H_n^a,
    \end{align}
equation (\ref{aihid2}), minus two times equation (\ref{aihid1}), plus
equation (\ref{aihid3}), gives the identity that we need.
\end{proof}

\subsection{Ratio between the eigenvalues of the Airy integral operator}
\begin{theorem}\label{ratiothm}
For any non-negative integers $m$ and $n$,
    \begin{align}
\frac{\lambda_m}{\lambda_n}=\frac{\int_0^\infty \psi_n'(x)\psi_m(x)\d
x}{\int_0^\infty \psi_n(x)\psi_m'(x)\d x}.
    \end{align}
\end{theorem}
\begin{proof}
The identity immediately follows from formula (\ref{es7}) in the proof of
Theorem \ref{eseries} in Appendix~\ref{sec:misc}.
\end{proof}

\subsection{Derivative of $\lambda_{n,c}$ with respect to $c$}\label{dldc}
A slightly different version of the following theorem is first proved in \cite{tw}. Here, we present a different proof.
\begin{theorem}
For all real $c$ and non-negative integers $n$, 
    \begin{align}
\pp{\lambda_{n,c}}{c}=-\frac{1}{2}\lambda_{n,c}\bigl(\psi_{n,c}(0)\bigr)^2.\label{dlc0}
    \end{align}
\end{theorem}

\begin{proof}
Given two real numbers $a,c$, define $\epsilon=\frac{c-a}{2}$. By (\ref{Teigfun}),
    \begin{align}
\hspace{-3em}\lambda_{n,c}\psi_{n,c}(x)\psi_{n,a}(x+\epsilon)=\psi_{n,a}(x+\epsilon)\int_0^\infty\ai(x+y+c)\psi_{n,c}(y)
\d y.\label{dlc1}
    \end{align}
We integrate both sides of (\ref{dlc1}) over the interval $[0,\infty)$ with
respect to $x$ to obtain
    \begin{align}
\hspace{-5em}\lambda_{n,c}\int_0^\infty \psi_{n,c}(x)\psi_{n,a}(x+\epsilon)
\d x=&\int_0^\infty\psi_{n,c}(y)\int_0^\infty \ai(x+y+c)
\psi_{n,a}(x+\epsilon)  \d x \d y\label{dlc2}\notag\\
\hspace{-5em}=&\int_0^\infty\psi_{n,c}(y)\int_\epsilon^\infty
\ai(y+\epsilon+s+a) \psi_{n,a}(s)  \d s \d y\notag\\
\hspace{-5em}=&\int_0^\infty\psi_{n,c}(y)\Bigl(\int_\epsilon^0
\ai(y+\epsilon+s+a) \psi_{n,a}(s) \d s\notag\\
&\quad\quad\quad\quad\quad\quad+\lambda_{n,a}\psi_{n,a}(y+\epsilon)\Bigr)  \d y,
    \end{align}
where the change of variable $s=x+\epsilon$ is applied in (\ref{dlc2}).
After rearranging the terms, we have
    \begin{align}
\hspace{-6.5em}(\lambda_{n,c}-\lambda_{n,a})\int_0^\infty
\psi_{n,c}(x)\psi_{n,a}(x+\epsilon) \d
x=\int_0^\infty\psi_{n,c}(y)\Bigl(\int_\epsilon^0 \ai(y+\epsilon+s+a)
\psi_{n,a}(s) \d s\Bigr)  \d y.\label{dlc5}
    \end{align}
Then, we divide both sides by $2\epsilon$ and take the limit $2\epsilon\to
0$. The left side of (\ref{dlc5}) becomes
    \begin{align}
&\lim_{2\epsilon\to
0}\frac{\lambda_{n,c}-\lambda_{n,a}}{2\epsilon}\int_0^\infty
\psi_{n,c}(x)\psi_{n,a}(x+\epsilon) \d x\notag\\
&=\pp{\lambda_{n,c}}{c}\lim_{a\to
c}\int_0^\infty \psi_{n,c}(x)\psi_{n,a}\bigl(x+\frac{c-a}{2}\bigr) \d x\notag\\
&=\pp{\lambda_{n,c}}{c}\norm{\psi_{n,c}}_2^2\notag\\
&=\pp{\lambda_{n,c}}{c}.\label{dlc8}
    \end{align}
The right side of (\ref{dlc5}) becomes
    \begin{align}
&\lim_{2\epsilon\to
0}\frac{1}{2\epsilon}\int_0^\infty\psi_{n,c}(y)\Bigl(\int_\epsilon^0
\ai(y+\epsilon+s+a) \psi_{n,a}(s) \d s\Bigr)  \d y\notag\\
&=-\frac{1}{2}\psi_{n,c}(0)\lim_{a\to
c}\int_0^\infty\ai\Bigl(y+\frac{c+a}{2}\Bigr)\psi_{n,c}(y) \d y\notag\\
&=-\frac{1}{2}\lambda_{n,c}\bigl(\psi_{n,c}(0)\bigr)^2.\label{dlc9}
    \end{align}
Finally, by combining (\ref{dlc5}), (\ref{dlc8}), and (\ref{dlc9}),
    \begin{align}
\pp{\lambda_{n,c}}{c}=-\frac{1}{2}\lambda_{n,c}\bigl(\psi_{n,c}(0)\bigr)^2.
    \end{align}
\end{proof}
The following corollaries are immediate consequences of the preceding one.
\begin{corollary}
For all real $c$ and non-negative integers $m,n$, 
    \begin{align}
\pp{}{c}\Bigl(\frac{\lambda_{m,c}}{\lambda_{n,c}}\Bigr)=\frac{\lambda_{m,c}\Bigl((\psi_{n,c}(0))^2-(\psi_{m,c}(0))^2\Bigr)}{2\lambda_{n,c}}.
    \end{align}
\end{corollary}
\begin{corollary}\label{dl2dc}
    For all real $c$ and non-negative integers $n$, 
        \begin{align}
\pp{\lambda_{n,c}^2}{c}=-\lambda_{n,c}^2\bigl(\psi_{n,c}(0)\bigr)^2
.
        \end{align}
\end{corollary}

\subsection{An uncertainty principle}
  \label{sec:uncert}

\begin{definition}
Suppose that the function $f\colon \R\to\R$ has an Airy transform
$\sigma\colon \R\to\R$ that is supported on the half-line $[a,\infty)$, so
that
  \begin{align}
f(x) = \int_{a}^\infty \ai(x+y)\sigma(y) \d y
    \label{fairytr}
  \end{align}
for all $x\in \R$. We call functions representable by integrals of the
form~(\ref{fairytr}) Airy-bandlimited.
\end{definition}

Since the Airy function $\ai(y)$ decays rapidly for
$y>0$, it is not difficult to see that the function $f$ can be extended to
an entire function, as the integral~(\ref{fairytr}) can always be
differentiated with respect to $x\in \C$ under the integral sign. Thus, $f$
cannot vanish identically over any subinterval of $\R$. In particular, $f$
cannot have its support restricted to the half-line $[b,\infty)$, for any
$b\in \R$. The following theorem bounds the proportion of the energy of $f$
on $[b,\infty)$.

\begin{theorem}[Uncertainty principle]\label{thm:uncert}
Let $f$ be a Airy-bandlimited function with an Airy transform $\sigma$ that is
supported on $[a,\infty)$. Define 
  \begin{align}
\alpha^2 = \frac{\int_b^\infty f^2 \d x}{\int_{-\infty}^\infty f^2 \d x},
    \label{alphadef}
  \end{align}
where $b\in \R$. Then
  \begin{align}
\alpha^2 \le \int_b^\infty \!\! \int_{a}^\infty (\ai\left(x+y\right))^2 \d y \d x.
    \label{alphaineq}
  \end{align}
\end{theorem}

\begin{proof}
Squaring both sides of~(\ref{fairytr}) and applying the Cauchy-Schwarz
inequality, we have that
  \begin{align}
f(x)^2 \le \int_{a}^\infty (\ai(x+y))^2 \d y \int_a^\infty \sigma(y)^2 \d y.
    \label{fairytr2}
  \end{align}
After integrating both sides over $[b,\infty)$, the inequality becomes
  \begin{align}
\int_b^\infty f(x)^2 \d x \le \int_b^\infty \!\! \int_{a}^\infty
(\ai(x+y))^2 \d y \d x
\int_a^\infty \sigma(y)^2 \d y.
    \label{fairyineq}
  \end{align}
By dividing both sides of the inequality by $\int_{-\infty}^\infty f(x)^2 \d
x$, we get
  \begin{align}
\alpha^2 \le \Bigl(\int_b^\infty \!\! \int_{a}^\infty (\ai(x+y))^2 \d y \d x\Bigr)
\frac{\int_a^\infty \sigma(y)^2 \d y}{\int_{-\infty}^\infty f(x)^2 \d
x}.
    \label{alphaineq_tmp}
  \end{align}
Since the Airy transform is unitary, $\int_{-\infty}^\infty f(x)^2 \d x =
\int_{-\infty}^\infty \sigma(y)^2 \d y$.  Furthermore, by our assumption that
$\sigma$ is supported on $[a,\infty)$, we have that 
  \begin{align}
\frac{\int_a^\infty \sigma(y)^2 \d y}{\int_{-\infty}^\infty \sigma(y)^2 \d y} = 1.
  \end{align}
Thus, the inequality (\ref{alphaineq_tmp}) becomes 
  \begin{align}
\alpha^2 \le \int_b^\infty \!\! \int_{a}^\infty (\ai\left(x+y\right))^2 \d y \d x.
  \end{align}
\end{proof}

\begin{remark}
The right hand side of inequality (\ref{alphaineq}) decays rapidly when $b
\ge -a$.  In other words, when the Airy transform $\sigma$ of a function $f$
is supported on $[a,\infty)$, the function $f$ cannot have a large
proportion of its energy on the half-line $[b,\infty)$ when $b\ge-a$.
Furthermore, the proportion of energy it can have on $[b,\infty)$ decreases
rapidly as $b$ increases.
\end{remark}

In the following theorem, we give a bound on the decay rate of $f(x)$ for $x\ge-a$,
as follows. 
\begin{theorem}\label{thm:frate}
Let $f$ be a Airy-bandlimited function with an Airy transform $\sigma$ that is
supported on $[a,\infty)$. Then
  \begin{align}
\abs{f(x)} \le \ai\left(x+a\right) \int_{a}^\infty \abs{\sigma(y)} \d y,
    \label{turnpntineq}
  \end{align}
for all $x\ge -a$. In a mild abuse of terminology, we say that $f$ has a
turning point at $x=-a$.
\end{theorem}
\begin{proof}
From~(\ref{fairytr}),
it follows that
  \begin{align}
\abs{f(x)} \le \int_{a}^\infty \abs{\ai(x+y)} \abs{\sigma(y)} \d y.
  \end{align}
Since $\ai(x+a)$ is positive and monotonically decreasing for $x \ge -a$, we
have that
  \begin{align}
\abs{f(x)} \le \ai(x+a) \int_{a}^\infty \abs{\sigma(y)} \d y,
  \end{align}
for all $x\ge -a$.
\end{proof}

\subsection{Extremal properties of the eigenfunctions $\psi_{n,c}$}
  \label{sec:extremal}
In this section, we describe the extremal properties of the Airy integral
operator's eigenfunctions, in relation to the uncertainty principle described
in Theorem \ref{thm:uncert}.

\begin{theorem}\label{thm:extremal}
Let $f$ be a Airy-bandlimited function with an Airy transform $\sigma$ that is
supported on $[a,\infty)$. Then, for arbitrary $b\in \R$, $\alpha^2$ (defined
in (\ref{alphadef})) attains its maximum value $\lambda_{0,a+b}^2$ for
$\sigma(y)=\psi_{0,a+b}(y-a)$, where $\lambda_{0,a+b}$ and $\psi_{0,a+b}$
denote the first eigenvalue and eigenfunction of the Airy integral operator
with parameter $a+b$ (see Section \ref{aiintch}). In other words, the
inequality~(\ref{alphaineq}) can be refined to a tight inequality $\alpha^2
\le \lambda_{0,a+b}^2$.

\end{theorem}

\begin{proof}
By definition, it's easy to see that 
\begin{align}
\alpha^2=
\norm{F_{a,b}[\sigma]}^2 / \norm{f}^2  =\norm{F_{a,b}[\sigma]}^2 /
\norm{\sigma}^2,
\end{align} where $F_{a,b}$ is defined by formula 
(\ref{for:F}). By the usual min-max principle for singular values, we know
that the maximum value of $\alpha$ is thus the largest singular value of
$F_{a,b}$, and that this maximum value is attained when $\sigma$ is equal to the
corresponding right singular function of $F_{a,b}$. We observe that
\begin{align}
F_{a,b} = T_{-a} F_{0,a+b} T_{-a} = T_{-a} T_{a+b} (T_{-a-b} F_{0,a+b})
T_{-a} = T_b \T_{a+b} T_{-a},\label{for:Fab}
\end{align}
where $T_{(\cdot)}$ represents the
translation operator (see Definition \ref{def:T}), and $\T_{a+b}$
represents the Airy integral operator with parameter $a+b$. Since
$\lambda_{0,a+b}$ is the eigenvalue of $\T_{a+b}$ with the largest
magnitude, and $\psi_{0,a+b}$ is the corresponding eigenfunction, it follows
that $\abs{\lambda_{0,a+b}}$ and $T_{a}[\psi_{0,a+b}]$ are the largest
singular value and corresponding right singular function of $F_{a,b}$. Thus,
the largest possible value of $\alpha^2$ is $\lambda_{0,a+b}^2$, and this
value is attained by the function $\sigma(y)=\psi_{0,a+b}(y-a)$. 
\end{proof}

\begin{remark}
The eigenfunction $\psi_{n,c}$, for $n\neq 0$, obeys the same
optimality result, except that it's optimal in the intersection of 
$L^2[0,\infty)$ and $\spn\{\psi_{0,c},\psi_{1,c},\dots,\psi_{n-1,c}\}^\perp$.
\end{remark}

Finally, we characterize the behavior of the right singular functions of
$F_{a,b}$. Without loss of generality, we only need to consider the right
singular functions of the operator $\tilde F_{c} = F_{0,c}$, i.e., the
eigenfunctions $\psi_{n,c}$ of the Airy integral operator $\T_c$, since the
general case of the operator $F_{a,b}$ is related to $\tilde F_{a+b}$ only
by translations (see the first equality in (\ref{for:Fab})).

\begin{theorem}\label{thm:tpsiineq}
For any real $c$, the analytic continuation $\tilde\psi_{n,c}$ of the eigenfunction
$\psi_{n,c}$ of the Airy integral operator with parameter $c$ has a turning
point at $x=-c$, in the sense of Theorem \ref{thm:frate}. Furthermore, 
    \begin{align}
\abs{\tilde\psi_{n,c}(x)}\leq \frac{1}{\abs{\lambda_{n,c}}}\cdot\abs{\ai\left(x+c\right)}\int_0^\infty
\abs{\psi_{n,c}(y)} \d y, \label{for:tpsiineq}
    \end{align}
for $x\geq-c$, where $\lambda_{n,c}$ is the corresponding eigenvalue of
$\psi_{n,c}$.
\end{theorem}

\begin{proof}
By Theorem \ref{thm:anacont}, we have that 
  \begin{align}
\tilde\psi_{n,c}(x)=\frac{1}{\lambda_{n,c}}\cA[\psi_{n,c}](x+c).\label{for:anacont_app}
  \end{align}
Note that by Theorem \ref{thm:frate}, the Airy transform $\cA[\psi_{n,c}]$
of the right singular function $\psi_{n,c}$ of $\tilde F_c$ has a turning
point at $x=0$, so $\tilde\psi_{n,c}$ has a turning point at $x=-c$ by
(\ref{for:anacont_app}).

Furthermore, by combining (\ref{for:anacont_app}) and inequality
(\ref{turnpntineq}), we have that 
  \begin{align}
\hspace{-2em}\abs{\tilde\psi_{n,c}(x)}=\frac{1}{\abs{\lambda_{n,c}}}\cdot
\abs{\cA[\psi_{n,c}](x+c)}\leq \frac{1}{\abs{\lambda_{n,c}}}\cdot
\abs{\ai\left(x+c\right)}\int_0^\infty \abs{\psi_{n,c}(y)}\d y,
  \end{align}
for all $x\geq -c$.
\end{proof}

\subsection{Qualitative descriptions of the eigenfunction $\psi_{0,c}$ and
its Airy transform}\label{sec:qual}
By the extremal property of $\psi_{0,c}$ (see Theorem \ref{thm:extremal}),
we have that, for any $\sigma$ supported on $[0,\infty)$,
the proportion of the energy of $\cA[\sigma]$ on
$[c,\infty)$, i.e., the quantity
\begin{align}
\alpha^2=\frac{\int_c^\infty (\cA[\sigma](x))^2 \d x }{\int_{-\infty}^\infty (\cA[\sigma](x))^2 \d x},
\end{align}
attains its maximum
$\lambda_{0,c}^2$ with the choice $\sigma(y)=\psi_{0,c}(y)$.  Below, we
characterize the behavior of $\psi_{0,c}$ and its Airy transform, for $c$ in
three different regions.

\begin{itemize}
\item When $c<-5$, we have that $1-\alpha^2=1-\lambda_{0,c}^2 < 1.0\e{-3}$,
which means that the proportion of the energy of $\cA[\psi_{0,c}]$ on
$(-\infty,c]$ is negligible. In other words, $\cA[\psi_{0,c}]$ only has
negligible tail oscillations on the left.
Asymptotically, both $\psi_{0,c}$ and $\cA[\psi_{0,c}]$ converge to scaled
Gaussian functions on $[0,-c]$ and $[c,0]$, respectively, as $c\to-\infty$,
by Theorem \ref{thm:herm} in Appendix \ref{sec:misc} and Theorem
\ref{thm:anacont}.

\item When $c>1.5$, we have that $\alpha^2=\lambda_{0,c}^2 < 1.0\e{-3}$.
Note that, by Theorem \ref{thm:tpsiineq} and the fact that the Airy function
decays superexponentially, the eigenfunction $\psi_{0,c}(x)$ decays
increasingly fast for $x\geq 0$ as $c$ increases.
However,
we know that $\norm{\psi_{0,c}}_2=1$, which implies that it approaches a scaled
delta function. It follows that $\cA[\psi_{0,c}]$ approaches a scaled
Airy function as $c$ increases. Asymptotically, $\psi_{0,c}(x)\to
\sqrt{2}c^{1/4}e^{-\sqrt{c}x}$ as $c\to\infty$ by Theorem \ref{thm:lag} in
Appendix \ref{sec:misc}.

\item When $c\in[-5,1.5]$, generally we have that neither $\alpha^2$ nor
$1-\alpha^2$ is negligible. The former implies that the proportion of energy over
$[c,\infty)$ is substantial, which guarantees that a relatively
large proportion of the total energy is supported around the maximum of
$\cA[\psi_{0,c}](x)$ (empirically, close to $x=-1.5$) by Theorem
\ref{thm:frate}. The latter suggests that $\cA[\psi_{0,c}]$ has tail
oscillations. In fact, by Theorem \ref{thm:tpsiineq}, we know that
$\psi_{0,c}(x)$ decays for $x\geq \max(-c,0)$, so $\psi_{0,c}(x)$ also has a
substantial proportion of its energy near $x=0$. Therefore,
$\cA[\psi_{0,c}](x)$ still resembles a scaled Airy function.

\end{itemize}

Examples of the eigenfunctions $\psi_{n,c}$ and the square of the eigenvalues
$\lambda_{0,c}$ are shown in
Figure~\ref{figeigfun} and Figure~\ref{fig:lambda0}, respectively.

\section{Numerical Algorithm}\label{numalgo}
In this section, we describe a numerical algorithm that computes the
eigenvalues of the Airy integral operator to full relative accuracy, and
computes the eigenfunctions in the form of an expansion in scaled Laguerre
functions, where the expansion coefficients are also computed to full
relative accuracy.

\subsection{Discretization of the eigenfunctions}\label{4.1}
The algorithm for the evaluation of the eigenfunctions $\psi_{j,c}$ is based
on the expression of those functions as a series of scaled Laguerre
functions (see (\ref{slag})) of the form
    \begin{align}
\psi_{j,c}(x)=\sum_{k=0}^\infty \beta_{k}^{(j)} h_k^a(x),\label{exp}
    \end{align}
where the coefficients $\beta_{k}^{(j)}$ depends on the parameter $c$.
\begin{remark}
By orthogonality of the scaled Laguerre functions and the fact that
$\norm{\psi_{j,c}}_2^2=1$, we conclude that 
    \begin{align}
\sum_{k=0}^\infty \bigl(\beta_{k}^{(j)}\bigr)^2 = 1.\label{sum1}
    \end{align}
\end{remark}

Now we substitute the expansion (\ref{exp}) into (\ref{eigode2}), which gives us
    \begin{align}
\sum_{k=0}^\infty \beta_{k}^{(j)}\L_c[h_k^a]=\chi_{j,c}\sum_{k=0}^\infty \beta_{k}^{(j)}h_k^a.\label{slageig2}
    \end{align}
It follows from Theorem \ref{thm:Lh} that the left side of (\ref{slageig2}) can be
expanded into a summation that only involves $h_0^a,h_1^a,\dots$ .
Therefore, as the scaled Laguerre functions are linearly independent, the sequence
$\beta_{0}^{(j)},\beta_{1}^{(j)},\dots$ satisfies the recurrence relation
    \begin{align}
\hspace*{-3em}A_{0,0}\cdot \beta_{0}^{(j)}+A_{0,1}\cdot \beta_{1}^{(j)}+A_{0,2}\cdot \beta_{2}^{(j)}=&\ \chi_{j,c}\cdot \beta_{0}^{(j)},\label{mat1}\\
\hspace*{-3em}A_{1,0}\cdot \beta_{0}^{(j)}+A_{1,1}\cdot \beta_{1}^{(j)}+A_{1,2}\cdot \beta_{2}^{(j)}+A_{1,3}\cdot \beta_{3}^{(j)}=&\ \chi_{j,c}\cdot \beta_{1}^{(j)},\label{mat2}\\
\hspace*{-3em}A_{k,k-2}\cdot \beta_{k-2}^{(j)}+A_{k,k-1}\cdot \beta_{k-1}^{(j)}+A_{k,k}\cdot \beta_{k}^{(j)}\quad &\notag\\
+A_{k,k+1}\cdot \beta_{k+1}^{(j)}+A_{k,k+2}\cdot \beta_{k+2}^{(j)}=&\ \chi_{j,c}\cdot \beta_{k}^{(j)},\label{mat3}
    \end{align}
for $k=2,3,\dots$, where $A_{k,k},A_{k,k+1},A_{k,k+2}$ are defined via the formulas
    \begin{align}
A_{k,k}=&\ \frac{1}{4a^2}(8+a^3+4ac+24k+2a^3k+8ack+24k^2),\label{A1}\\
A_{k,k+1}=&\ A_{k+1,k}=\frac{1}{4a^2}(k+1)\bigl(a^3-4ac-16(k+1)\bigr),\label{A2}\\
A_{k,k+2}=&\ A_{k+2,k}=\frac{1}{a^2}(k+1)(k+2),\label{A3}
    \end{align}
for $k=0,1,\dots$ . Note that (\ref{mat1})--(\ref{mat3}) can
be written in the form of the following linear system:
    \begin{align}
(A-\chi_{j,c}I)\cdot \Bigl(\beta_{0}^{(j)},\beta_{1}^{(j)},\dots\Bigr)^T=0,
    \end{align}
where $I$ is the infinite identity matrix, and the non-zero entries of the
infinite symmetric matrix $A$ are given above.

Suppose that $k$ is a non-negative integer. Although the matrix $A$ is
infinite, and its entries do not decay with
increasing row or column number, the components of each eigenvector
$\beta^{(k)}$ decay super-algebraically (see Theorem \ref{decaypf}). More
specifically, the absolute values of components of the $k$-th eigenvector
will look like a bell-shaped curve centered at the $k$-th entry of the
eigenvector. Therefore, if we need to evaluate
the first $n+1$ eigenvalues $\chi_{0,c},\chi_{1,c},\dots,\chi_{n,c}$ and
eigenvectors $\beta^{(0)},\beta^{(1)},\dots,\beta^{(n)}$ numerically, we can
replace the infinite matrix $A$ with its $(N+1)\times (N+1)$ upper left
square submatrix, where $N=\O(n)$ is sufficiently large, which results in a
symmetric five-diagonal eigenproblem. It follows that we can replace the
series expansion (\ref{exp}) with a truncated one
    \begin{align}
\psi_{j,c}(x)=\sum_{k=0}^N \beta_{k}^{(j)} h_k^a(x),\label{exptrun}
    \end{align}
for $j=0,1,\dots,n$.

Assuming that we are interested in the first $n+1$ eigenfunctions of the
differential operator $\L_c$, it's important to pick the scaling factor $a$
such that $\psi_{n,c}$ gets best approximated, in the sense that the
bell-shape of the expansion coefficients of $\psi_{n,c}$ are concentrated
around $k=n$. By (\ref{sum1}), it follows that a considerably smaller matrix
will be required to calculate the $\psi_{n,c}$ accurately, compared with
other choices of $a$. Note that such an $a$ is not optimal for the
rest of the $n$ eigenfunctions (the eigenfunctions with indices from $0$ to
$n-1$), especially for the leading ones $\psi_{0,c}, \psi_{1,c},\dots$ .
However, in practice, if we can represent $\psi_{n,c}$ accurately, then the
rest of the $n$ eigenfunctions can be represented with at most the same
number of basis functions. Therefore, we only need to choose $a$ to
efficiently represent $\psi_{n,c}$.

To get a best approximation for $\psi_{n,c}$, we want the behavior of
$h_{n}^a$ to be similar to $\psi_{n,c}$. Notice that by (\ref{slagode}) and
(\ref{eigode}), the two ODEs satisfied by $h_n^a$, $\psi_{n,c}$ only differ
by the coefficient of the zero-th order term. It follows that the turning
point of $h_n^a$ is 
\begin{align}
    x=\frac{4n+2}{a},
\end{align}
while the turning point of $\psi_{n,c}$ is 
\begin{align}
    x=\frac{-c+\sqrt{c^2+4\chi_{n,c}}}{2}.
\end{align}
Matching the turning points of the two solutions, we get the following
approximation to the optimal $a$:
\begin{align}
a=\frac{4(2n+1)}{-c+\sqrt{c^2+4\chi_{n,c}}}.\label{choicea}
\end{align}
With this choice of $a$, $\beta_{k}$ decays quickly for $k\geq n$, for the
entire range of $c\in\R$. We note that the decay behavior of $\beta_k$ is
highly sensitive to the choice of $a$; other values of $a$ will often cause
$\beta_k$ to oscillate for a long time before it decays.
To simplify the notation, we will use $h_k(x)$ to denote
$h_k^{a}(x)$ with $a$ given by (\ref{choicea}), in the rest of the paper.
\begin{observation}\label{obschi}
    By applying the method of least squares to our numerical experiments,
    $\chi_{n,c}\approx 19.3c+11.1n+1.19\cdot 10^{-2}n^2+7.4\cdot 10^{-5}
    cn^2$ turns out to be a good approximation to the eigenvalues of the
    differential operator for $c\in[-50,50], n=0,1,\dots,800$. 
\end{observation}
\begin{observation}\label{obsN}
    Empirically, $\beta_k^{(n)}$ is much smaller than machine epsilon for
    $k\geq N$, where $N=1.1n+|c|+100$.
\end{observation}
\begin{observation}
    One might hope that, by a certain selection of basis functions, it's
    possible to split this five-diagonal eigenproblem into two tridiagonal
    eigenproblems (see, for example, \cite{prolb,royt}). However, it turns
    out that none of the classical orthogonal polynomials (Laguerre
    polynomials, Hermite polynomials, or their rescaled versions) defined on
    the interval $[0,\infty)$ have the capability to split our five-diagonal
    eigenproblem.
\end{observation}
\begin{observation}
When $c$ is negative, the leading few eigenvalues, say,
$\chi_{0,c},\chi_{1,c},\dots,\chi_{n',c}$, are negative, where $n'$ is
usually smaller than $100$ in practical situations. In this case, provided
one is only interested in the first $n$ eigenfunctions, where $n-1\leq n'$,
it would appear that the approximation of $a$ given by formula
(\ref{choicea}) may fail, since $c^2+4\chi_{n,c}$ can be negative. However,
$c^2+4\chi_{n,c}$ turns out to always be positive. To estimate $a$, we use
an approximation to $\chi_{n,c}$, for which the quantity $c^2+4\chi_{n,c}$
can, at least in principle, be negative. This turns out to also not be a
problem, since even when we only care about a small number of
eigenfunctions, we can always compute more, say, $n+100$, for which
$c^2+4\chi_{n+99,c}$ is positive.
\end{observation}

\subsection{Relative accuracy evaluation of the expansion coefficients of
the eigenfunctions}
\label{acceigfun}
Suppose that $n$ is a non-negative integer. In Section \ref{4.1}, we expand
each of the eigenfunctions $\psi_0,\psi_1,\dots,\psi_n$ into a series of
scaled Laguerre functions, and formulate an eigenproblem to solve for the
expansion coefficients $\{\beta^{(j)}_k\}$ of $\psi_j$. We showed that, for
the choice of basis functions described in Section \ref{4.1}, the number of
required expansion coefficients $N$ is not much larger than $n$. In fact, by
Observation \ref{obsN}, the choice $N=1.1n+|c|+100$ is sufficient for all
$c\in \R$. The coefficients are thus the solution to an eigenproblem
involving a $(N+1)\times(N+1)$ five-diagonal matrix. Intuitively, one may
suggest applying a standard eigensolver to solve for all eigenpairs of the
five-diagonal matrix $A$. However, in this case, the eigenvalues and
eigenvectors will only be evaluated to absolute precision, which turns out
not to be sufficient for the relative accuracy evaluation of the spectrum of
the Airy integral operator $\T_c$. Instead, we use the fact that, since the
matrix is five-diagonal, the eigenvalues can be evaluated to relative
precision and the eigenvectors can be evaluated to coordinate-wise relative
precision using the inverse power method (see~\cite{osipov} for a discussion
of the phenomenon). We derive the following algorithm for the relative
accuracy evaluation of expansion coefficients of eigenfunctions
$\{\psi_{j}\}_{j=0,1,\dots,n}$ and the spectrum of $\L_c$:

\begin{enumerate}
    \item Construct an $(N+1)\times (N+1)$ five-diagonal symmetric real
    matrix $A$ whose entries are defined via (\ref{A1})--(\ref{A3}), 
    where $a$ is chosen by formula (\ref{choicea}) and
    Observation \ref{obschi}, and $N$ is given by Observation \ref{obsN}.
    \item Apply a standard symmetric five-diagonal eigenvalue solver to $A$
    to get a approximation of its eigenvalues $\chi_0,\chi_1,\dots,\chi_N$
    to absolute precision.
    \item Apply the shifted inverse power method to $A$ with an initial
    shift of $\chi_0,\chi_1,\dots,\chi_n$, until convergence. This leads to
    an approximation of the expansion coefficients of
    $\{\psi_j\}_{j=0,1,\dots,n}$ to coordinate-wise relative precision, and
    the spectrum of $\L_c$ to relative precision.
\end{enumerate}

\begin{remark}\label{abserr}
For any $j\in \{0,1,\dots, n\}$, let
$\tilde\beta^{(j)}=\Bigl(\tilde\beta_0^{(j)},\tilde\beta_1^{(j)},\dots,\tilde\beta_N^{(j)}\Bigr)
\in \R^{N+1}$ denote the exact values of the first $N+1$ coefficients of the
expansion of $\psi_j$.  Then, each component of the approximation $\beta_j$
produced by the shifted inverse power method in the third step of the
algorithm has the following property, no matter how tiny the component is:
\begin{align}
    \frac{\abs{\beta_k^{(j)}-\tilde\beta_{k}^{(j)}}}{|\tilde\beta_{k}^{(j)}|}<\epsilon,\
    \forall k\in\{0,1,\dots,N\},
\end{align}
where $\epsilon$ represents the machine epsilon (see \cite{osipov} for more details). 
However with a standard eigensolver, one can only achieve
    \begin{align}
        \abs{\beta_k^{(j)}-\tilde\beta_{k}^{(j)}}<\epsilon,\ \forall
        k\in\{0,1,\dots,N\},
    \end{align}
although in norm,
  \begin{align}
\frac{\norm{\beta^{(j)}-\tilde\beta^{(j)}}_2}{\norm{\tilde\beta^{(j)}}_2}<\epsilon.
  \end{align}
In other words, the standard eigensolver can only achieve absolute precision
for each coordinate of the eigenvectors, while the shifted inverse power
method achieves relative precision. This is because the small
entries in the eigenvector only interact with adjacent entries in the
eigenvector in the course of a solve step during the shifted inverse power
method.
\end{remark}

\begin{observation}\label{obs:relacc}
The relative accuracy evaluation of expansion coefficients is essential both
for performing high accuracy spectral differentiation of the eigenfunctions,
and for relative accuracy evaluation of the eigenfunctions $\psi_{j,c}(x)$
for large $x$, where the eigenfunctions are small. 
\end{observation}

\begin{remark}
    The eigenvectors $\beta^{(n+1)},\beta^{(n+2)},\dots,\beta^{(N)}\in
    \R^{N+1}$ are never used in our algorithm, since they do not have
    sufficient number of terms to represent $\psi_{n+1},\psi_{n+2},\dots,
    \psi_{N}$, respectively.
\end{remark}

\begin{remark}
The first and second steps of the algorithm cost $\O(n)$ and $\O(n^2)$
operations, respectively. The shifted inverse power method is applied to $n$
eigenpairs in the third step, and each iteration costs $\O(n)$ operations.
The convergence usually requires less than five iterations, since the
initial guesses for the eigenvalues are correct to absolute precision, 
the eigenvalues are well-separated (see Section \ref{commutechap}), and the
inverse power method converges cubically in the vicinity of the solution.
Thus, the third step costs $\O(n^2)$ operations. So, in total, the cost of
the algorithm is $\O(n^2)$ operations. 
\end{remark}

\subsection{Relative accuracy evaluation of the spectrum of the integral operator}
In this subsection, we introduce an algorithm that 
evaluates the Airy integral operator $\T_c$'s eigenvalues
$\lambda_0,\lambda_1,\dots,\lambda_n$ to relative precision, using the
expansion coefficients of the eigenfunctions computed by the algorithm in
Section \ref{acceigfun}.

\subsubsection{Evaluation of the first eigenvalue}
By (\ref{Teigfun}), we know that
    \begin{align}
\lambda_j=\frac{\int_0^\infty \ai(x+y+c)\psi_{j}(y)  \d y}{\psi_j(x)}.\label{naive}
    \end{align}
We will show that, when the expansion coefficients of $\psi_{0}$ are known
to relative accuracy, for a particular choice of $x$, (\ref{naive}) can be
used to evaluate $\lambda_{0}$ to relative accuracy.

Firstly, we discuss how to pick an optimal $x$, such that the evaluation is
well-conditioned. Mathematically, the choice of $x$ makes no
difference to the value of $\lambda_0$, but numerically, it's better to
select $x$ such that there's minimal cancellation in evaluating both
$\psi_0(x)$ and $\int_0^\infty \ai(x+y+c)\psi_{0}(y) \d y$. To achieve this,
we notice that the Airy function is smooth and decaying on the right
half-plane, and oscillatory on the left half-plane. When $c$ is
non-negative, the integrand is decaying superexponentially fast for any value of
$x\geq 0$, and $x=0$ becomes a natural choice, since, for this value of $x$, the
integrand is the largest. When $c$ is negative, the integrand decays
superexponentially fast only when $x\geq -c$, so, in that case, $x=-c$ is
similarly a natural choice. Therefore, we define $x$ to be
\begin{align}\label{xchoice}
x=\begin{cases}
    0,\quad &\text{if } c\geq 0\\
    -c,\quad &\text{otherwise} \end{cases}.
\end{align}
We note that, when $j=0$, formula (\ref{naive}) is well-defined 
when $x$ is given by formula (\ref{xchoice}), as follows. When $c\geq 0$,
Theorem \ref{thm:rec} in Appendix~\ref{sec:misc} shows that $\psi_{0}(0)\neq 0$.
When $c<0$, we have that $-c>0$, so $\psi_0(-c)\neq 0$ by the Sturm
oscillation theorem.

Once the value of $x$ is chosen, we substitute the truncated expansion
(\ref{exptrun}) of $\psi_0$ into (\ref{naive}), to get
    \begin{align}
\lambda_0=\frac{\sum_{k=0}^{N} \beta_{k}^{(0)} \big( \int_0^\infty
\ai(x+y+c)h_k^a(y)\d y\big)}{\sum_{k=0}^{N}
\beta_{k}^{(0)}h_k^a(x)}.\label{lambda0eq2}
    \end{align}
Note that the scaled Laguerre functions are easy to evaluate, and in the
last section, we've already solved for $\{\beta_k^{(0)}\}_{k=0,1,\dots,N}$
to relative accuracy. Thus, it's straightforward to compute the denominator
of (\ref{lambda0eq2}), and for our choice of $x$, it is evaluated without
cancellation error. However, the computation of the numerator is more
difficult due to the presence of integral $\int_0^\infty
\ai(x+y+c)h_k^a(y)\d y$. The integrand is both highly oscillatory and
rapidly decaying as $k$ gets larger, which implies that a standard
quadrature rule will be insufficient. Instead, we derive a five-term linear
homogeneous recurrence relation for $\int_0^\infty \ai(x+y+c)h_k^a(y)\d y$
that satisfies a certain linear condition involving the first four terms (see
Theorem \ref{recur}), and by combining it with the inverse power
method, we find that the integrals are evaluated to relative
accuracy, for all values of $k=0,1,\dots,N$. The main ideas of the algorithm
are as follows.

For consistency, we use $H_k^a$, which is first defined in Theorem
\ref{recur}, to represent the integral $\int_0^\infty \ai(x+y+c)h_k^a(y)\d
y$. It follows that the variable $s$, defined in formula~(\ref{Hna}) of Theorem
\ref{recur}, equals $x+c$ in our case. Clearly, the absolute value of
$H_k^a$ decays exponentially fast as $k$ increases, since the integrand
becomes more and more oscillatory (See Theorem \ref{decaycorr}). The key
empirical observation is that only one of the three linearly independent
solutions to the five-term linear homogeneous recurrence relation satisfying
(\ref{fiveterm}), for $n=1$, decays as $k\to\infty$. This implies that, by
truncating the infinite matrix associated with the recurrence relation and
evaluating the eigenvector corresponding to the zero eigenvalue, we can
solve for $H_k^a$ in a manner similar to Section \ref{4.1}. To put it more
precisely, we first write out the recurrence relation in the form of a
linear system:
    \begin{align}
\hspace*{0em}B_{1,0}H_0^a + B_{1,1}H_1^a+B_{1,2}H_2^a+B_{1,3}H_{3}^a=&\
0,\label{constraint}\\
\hspace*{-3em}B_{k-2,k}H_{k-2}^a+B_{k-1,k}H_{k-1}^a+B_{k,k}H_k^a+B_{k+1,k}H_{k+1}^a+B_{k+2,k}H_{k+2}^a=&\
0,
    \end{align}
for $k=2,3,\dots$, where $B_{k-2,k},B_{k-1,k},B_{k,k},B_{k+1,k},B_{k+2,k}$
are defined via the formulas
\begin{align}
    B_{k-2,k}=&\ k-1,\label{Bmat_st}\\
    B_{k-1,k}=&\ -(4k-1+a(x+c)-\frac{1}{4}a^3),\\
    B_{k,k}=&\ 6k+3+2a(x+c)+\frac{1}{2}a^3,\\
    B_{k,k+1}=&\ -(4k+5+a(x+c)-\frac{1}{4}a^3),\\
    B_{k,k+2}=&\ k+2,\label{Bmat}
\end{align}
for $k=1,2,\dots$ . Note that the first row of the infinite matrix $B$ is
all zeros.  If we consider the eigenproblem for the infinite matrix $B$, by
our observation, it must have an eigenvector corresponding to the zero
eigenvalue, and the coordinates of the eigenvector decay exponentially fast.
Therefore, if we want to evaluate the first $N+1$ coordinates of the
eigenvector with eigenvalue zero, we can replace the infinite matrix $B$
with its $(N'+1)\times (N'+1)$ upper left square submatrix, where
$N'=\O(N)$ is sufficiently large, and apply the inverse power method to $B$.
The empirical fact that there is only one decaying solution to the recurence
relation which satisfies (\ref{constraint}) means that this leads to an
eigenvector $\{\tilde H_{k}^a\}_{k=0,1,\dots,N'}$ whose first $N+1$
coordinates match $\{H_{k}^a\}_{k=0,1,\dots,N}$ to relative accuracy, up to
some scalar factor. 
\begin{remark}
    To avoid division by zero, we set $B_{0,0}$ to be $\epsilon$ during
    computation, where $\epsilon$ is the smallest floating-point number.
    Since we are performing the inverse power method, division by a tiny
    number is numerically stable.
\end{remark}

Therefore, the last step is to rescale the
eigenvector, such that its
$k$-th coordinate equals $H_k^a$, for all $k$. This can be achieved by first
computing $H_0^a$ to relative precision, and multiplying every coordinate of
the eigenvector by $H_0^a/\tilde H_0^a$. Note that, by our particular choice
of $x$, the integrand of $H_0^a=\int_{0}^\infty \ai(x+y+c)h_0^a(y)\d y$ is
smooth and decays superexponentially and monotonically. Thus, the evaluation can
be done rapidly and accurately via quadrature.

\begin{observation}
  It's important to truncate the domain of the integral $\int_0^\infty
  \ai(x+y+c)h_0^a(y) \d y$ properly when it is integrated numerically, since
  otherwise it's either impossible or too expensive to compute the integral
  to full relative precision. A good rule for truncating the domain of the 
  integral is to choose the domain where the absolute value of the integrand is
  larger than machine epsilon times the $L^\infty$ norm of the integrand.
  Since $\max_{y\geq 0} \ai(x+y+c)h_0^a(y)= \sqrt{a}\ai(x+c)$, where $x$ is chosen
  by (\ref{xchoice}), we construct an approximate formula for the cutoff
  point $y_{\text{max}}$ such that
  $\ai(x+y_{\text{max}}+c)h_0^a(y_{\text{max}})\approx \epsilon \sqrt{a}\ai(x+c)$ by
  using Remark \ref{asymp} and symbolic computation, where $\epsilon$
  represents the machine epsilon. 
\end{observation}  

\begin{observation}\label{obsNp}
    Empirically, $N'=N+40$ is a safe choice for the truncation of the
    infinite matrix $B$.
\end{observation}

The first eigenvalue of the integral operator $\T_c$ can now be evaluated to
relative precision by (\ref{lambda0eq2}), using our computed expansion
coefficients $\beta^{(0)}$ and the solution to the recurrence relation
$\{H_{k}^a\}_{k=0,1,\dots,N}$. 

\begin{remark}
One may suggest using numerical integration to compute $\int_0^\infty
\ai(x+y+c)\psi_{0}(y) \d y$ directly, since the integrand decays
superexponentially and is smooth. However, it's rather involved to generate sets
of good quadrature nodes that integrate $\int_0^\infty \ai(x+y+c)\psi_{0}(y)
\d y$ to full relative precision for all ranges of $c$, since the behavior of
the eigenfunction $\psi_0$ is strongly dependent on $c$. Adaptive quadrature
could be applied to overcome this issue, but it is generally not efficient
and robust enough to be used in an algorithm for computing special
functions. On the other hand, the algorithm that we propose only requires
the numerical integration of $\int_0^\infty \ai(x+y+c)h_0^a(y) \d y$, whose
behavior is substantially easier to characterize, since
$h_0^a(y):=\sqrt{a}e^{-ay/2}$ is only weakly dependent on $c$ (see formula
(\ref{choicea})).
\end{remark}

\subsubsection{Evaluation of the rest of the eigenvalues}
The standard way to overcome the obstacle for the numerical evaluation of small
$\lambda_j$'s is to compute all the ratios $\frac{\lambda_1}{\lambda_0},\dots,
\frac{\lambda_n}{\lambda_{n-1}}$, and then evaluate the eigenvalue $\lambda_j$ via
the formula
    \begin{align}
\lambda_j=\lambda_0\cdot
\frac{\lambda_1}{\lambda_0}\cdot\cdots\cdot\frac{\lambda_j}{\lambda_{j-1}},
\label{ratios}
    \end{align}
where the ratio $\frac{\lambda_{n+1}}{\lambda_n}$ can be computed by Theorem
\ref{ratiothm}: 
\begin{align}
\frac{\lambda_{n+1}}{\lambda_n}=\frac{\int_0^\infty \psi_n'(x)\psi_{n+1}(x)\d
x}{\int_0^\infty \psi_n(x)\psi_{n+1}'(x)\d x},\label{rationp1}
\end{align}
(see Section 10.2 in \cite{prolb}).

We note that the computation of the ratio can be done spectrally: for example, one can
evaluate the numerator of (\ref{rationp1}) by first computing the expansion
of $\psi_{n}'$ via Corollary \ref{spcdiff}, then computing
the inner product of the two series expansions of $\psi_{n}'$ and
$\psi_{n+1}$ by the orthogonality of the basis functions. The denominator is
symmetric to the numerator, and can be computed in essentially the same way.
Therefore, it takes $\O(N)$ operations to compute
$\frac{\lambda_{n+1}}{\lambda_n}$, and takes $\O(nN)$ operations in total to
compute $\lambda_j$ for $j=1,2,\dots,n$. Recalling that $N=1.1n+|c|+100$, we
see that the cost is $\O(n^2+|c|n)$.

\begin{remark}\label{H0k}
One may also compute the expansion of the derivative of $\psi_n$ by applying
a differentiation matrix (see formula (\ref{dslag})) to the expansion coefficients
$\beta^{(n)}$ of $\psi_n$. However, this will cost $\O(N^2)$ operations for
each differentiation, which makes the total cost $\O(nN^2)$ operations. 
\end{remark}

\begin{observation}
It's important that the expansion coefficients of the eigenfunctions are
computed to relative accuracy, since otherwise the spectral differentiation of the
eigenfunctions in formula (\ref{rationp1}) will lead to a loss of accuracy
proportional to the order of the expansion (see Observation
\ref{obs:relacc}).  
\end{observation}

\medskip
Given the expansion coefficients $\{\beta_k^{(j)}\}$ computed by
the algorithm stated in Section \ref{acceigfun}, we summarize the
algorithm for computing the eigenvalues as follows.
\begin{enumerate}
    \item Construct an $(N'+1)\times (N'+1)$ five-diagonal real
    matrix $B$ whose entries are defined via (\ref{Bmat_st})--(\ref{Bmat}), where $N'=N+40$
    (see Observation \ref{obsNp}), and $N$ is given by Observation
    \ref{obsN}.
    \item Apply the inverse power method to $B$ until convergence. This leads to
    an approximation of an eigenvector $\{\tilde H_{k}^a\}_{k=0,1,\dots,N'}$
    whose first $N+1$ coordinates match $\{H_{k}^a\}_{k=0,1,\dots,N}$ to
    relative accuracy, up to some scalar factor. 
    \item Compute $H_0^a$ to relative precision by numerical integration
    (see Remark \ref{H0k}). Rescale the computed eigenvector by multiplying
    every coordinate by $H_0^a/\tilde H_0^a$.
    \item Compute $\lambda_0$ using the previously computed $\{\beta_k^{(0)}\}$
    and $\{H_{k}^a\}$ via formula (\ref{lambda0eq2}), where the
    value of $x$ inside that formula is chosen by (\ref{xchoice}).
    \item Compute the rest of the eigenvalues by formulas (\ref{ratios}),
    (\ref{rationp1}) with the use of spectral differentiation (see Corollary
    \ref{spcdiff}) and the orthogonality of the basis functions.
\end{enumerate}

\section{Applications}
In this section, we discuss two applications of the eigendecomposition of
the Airy integral operator.  In Section~\ref{sec:klevel}, we discuss an
application to the distributions of the $k$-th largest level at the soft edge
scaling limit of Gaussian ensembles, and in Section~\ref{optics}, we discuss
an application to finite-energy Airy beams in optics.

\subsection{Distributions of the $k$-th largest level at the soft edge
scaling limit of Gaussian ensembles}\label{sec:klevel}

The cumulative distribution function of the $k$-th largest level at the soft
edge scaling limit of the GUE is given by the formula
  \begin{align}
F_2(k;s)= \sum_{j=0}^{k-1} \frac{(-1)^j}{j!} \pp{^j}{z^j} \det\bigl(I - z
\cK|_{L^2[s,\infty)}\bigr) \Bigr|_{z=1},
  \end{align}
where $\cK|_{L^2[s,\infty)}$ denotes the integral operator on $L^2[s,\infty)$
with kernel
\begin{align}
    K_{Ai}(x,y)=\int_s^\infty \ai(x+z-s)\ai(z+y-s) \d z.
\end{align}
It's clear that 
    \begin{align}
\cK|_{L^2[s,\infty)}[f]=\G_s^2 [f],\label{KGeq}
    \end{align}
where $\G_s$ is the associated Airy integral operator defined
in Section \ref{aiintch}.

Using the fact that 
  \begin{align}
  \det \bigl(I-z\cK|_{L^2[s,\infty)}\bigr)=\prod_{i=0}^\infty (1-z\lambda_{i,s}^2),
  \end{align}
$F_2(k;s)$ can be expressed in the following form:
    \begin{align}
\hspace*{-5em}F_2(k;s)=\sum_{j=0}^{k-1} \frac{1}{j!} \sum_{i_1=0}^\infty
\lambda_{i_1,s}^2 \sum_{\substack{i_2=0,\\i_2\neq i_1}}^\infty
\lambda_{i_2,s}^2\quad \dots \sum_{\substack{i_j=0,\\i_j\neq
i_1,\dots,i_{j-1}}}^\infty\lambda_{i_j,s}^2\prod_{\substack{i=0,\\i\neq
i_1,\dots,i_j}}^\infty(1-\lambda_{i,s}^2),\label{f2k}
    \end{align}
where $\lambda_i$ is the $(i+1)$-th eigenvalue of $\G_s$. The formula 
\begin{align}
    \hspace*{-6em}\frac{d}{ds}F_2(k;s)=\frac{1}{(k-1)!} \sum_{i_1=0}^\infty
    \lambda_{i_1,s}^2\sum_{\substack{i_2=0,\\i_2\neq i_1}}^\infty
    \lambda_{i_2,s}^2\quad\dots \sum_{\substack{i_{k}=0,\\i_{k}\neq
    i_1,\dots,i_{k-1}}}^\infty
    (-\pp{\lambda_{i_k,s}^2}{s})\prod_{\substack{i=0,\\i\neq
    i_1,\dots,i_k}}^\infty (1-\lambda_{i,s}^2)\label{df2k_tmp}
\end{align}
for the probability density function $\frac{d}{ds}F_2(k;s)$ of the $k$-th
largest level at the soft edge scaling limit of the GUE is obtained from
(\ref{f2k}) by a lengthy calculation in which many terms cancel. By applying
the identity
\begin{align}
    \pp{\lambda_{n,s}^2}{s}=-\lambda_{n,s}^2\bigl(\psi_{n,s}(0)\bigr)^2
\end{align}
(see Corollary \ref{dl2dc}) to formula (\ref{df2k_tmp}), the PDF
$\frac{d}{ds}F_2(k;s)$ gets expressed in terms of the eigenvalues
$\{\lambda_{i,s}\}$ and the values of the eigenfunctions $\{\psi_{i,s}(x)\}$ of
the Airy integral operator $\T_s$ at $x=0$:
    \begin{align}
\hspace*{-6em}\frac{d}{ds}F_2(k;s)=\frac{1}{(k-1)!} \sum_{i_1=0}^\infty
\lambda_{i_1,s}^2\sum_{\substack{i_2=0,\\i_2\neq i_1}}^\infty
\lambda_{i_2,s}^2\quad\dots \sum_{\substack{i_{k}=0,\\i_{k}\neq
i_1,\dots,i_{k-1}}}^\infty \lambda_{i_{k},s}^2 \big(\psi_{i_k,s}(0)\big)^2
\prod_{\substack{i=0,\\i\neq i_1,\dots,i_k}}^\infty
(1-\lambda_{i,s}^2).\label{df2k}
    \end{align}
Clearly, with the eigenvalues $\{\lambda_{j,s}\}$ and expansion coefficients
$\{\beta^{(j)}\}$ of the eigenfunctions $\{\psi_{j,s}\}$ computed to full
relative precision, the PDF $\frac{d}{ds}F_2(k;s)$ can be evaluated to
relative precision everywhere, except in the left tail, for any positive
integer $k$. We note that, in this case, knowing the eigenvalues to relative
precision is essential, since if the eigenvalues are only computed to
absolute precision, $\frac{d}{ds}F_2(k;s)$ loses accuracy exponentially fast
for any fixed $s$ as $k$ increases. Finally, we observe
that the left tail of the PDF is evaluated only to absolute precision due to
the cancellation error in the computation of $\psi_{j,s}(0)$ and
$1-\lambda_{j,s}^2$. 
\begin{observation}
When $k=1$, $\frac{d}{ds}F_2(k;s)$ reduces to the PDF of the Tracy-Widom
distribution $\frac{d}{ds}F_2(s)$, and, by the discussion above, the number
of correct digits of $\frac{d}{ds}F_2(s)$ is 
approximately equal to the number of correct digits of $\lambda_{0,s}$, for
all $s$ except in the left tail. Although, in general, the Fredholm
determinant method introduced in \cite{borndet} only solves eigenvalues to
absolute precision, the first eigenvalue $\lambda_{0,s}$ is actually
computed to relative precision. Therefore, by using formula
(\ref{df2k}), the Tracy-Widom distribution can be evaluated to relative
precision everywhere with Bornemann's method, except in the left tail.
However, to our knowledge, formula (\ref{df2k}) was not used in the
computation of the PDF until this paper. We also recall that evaluating
$\frac{d}{ds}F_2(k;s)$ for $k\geq 2$ to relative precision requires the
eigenvalues beyond $\lambda_{0,s}$ to be computed to relative precision.
\end{observation}

\begin{observation}\label{fasttw}
Provided that the eigenvalues
$\lambda_{i,s}$ and the values of the eigenfunctions $\psi_{i,s}$ at zero
are given, and each  series in the nested representations (\ref{f2k}),
(\ref{df2k}) is truncated at the $n$-th term,
the time complexities of computing $F_2(k;s)$ and $\frac{d}{ds}F_2(k;s)$ via
the series (\ref{f2k}), (\ref{df2k}) are $\O(n^k)$
and $\O(n^{k+1})$, respectively. The cost appears at first
glance to be unaffordable when $k$ is large, but, in reality,
only a fixed constant number of terms in the infinite
series is required to
compute $F_2(k;s)$ and $\frac{d}{ds}F_2(k;s)$ for all $k$ to full relative
accuracy, owing to the exponential decay of the eigenvalues $\lambda_{i,s}$.
Thus, the time complexity of evaluating $F_2(k;s)$ and
$\frac{d}{ds}F_2(k;s)$ is $\O(k)$.  We also recall that the
computation of $\{\lambda_{i,s}\}_{i=0,1,\dots,n-1}$ and
$\{\psi_{i,s}(0)\}_{i=0,1,\dots,n-1}$ requires $\O(n^2)$ operations (see
Section \ref{numalgo}).

\end{observation}

Similarly, the cumulative distribution function $F_1(k;s)$ of the $k$-th
largest level at the soft edge scaling limit of the GOE equals
    \begin{align}
\hspace*{-5em}F_1(k;s)=\frac{1}{2}\sum_{j=0}^{k-1}\frac{(-1)^j}{j!}\pp{^j}{z^j}\biggl(&\,\biggl(1+\sqrt{\frac{z}{2-z}}\biggr)\det\Bigl(I
- \sqrt{z(2-z)} \T_{s/2}|_{L^2[0,\infty)}\Bigr)\notag\\
  +&\,\biggl(1-\sqrt{\frac{z}{2-z}}\biggr)\det\Bigl(I + \sqrt{z(2-z)}
  \T_{s/2}|_{L^2[0,\infty)}\Bigr)\biggr)\bigg|_{z=1},
    \end{align}
and the cumulative distribution function $F_4(k;s)$ of the $k$-th largest
level at the soft edge scaling limit of the GSE can be written as
    \begin{align}
\hspace*{-6em}F_4(k;s)=\frac{1}{2}\sum_{j=0}^{k-1}\frac{(-1)^j}{j!}\pp{^j}{z^j}\biggl(\det\Bigl(I
- \sqrt{z} \T_{s/2}|_{L^2[0,\infty)}\Bigr)+\det\Bigl(I + \sqrt{z}
\T_{s/2}|_{L^2[0,\infty)}\Bigr)\biggr)\bigg|_{z=1},
    \end{align}
(see \cite{bornrmt}). It follows that the distributions (including both the
CDFs and PDFs) can be expressed in terms of the eigenvalues and
eigenfunctions of the Airy integral operator $\T_{s/2}$, in a manner similar to the
GUE case (see formulas (\ref{f2k}), (\ref{df2k})). Thus, the distributions
can also be computed to high accuracy using our method.

\begin{remark}
Two popular methods for computing the Tracy-Widom distribution are: solving for a Painlev\'e
transcendent \cite{bornrmt,rao}, and approximating a Fredholm determinant of
an integral operator \cite{borndet}. When high accuracy is not required,
other effective methods can be used, including methods based on a shifted
Gamma distribution approximation \cite{chiani}, and direct statistical simulation
\cite{olof}.
\end{remark}

\subsection{Connection to Airy beams in optics}
  \label{optics}

In this section, we describe an application of the eigenfunctions of the
Airy integral operator to the construction of an optimal finite-energy
approximation to a certain optical beam called the Airy beam. We begin by
describing the equations governing the propagation of light in free space.

The propagation of light in free space, in the absence of currents and charges,
is governed by Maxwell's equations
  \begin{align}
\nabla \times H - \frac{\epsilon}{c} E' &= 0, \label{heqn} \\
\nabla \times E + \frac{\mu}{c} H' &= 0, \label{eeqn} \\
\nabla \cdot E &= 0, \label{enodiv} \\ 
\nabla \cdot H &= 0, \label{hnodiv}
  \end{align}
where $E$ and $H$ denote the electric and magnetic fields, respectively, $\epsilon$
is the permittivity, and $\mu$ is the magnetic permeability.
From~(\ref{heqn})--(\ref{hnodiv}),
it can be shown that
  \begin{align}
\nabla^2 E - \frac{\epsilon \mu}{c^2} E'' &= 0, \label{eeqn2} \\
\nabla^2 H - \frac{\epsilon \mu}{c^2} H'' &= 0, \label{heqn2}
  \end{align}
where the equations are satisfied separately by each of the components of
$E=(E_x,E_y,E_z)$ and $H=(H_x,H_y,H_z)$, respectively (see, for
example,~\cite{born}). When the light is monochromatic or time-harmonic with
frequency $\omega$, the electric field takes the form $E(r) = \Re(U(r)
e^{-i\omega t})$, where, after subtituting into~(\ref{eeqn2}), we find that
$U$ solves the Helmholtz equation
  \begin{align}
\nabla^2 U + k_0^2 n^2 U = 0,
  \end{align}
where $k_0=\omega/c$ is the reduced or vacuum wavenumber, $n=\sqrt{\epsilon \mu}$
is the absolute refractive index of the medium, and where the equation is 
again satisfied separately by each component of $U=(U_x,U_y,U_z)$.
Letting $\psi$ denote a single component of $U$ and letting $k_H = k_0 n$,
we have that
  \begin{align}
\nabla^2 \psi + k_H^2 \psi = 0,
    \label{psieqn}
  \end{align}
where $k_H$ is called the free space wavenumber.

\subsubsection{Propagation-invariant optical fields}

If we suppose that $\psi$ has the form 
  \begin{align}
\psi(x,y,z) = \Psi(x,y) e^{ik_z z},
    \label{psiform}
  \end{align}
then the intensity of that particular component of the electric field will
be invariant along the $z$-axis (which we call the axial direction).
Substituting $\psi$ into~(\ref{psieqn}), we find that
  \begin{align}
\nabla^2 \Psi + k_t^2 \Psi = 0,
    \label{Psieqn}
  \end{align}
where
  \begin{align}
k_t = \sqrt{k_H^2 - k_z^2},
  \end{align}
and $k_t$ denotes the transverse wave number. Suppose that each component of
the electric field has the form~(\ref{psiform}).  If $k_z>0$, then the
transverse parts of the $E_x$ and $E_y$ components of the
electric field can be chosen to be any two solutions of~(\ref{Psieqn}), and
the axial component $E_z$ is then determined by Maxwell's equations
(see, for example,~\S3.1 of~\cite{pio54}). If $k_z \approx k_H$, then
most of the propagation will be in the axial (meaning $z$) direction, and the
component $E_z$ will be very small. In this situation, the overall intensity
of the electric field is well approximated by the intensity of the field in
just the transverse ($x$-$y$) plane. Solutions to~(\ref{psieqn}) 
are known as waves, and waves of the form~(\ref{psiform}) are examples of
propagation-invariant optical fields (PIOFs) (see, for
example,~\cite{pio54} and~\cite{pio61}).

\subsubsection{The paraxial wave equation}

Instead of assuming that the transverse part of the 
field component $\psi$ is invariant in the $z$ direction, suppose that 
the transverse component varies slowly with respect to $z$, so that
  \begin{align}
\psi(x,y,z) = \Psi(x,y,z) e^{ik_H z},
    \label{Psiform2}
  \end{align}
where $\Psi$ varies slowly with $z$. Substituting~(\ref{Psiform2})
into~(\ref{psieqn}), we have
  \begin{align}
\nabla_t^2 \Psi e^{i k_H z} + \pp{^2\Psi}{z^2} e^{i k_H z} + 2i k_H \pp{\Psi}{z} 
  e^{i k_H z} = 0,
    \label{parax1}
  \end{align}
where $\nabla_t^2 = \pp{^2}{x^2} + \pp{^2}{y^2}$. Since we assumed that 
$\Psi$ varies slowly with respect to $z$, $\abs{\pp{^2}{z^2} \Psi} \ll \abs{2 i k_H
\pp{}{z} \Psi}$. Thus, equation~(\ref{parax1}) becomes 
  \begin{align}
\nabla_t^2 \Psi + 2i k_H \pp{}{z} \Psi = 0,
    \label{parax2}
  \end{align}
which is an equation describing the transverse profile of a beam propagating
along the $z$-axis. Equation~(\ref{parax2}) is called the paraxial
wave equation.

\begin{remark}
We note that (\ref{parax2}) is just Schr\"odinger's equation, where $z$
represents time.
\end{remark}

\subsubsection{Airy beams}

Separating variables, we write the solution $\Psi$ to the paraxial wave
equation~(\ref{parax2}) as
  \begin{align}
\Psi(x,y,z) = \Phi_x(x,z) \Phi_y(y,z).
  \end{align}
From~(\ref{parax2}), we obtain
  \begin{align}
\pp{^2}{x^2} \Phi_x + 2i k_H \pp{}{z}\Phi_x &= 0, \\
\pp{^2}{y^2} \Phi_y + 2i k_H \pp{}{z}\Phi_y &= 0.
  \end{align}
Letting $x_0$ and $y_0$ be arbitrary transverse scaling factors, and setting
  \begin{align}
s_x = \frac{x}{x_0}, \quad s_y=\frac{y}{y_0}, \quad \xi_x=\frac{z}{k_H x_0^2},
\quad \xi_y = \frac{z}{k_H y_0^2},
  \end{align}
we have the equations
  \begin{align}
\frac{1}{2} \pp{^2}{s_x^2} \Phi_x(s_x,\xi_x) + i \pp{}{\xi_x}
\Phi_x(s_x,\xi_x) &= 0, \label{paraxx} \\
\frac{1}{2} \pp{^2}{s_y^2} \Phi_y(s_y,\xi_y) + i \pp{}{\xi_y}
\Phi_x(s_y,\xi_y) &= 0. \label{paraxy}
  \end{align}
One particular solution to~(\ref{paraxx}) is given by the formula
  \begin{align}
\Phi_x(s_x,\xi_x) =\ai\Bigl(s_x - \Bigl(\frac{\xi_x}{2}\Bigr)^2\Bigr) 
\exp\Bigl(i \Bigl(-\frac{\xi_x^3}{12}
  + s_x \frac{\xi_x}{2}\Bigr)\Bigr).
    \label{airybeam}
  \end{align}
Note that $\Phi_x(s_x,0)= \ai(s_x)$.  An identical solution exists for
$\Phi_y$, but for the sake of simplicity we take $\Phi_y \equiv 1$, and
denote $s_x$ and $\xi_x$ by $s$ and $\xi$.  Beams $\Psi$ for which $\Phi_x$
is given by~(\ref{airybeam}) and $\Phi_y\equiv 1$ are called Airy-Plane
beams (see, for example,~\S3.1.5 of~\cite{pio61}).  The transverse profile of the
Airy beam is invariant in the $\xi$-direction in the unusual sense that the
profile does not change, except that it is translated in the $s$-direction
by $(\xi/2)^2$. Thus, the Airy beam is non-diffracting, and is
self-accelerating due to its translation. This seemingly paradoxical
phenomenon (recall that the center of mass of the profile of a beam must
remain invariant with respect to $\xi$ in the absence of external fields) is
explained by the fact that the energy of the Airy beam is infinite, since
$\int_{-\infty}^\infty \ai(x)^2 \d x =\infty$, and so the center of mass of
the beam is undefined.

\subsubsection{Airy eigenfunction beams}

While the Airy beam~(\ref{airybeam}) is perfectly non-diffracting and
self-accelerating, its energy is infinite. Since such a beam is not
realizable, it would be desirable to construct a beam exhibiting the same
properties, but with finite energy.

One well-known solution is the finite Airy beam (see, for
example,~\cite{siviloglou,siviloglou2,jiang}), which is generated by
introducing an exponential aperture function to the initial field envelope of
the Airy beam, i.e., 
  \begin{align}
\Phi(s,0)=(8\pi \alpha)^{1/4} \ai(s)\exp\Bigl(-\frac{\alpha^3}{3}+\alpha s\Bigr),\label{for:finiteic}
  \end{align}
where $\alpha>0$. Note that, for simplicity, the initial field envelope has been
normalized such that $\norm{\Phi(s,0)}_2=1$. Solving equation
(\ref{paraxx}) under the initial condition (\ref{for:finiteic}), we have
that the beam evolves according to
  \begin{align}
\hspace*{-5em}\Phi(s,\xi)=(8\pi \alpha)^{1/4}
\ai\Bigl(s-\Bigl(\frac{\xi}{2}\Bigr)^2+i\alpha\xi\Bigr)\exp\Bigl(-\frac{\alpha^3}{3}+\alpha s
-\frac{\alpha\xi^2}{2}+i\Bigl(-\frac{\xi^3}{12}+\frac{\alpha^2\xi}{2}+\frac{s\xi}{2}\Bigr)
\Bigr).\label{for:finitepro}
  \end{align}
Although these beams only have finite energy, it has been shown both
theoretically and experimentally that, when $\alpha$ is small, the finite
Airy beams exhibit the key characteristics of the Airy beam, i.e., the
ability to remain diffraction-free over long distances, and to freely
accelerate during propagation. To be more specific,
as $\alpha \to 0$, the resulting beam $\Phi$ approaches a scaled Airy
function. When $\alpha$ gets bigger, on the other hand, the beam $\Phi$
resembles a Gaussian. The beam profiles for several values of $\alpha$ are
illustrated in Figures~\ref{fig:cm2_a0202} and~\ref{fig:cm1_a0108}. 

Below, we show that the Airy transform of the eigenfunctions of the Airy
integral operator, which also have finite energy, resemble the
infinite-energy Airy beam in a different way, in that they maximally
concentrate the energy near the main lobes in their initial profiles, while remaining
diffraction-free over the longest possible distances. We note that the
eigenfunction beams achieve their long diffraction-free distances by
spreading their energy as evenly as possible in their side lobes.  For
simplicity, we name them Airy eigenfunction beams.

It is not hard to see that, for any density function $\sigma$, the beam with
transverse profile
  \begin{align}
\Phi(s,\xi) = \int_0^\infty \sigma(v) 
\ai\Bigl(s+v - \Bigl(\frac{\xi}{2}\Bigr)^2\Bigr) 
\exp\Bigl(i \Bigl(-\frac{\xi^3}{12}
  + (s + v)\frac{\xi}{2}\Bigr)\Bigr)
\d v
    \label{paraxsoln}
  \end{align}
is a solution to~(\ref{paraxx}), since~(\ref{paraxsoln}) can be
differentiated under the integral sign due to the rapid decay of $\ai(v)$ as
$v\to \infty$. Note that, when $\xi=0$,
  \begin{align}
\Phi(s,0) = \int_0^\infty \sigma(v) 
\ai(s+v) \d v.
    \label{paraxsoln0}
  \end{align}
When $\sigma$ is a delta function, the beam $\Phi$ is perfectly
non-diffracting, since then it is just an Airy function. When $\sigma$ is
supported over some interval of positive width, however, the beam will
diffract due to interference between different modes. This diffraction is
caused by the term $\exp(iv\xi/2)$ in~(\ref{paraxsoln}), without which the
beam would be perfectly non-diffracting for any $\sigma$.  If the goal is to
construct a non-diffracting and self-accelerating beam, then the support of
$\sigma$ should be as small as possible, so that the beam resembles the Airy
beam as much as possible.
However, when $\sigma$ is highly concentrated around $v=0$, the energy in
$\Phi$ will be very poorly localized, resulting in an overall weak beam
intensity. This trade-off between the localization of $\Phi$ and
the localization of $\sigma$ is a result of the uncertainty principle
described in Section~\ref{sec:uncert}. Consequently, the extremal property
of the eigenfunction $\psi_{0,c}$ can be utilized to optimize the localization
of both the beam intensity $\Phi$ and the density $\sigma$. To be more specific,
we let $\sigma(v)=\psi_{0,c}(v)$ for some real number $c$, such that
formula (\ref{paraxsoln0}) becomes
  \begin{align}
\Phi(s,0) = \int_0^\infty \ai(s+v)\psi_{0,c}(v) \d v=\cA[\psi_{0,c}](s).\label{for:initpro}
  \end{align}
Based on Section \ref{sec:qual}, when $c\in[-5,1.5]$, the resulting Airy
eigenfunction beam concentrates the most energy in $[c,\infty)$, while
remaining Airy-bandlimited.

The densities and corresponding beam profiles for
several values of $c$ are illustrated in Figures~\ref{fig:cm2_a0202},
~\ref{fig:cm1_a0108} and \ref{fig:sigphieig}.

\section{Numerical Experiments}
In this section, we illustrate the performance of the algorithm with several
numerical examples. 

We implemented our algorithm in FORTRAN 77, and compiled it using
Lahey/Fujitsu Fortran 95 Express, Release L6.20e. For the timing
experiments, the Fortran codes were compiled using the Intel Fortran
Compiler, version 2021.2.0, with the \texttt{-fast} flag. We conducted all
experiments on a ThinkPad laptop, with 16GB of RAM and an Intel Core
i7-10510U CPU. 

\subsection{Computation of the eigenfunctions and spectra}\label{sec:eigplot} 
In this section, we report the plots of the eigenfunctions and spectra for
different values of $c$ and $n$ in Figures
\ref{figeigfun}--\ref{fig:lambda0}, and the corresponding computation times
in Table \ref{tablecomptime}. We normalize the eigenfunctions $\psi_{n,c}$
by requiring that $\psi_{n,c}(0)>0$ (recall that $\psi_{n,c}(0)\neq 0$, see
(\ref{rec4})). In addition, we illustrate the importance of selecting the
optimal scaling factor of the scaled Laguerre functions in Figure
\ref{figa}.
\begin{figure}[h]
    \centering
    \begin{subfigure}{0.49\textwidth}
      \centering
      \includegraphics[width=\textwidth]{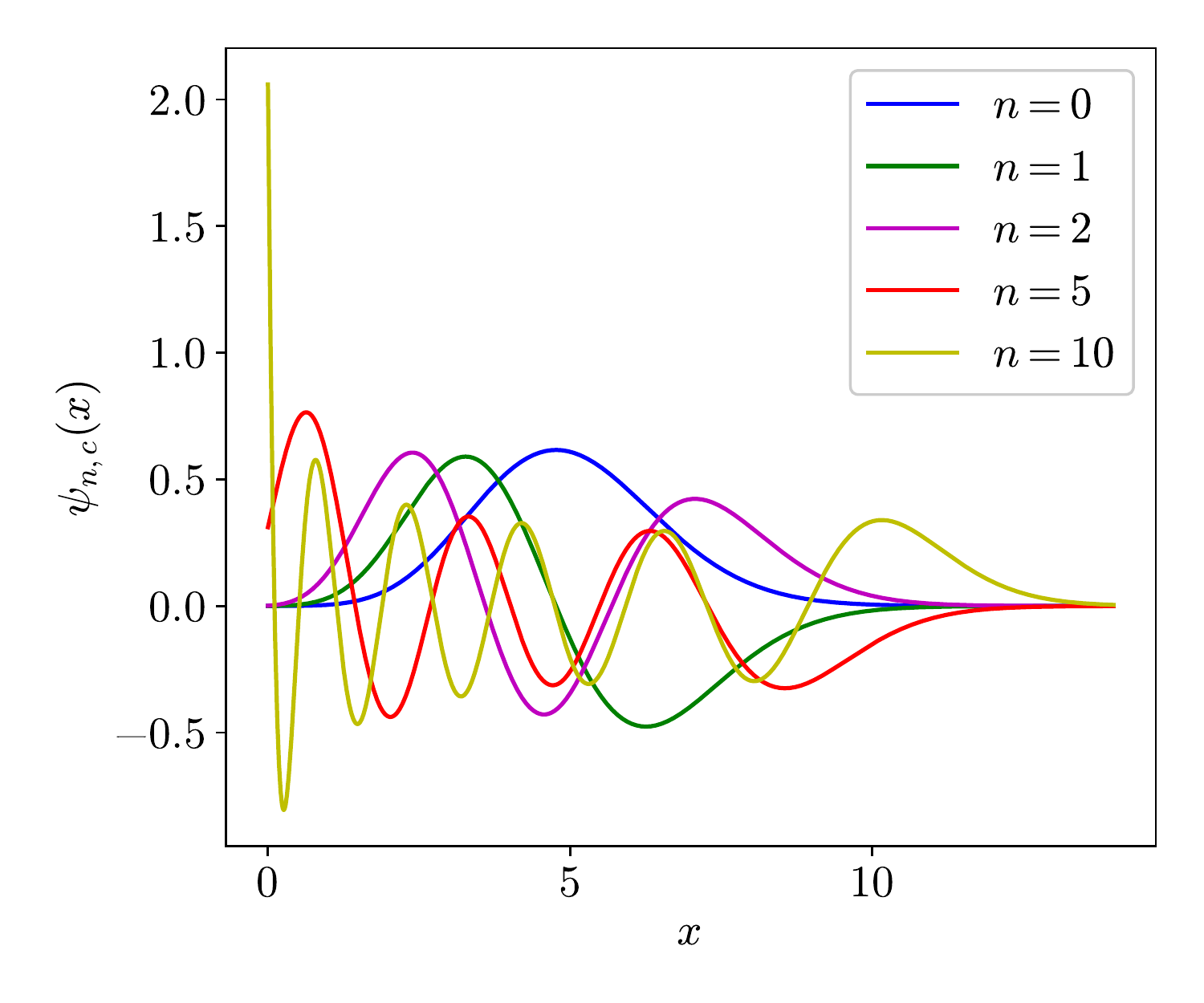}
      \caption{$c=-10$ \label{fig1}}
    \end{subfigure}
    \begin{subfigure}{0.49\textwidth}
      \centering
      \includegraphics[width=\textwidth]{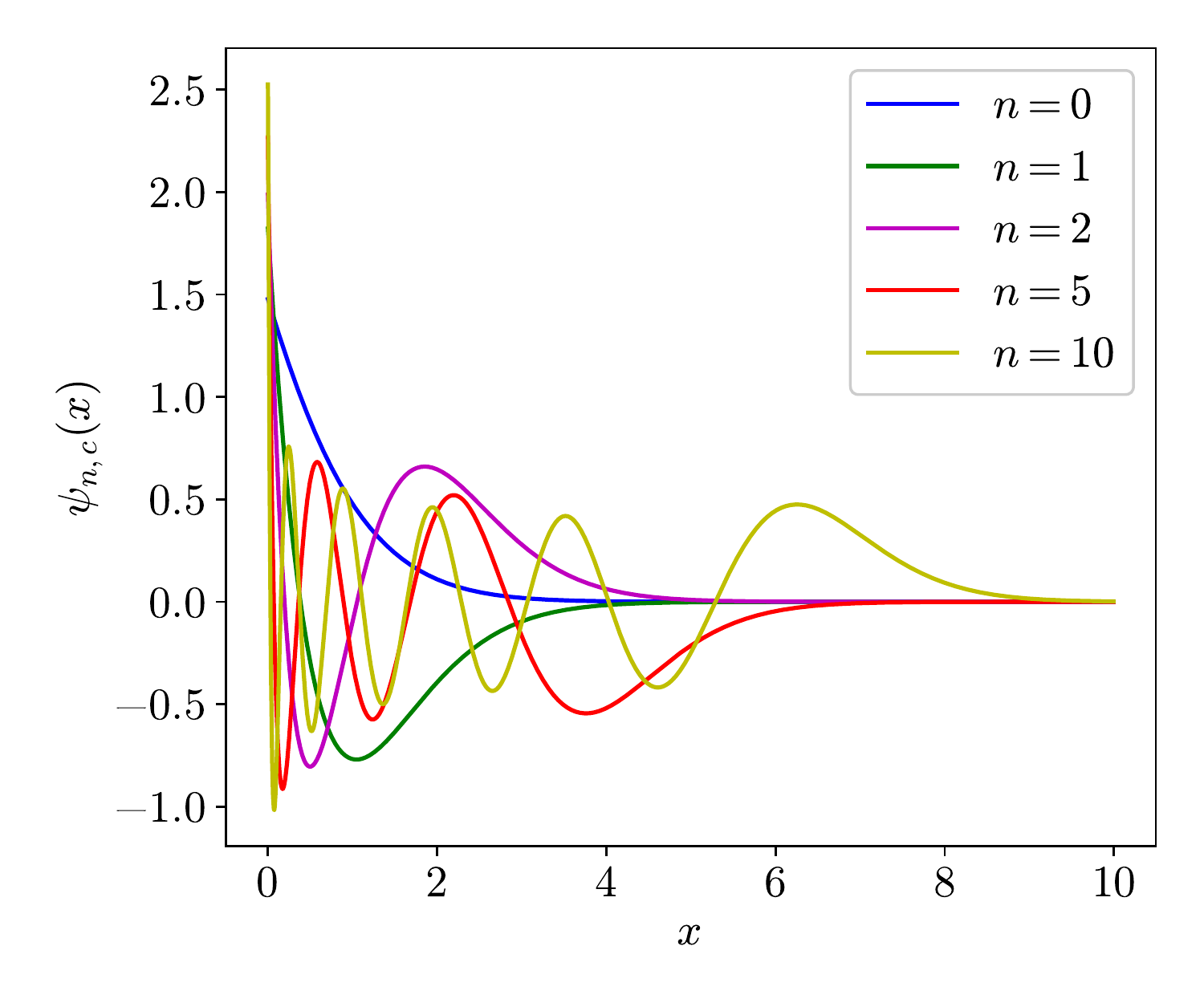}
      \caption{$c=0$ \label{fig2}}
    \end{subfigure}
  \begin{subfigure}{0.49\textwidth}
    \centering
    \includegraphics[width=\textwidth]{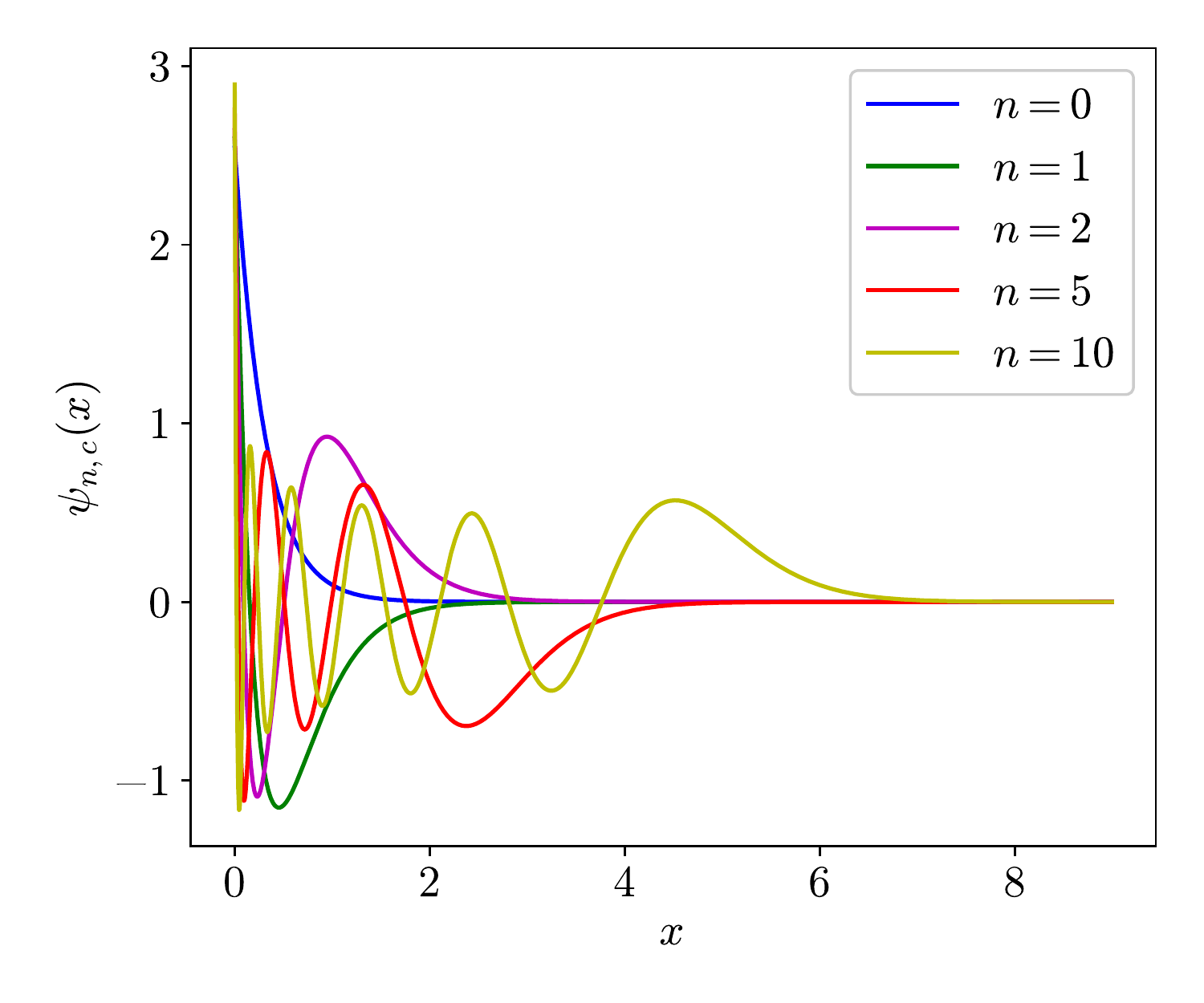}
    \caption{$c=10$ \label{fig3}}
  \end{subfigure}
  \begin{subfigure}{0.49\textwidth}
    \centering
    \includegraphics[width=\textwidth]{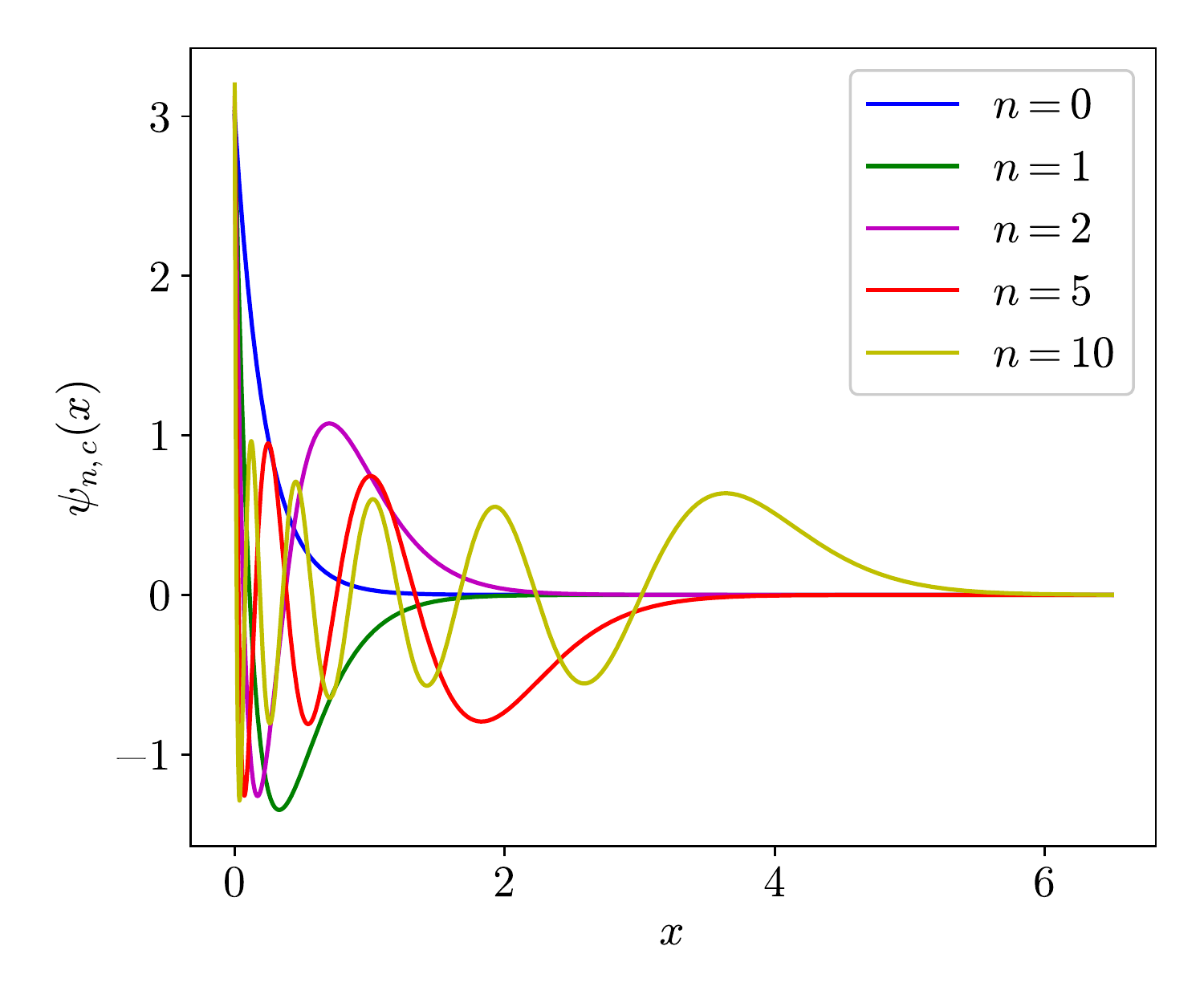}
    \caption{$c=20$ \label{fig4}}
  \end{subfigure}
  \caption{{\bf Eigenfunctions $\psi_{n,c}$, defined by (\ref{Teigfun}), of
  different orders with different parameters $c$}.}
  \label{figeigfun}

\end{figure}

\begin{figure}[h]
    \centering
    \includegraphics[width=0.75\textwidth]{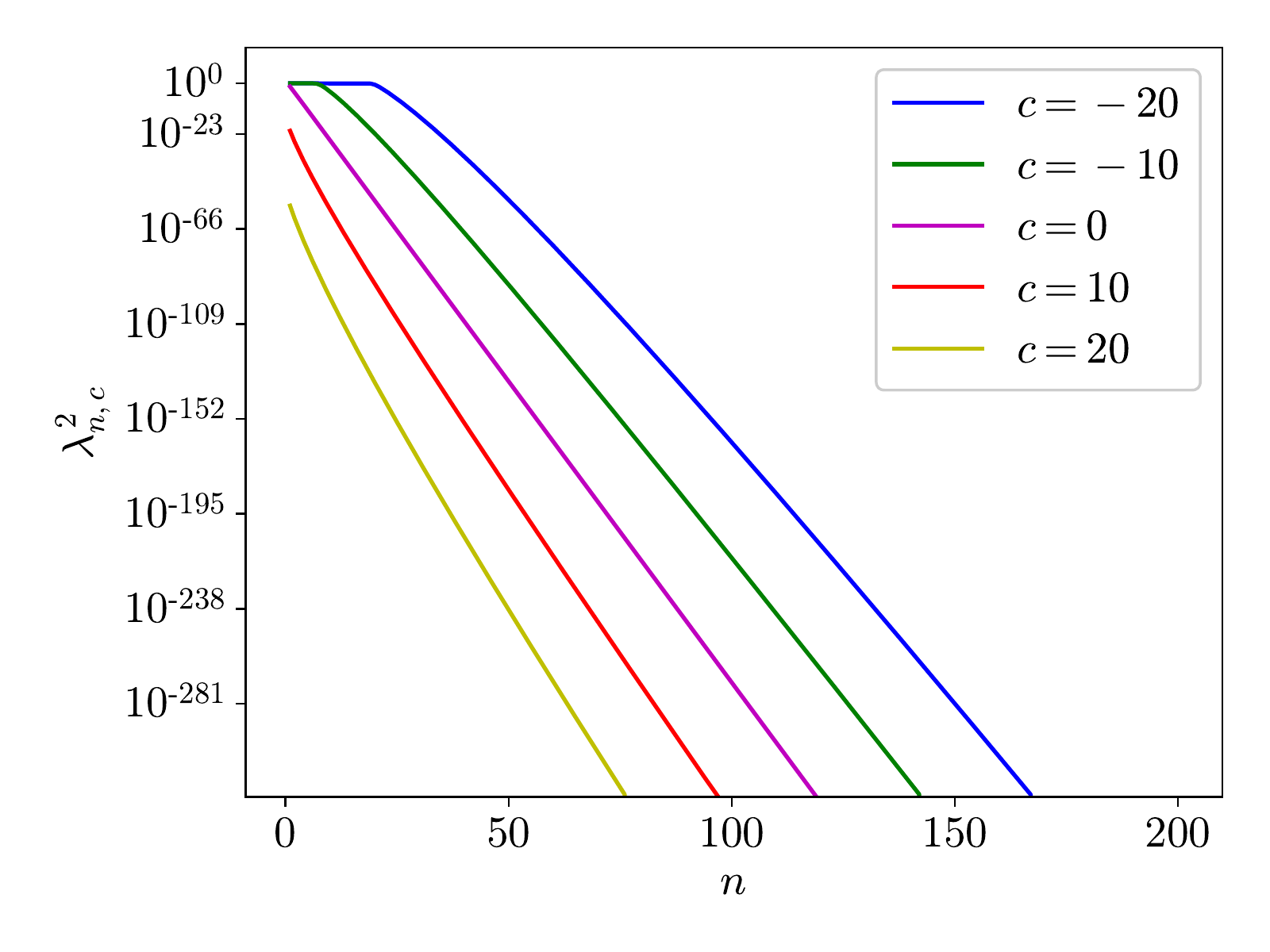}

  \caption{
      {\bf Squares of the spectra of the Airy integral operators $\T_c$ with
      different parameters $c$}. This corresponds to the spectra of the
      integral operator $\cK$ (see formulas
      (\ref{f2kai}), (\ref{KGeq})). Note that the square of the leading
      eigenvalues converge to $1$ as $c\to -\infty$.}
   \label{figspctm}

\end{figure}

\begin{figure}[h]
    \centering
    \includegraphics[width=0.75\textwidth]{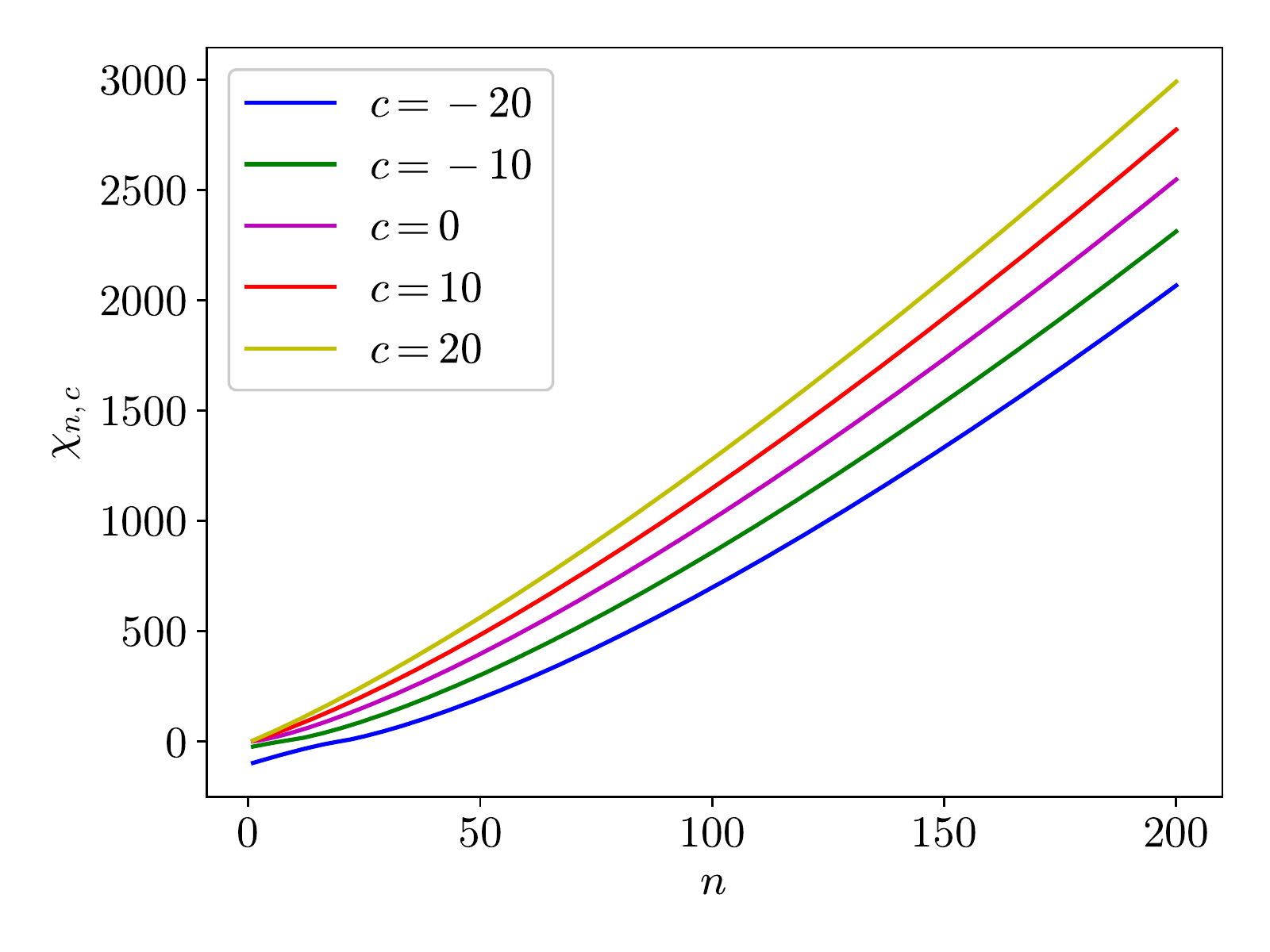}

  \caption{
      {\bf The spectra of the commuting differential operators $\L_c$ with
      different parameters $c$}. Note the presence of negative eigenvalues
      for sufficiently negative values of $c$. }
   \label{figspctm_chi}

\end{figure}

\begin{figure}[h]
    \centering
    \includegraphics[width=0.75\textwidth]{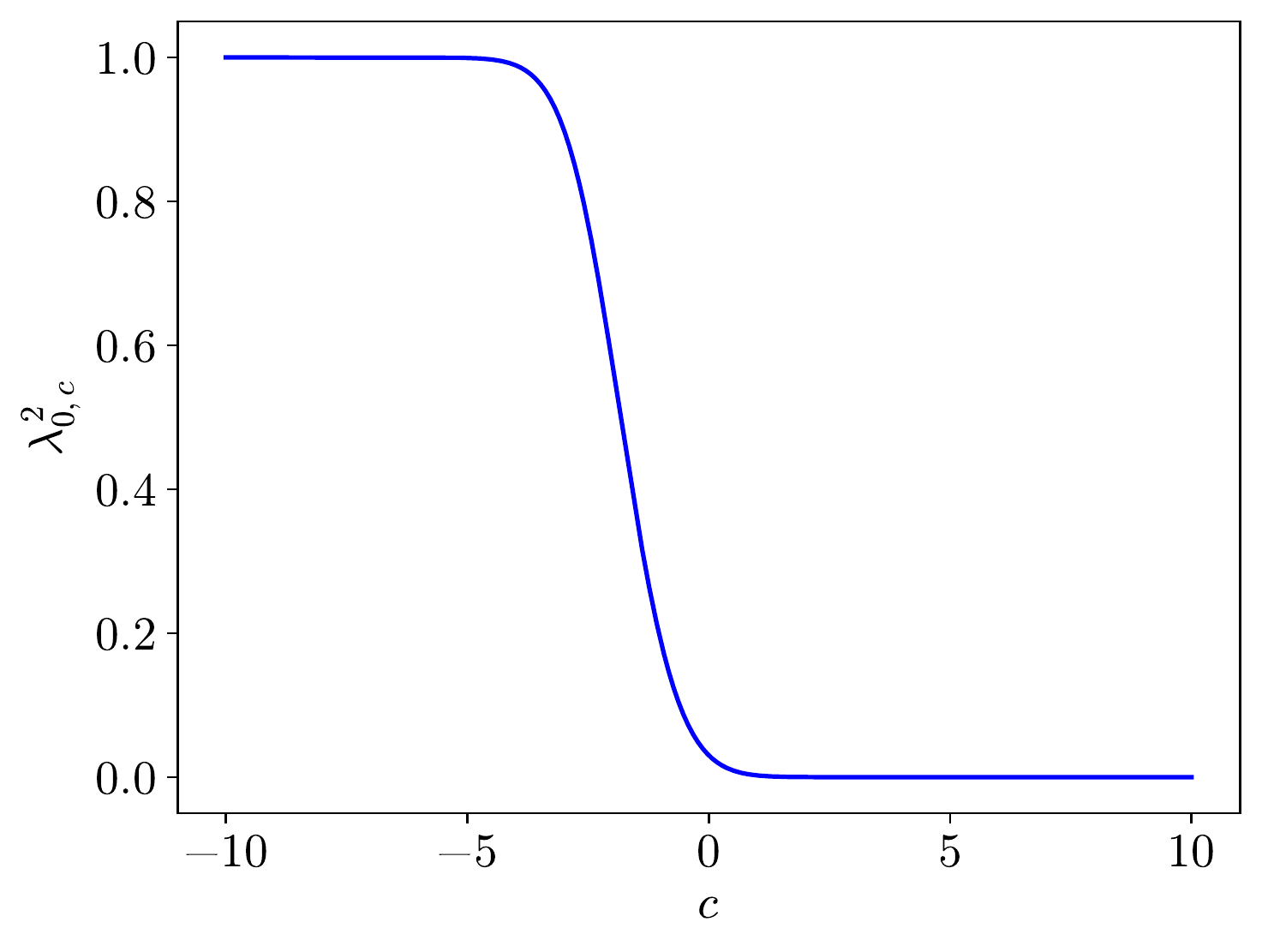}

  \caption{
      {\bf Squares of the first eigenvalues $\lambda_{0,c}$ of the Airy
      integral operators $\T_c$ with different parameters $c$}. Note that
      $\lambda_{0,c}^2$ is equal to the maximal proportion of energy
      a function that is Airy-bandlimited to $[0,\infty)$ can have 
      on $[c,\infty)$ (see Theorems \ref{thm:uncert} and \ref{thm:extremal}).}
   \label{fig:lambda0}

\end{figure}

\begin{table}[h!!]
    \begin{center}
    \begin{tabular}{cccc}
    $c$& $n$ & $N$ & Time  \\

    \midrule
    \addlinespace[.5em]
    $20$  & 50 & 175 & 2.10\e{-3} secs  \\
    \addlinespace[.25em]
      & 100 & 230 & 3.64\e{-3} secs  \\
    \addlinespace[.25em]
      & 200 & 340 & 8.76\e{-3} secs  \\
    \addlinespace[.25em]
      & 400 & 560 & 3.35\e{-2} secs  \\
    \addlinespace[.25em]
    $0$  & 50 & 155 & 3.36\e{-3} secs  \\
    \addlinespace[.25em]
    & 100 & 210 & 4.93\e{-3} secs  \\
    \addlinespace[.25em]
     & 200 & 320 & 9.75\e{-3} secs  \\
    \addlinespace[.25em]
      & 400 & 540 & 3.17\e{-2} secs  \\
    \addlinespace[.25em]
    $-20$  & 50 & 175 & 4.32\e{-3} secs  \\
    \addlinespace[.25em]
     & 100 & 230 & 5.64\e{-3} secs  \\
    \addlinespace[.25em]
      & 200 & 340 & 1.14\e{-2} secs  \\
    \addlinespace[.25em]
      & 400 & 560 & 3.37\e{-2} secs  \\
    \end{tabular}
    \caption{
    {\bf The computation time of the eigenfunctions and spectra of the
    integral operator for different value of $c$ and $n$}. The value of $N$
    is determined by Observation \ref{obsN}. The time cost is of order
    $\O(N^2)$.
    }
    \label{tablecomptime}

    \end{center}
\end{table}

\begin{figure}[h]
    \centering
    \begin{subfigure}{0.49\textwidth}
      \centering
      \includegraphics[width=\textwidth]{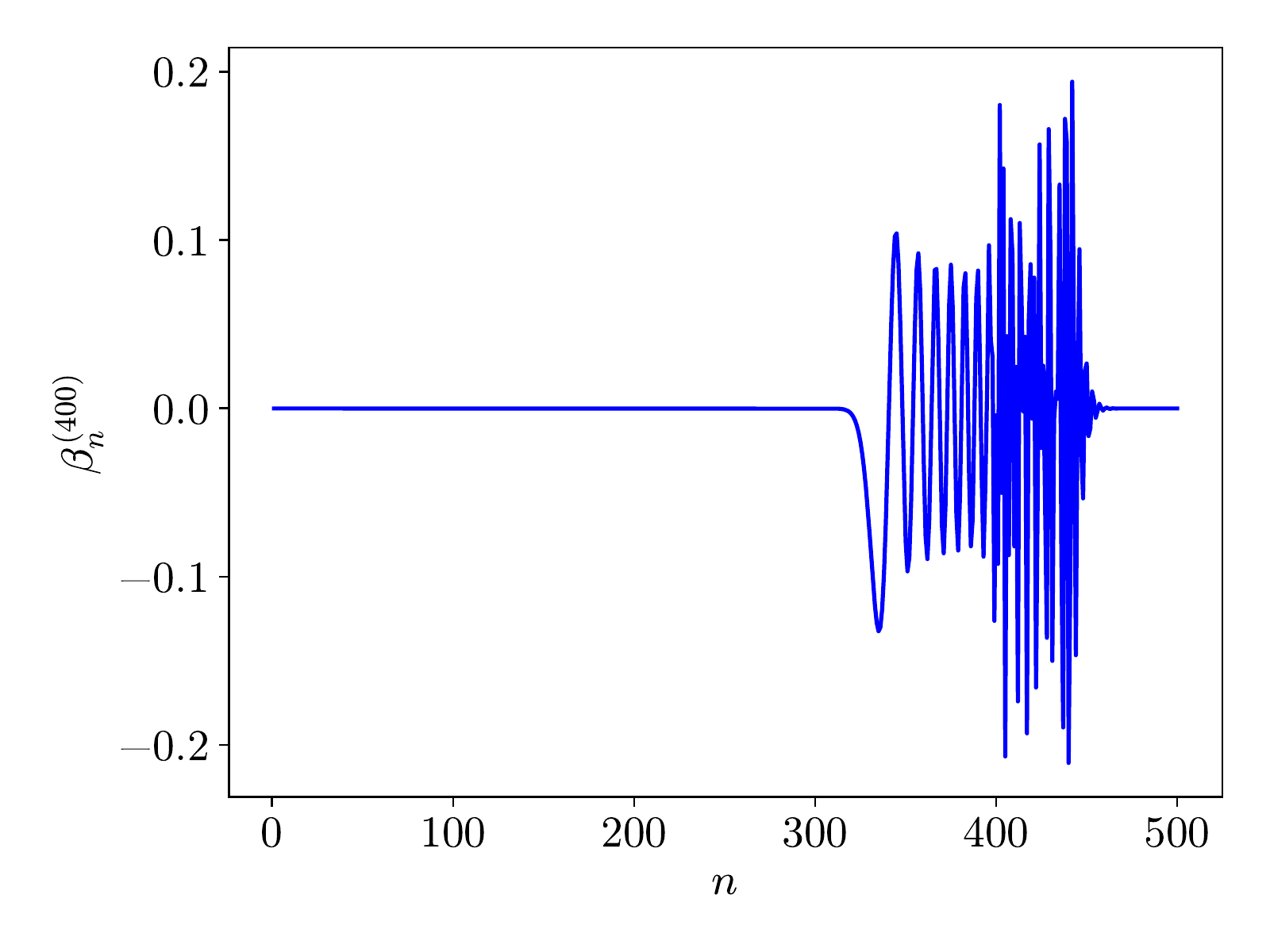}
      \caption{$a=\tilde a$ \label{fig11}}
    \end{subfigure}
    \begin{subfigure}{0.49\textwidth}
      \centering
      \includegraphics[width=\textwidth]{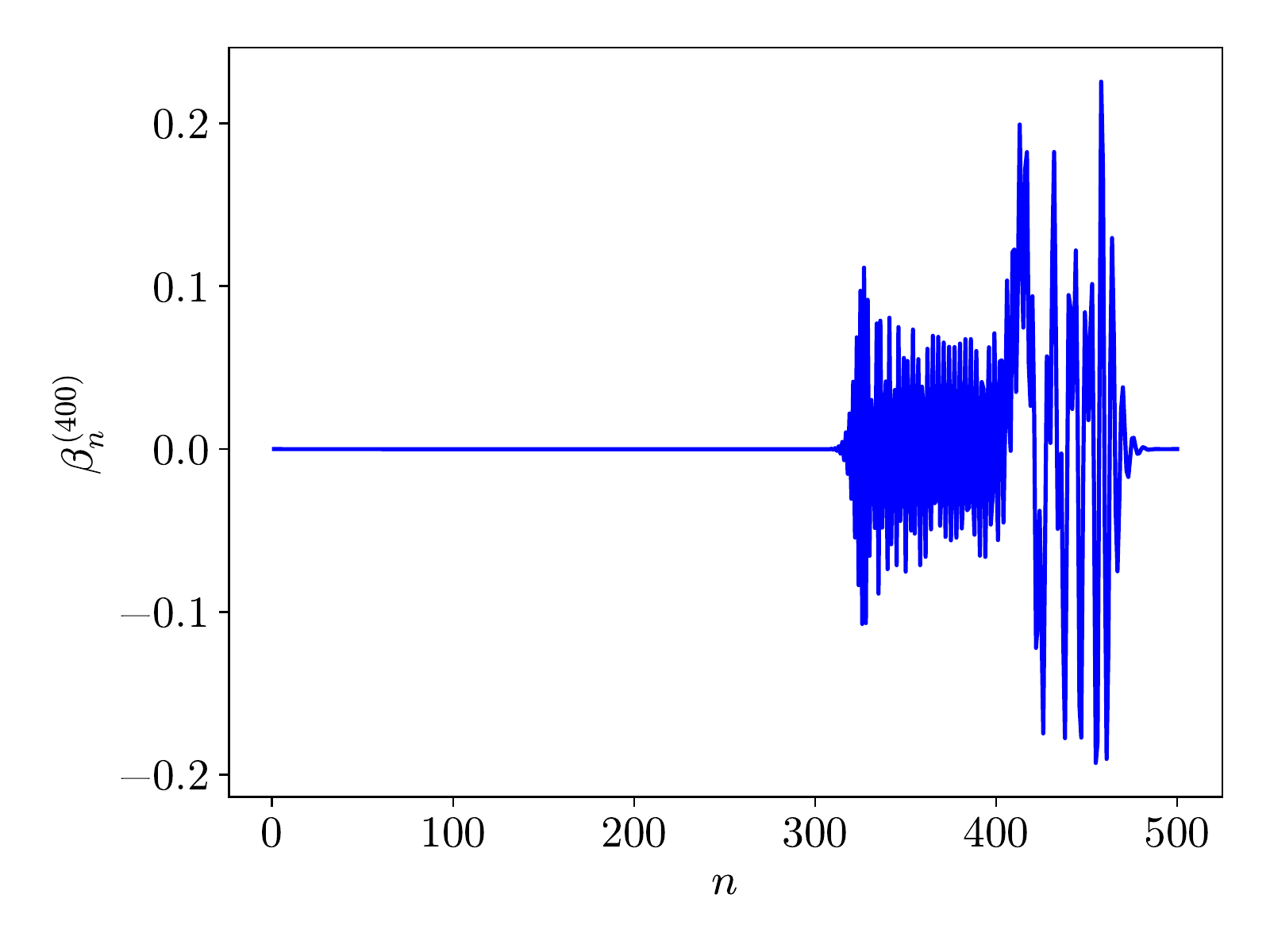}
      \caption{$a=\tilde a-4$ \label{fig22}}
    \end{subfigure}
  \begin{subfigure}{0.49\textwidth}
    \centering
    \includegraphics[width=\textwidth]{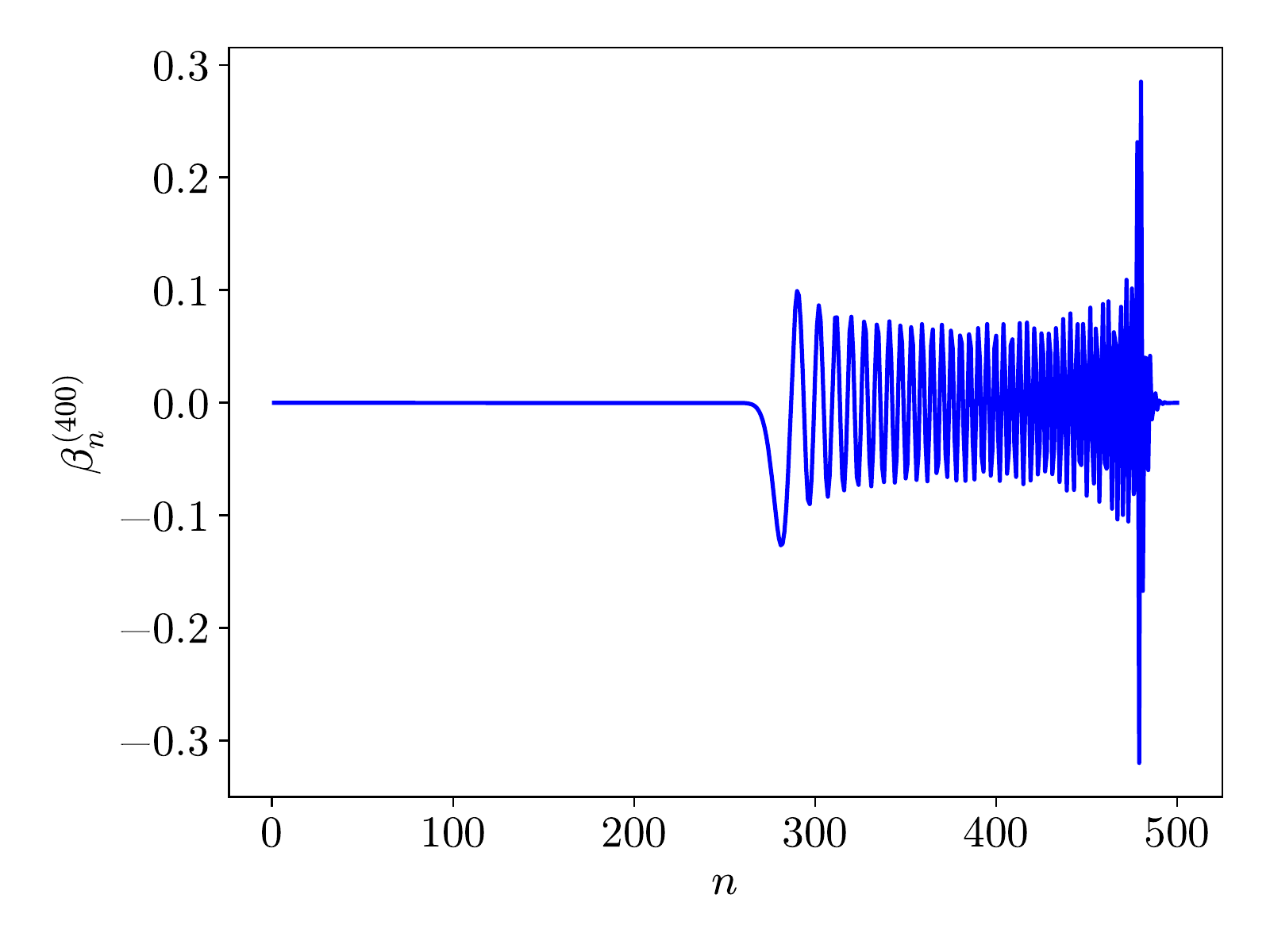}
    \caption{$a=\tilde a+4$ \label{fig33}}
  \end{subfigure}
  \begin{subfigure}{0.49\textwidth}
    \centering
    \includegraphics[width=\textwidth]{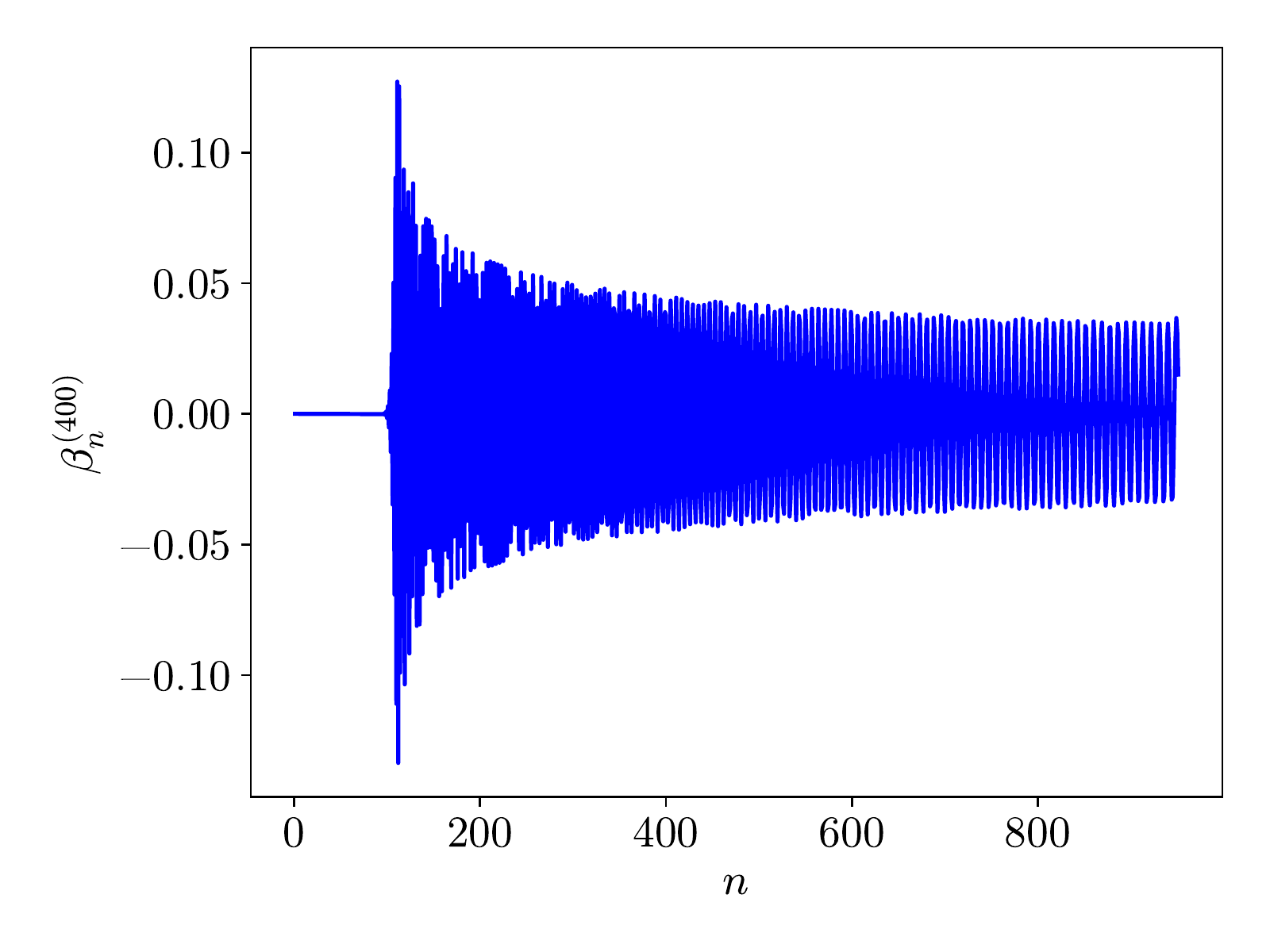}
    \caption{$a=1$ \label{fig44}}
  \end{subfigure}
  \caption{{\bf Expansion coefficients of $\psi_{n,c}$ in the basis of
  scaled Laguerre functions with different scaling factors $a$, where
  $n=400$ and $c=10$}. Note that the optimal scaling factor $\tilde a\approx
  20.62$, and is selected by formula (\ref{choicea}). It's clear that our
  basis functions become optimal when $a=\tilde a$. Figure~(\subref{fig44})
  shows that the unscaled Laguerre functions are unsuitable for
  approximating the eigenfunctions of the Airy integral operator.}
  \label{figa}

\end{figure}
\clearpage

\subsection{Computation of the distribution of the $k$-th largest
eigenvalue of the Gaussian unitary ensemble}
  \label{sec:computtw}
In this section, we report the computation time and the numerical errors of
the PDF $\frac{d}{ds}F_2(k;s)$ and CDF $F_2(k;s)$ in Tables \ref{pdfexp} and
\ref{cdfexp}, for different values of $k$ and $s$. The reference solutions
are computed by our solver using extended precision. We note that Pr\"ahofer
tabulated the values of $F_2(1;s)$ and $\log\frac{d}{ds}F_2(1;s)$ 
to 16 digits of relative accuracy in \cite{table}, and our computed values
match with the values reported there. We also show the plots of
$\frac{d}{ds}F_2(k;s)$ and $F_2(k;s)$ for $k=1,2,3$ in Figure \ref{twplot}.
\begin{table}
    \begin{center}
    \begin{tabular}{cccccccc}
    \multirow{2}{*}{$k$} &\multirow{2}{*}{$s$}&  \multirow{2}{*}{$n$} & \multirow{2}{*}{$N$} & \multirow{2}{*}{Time} & Relative & Absolute & \multirow{2}{*}{$\frac{d}{ds}F_2(k;s)$} \\
    & & &   &  & error & error&\\
    \midrule
    \addlinespace[.5em]
    1 & $50$  & 30 & 50 & 1.70\e{-4} secs & 2.53\e{-14} & 3.76\e{-222} & 1.48437\e{-208}\\
    \addlinespace[.25em]
    & $25$  & 20  & 40  & 1.20\e{-4} secs & 1.16\e{-14} & 7.61\e{-90} & 6.56096\e{-76}\\
    \addlinespace[.25em]
    & $10$  & 20  &  40 & 1.25\e{-4} secs & 2.18\e{-15} & 4.14\e{-36}&1.90064\e{-21}\\
    \addlinespace[.25em]
    & $5$  & 20  & 40 & 1.36\e{-4} secs & 3.28\e{-16} & 8.27\e{-25}&2.52106\e{-9} \\
    \addlinespace[.25em]
    & $2$  & 20  & 40 & 1.44\e{-4} secs & 4.15\e{-15} & 1.57\e{-18} & 3.79199\e{-4} \\
    \addlinespace[.25em]
    & $0$  & 20  & 40 & 1.50\e{-4} secs & 2.07\e{-16} & 1.39\e{-17} & 6.69753\e{-2}\\
    \addlinespace[.25em]
    & $-2$  & 20  & 40 & 1.49\e{-4} secs & 5.03\e{-16} & 2.22\e{-16} & 4.41382\e{-1}\\
    \addlinespace[.25em]
    & $-5$  & 40 &60  & 2.96\e{-4} secs & 7.25\e{-13} & 9.71\e{-17} & 1.34039\e{-4} \\
    \addlinespace[.25em]
    & $-10$  & 50 & 80  & 4.39\e{-4} secs & 2.53\e{-4} & 2.66\e{-39} & 1.05359\e{-35} \\
    \addlinespace[.25em]
    & $-20$  & 70  & 120 & 9.45\e{-4} secs & 8.50\e{87} & 1.50\e{-200} & 1.77193\e{-288}\\
    \addlinespace[.25em]
    2 & $30$  & 30 & 50 & 1.79\e{-4} secs & 3.23\e{-14} & 2.85\e{-217} & 8.88120\e{-204}\\
    \addlinespace[.25em]
     & $0$  & 20 & 40 & 1.50\e{-4} secs & 1.11\e{-15} & 1.36\e{-20} & 1.21766\e{-5}\\
    \addlinespace[.25em]
     & $-4$  & 30 & 50 & 2.23\e{-4} secs & 3.08\e{-15} & 1.55\e{-15} & 5.05206\e{-1}\\
    \addlinespace[.25em]
     & $-6$  & 50 & 80 & 4.88\e{-4} secs & 1.43\e{-13} & 3.02\e{-16} & 2.10626\e{-3}\\
    \addlinespace[.25em]
     & $-10$  & 50 & 100 & 7.38\e{-4} secs & 1.35\e{-6} & 2.27\e{-30} & 1.67893\e{-24}\\
    \addlinespace[.25em]
     & $-12$  & 60 & 120 & 1.07\e{-3} secs & 2.80\e{-2} & 3.20\e{-48} & 1.14082\e{-46}\\
    \addlinespace[.25em]
    3 & $15$  & 30 & 50 & 2.46\e{-4} secs & 4.10\e{-15} & 1.02\e{-140} & 2.48166\e{-126}\\
    \addlinespace[.25em]
     & $4$  & 20 & 40 & 2.11\e{-4} secs & 1.50\e{-15} & 8.21\e{-48} & 5.50657\e{-33}\\
    \addlinespace[.25em]
     & $-4$  & 30 & 50 & 3.03\e{-4} secs & 1.15\e{-14} & 1.44\e{-15} & 1.25051\e{-1}\\
    \addlinespace[.25em]
     & $-8$  & 50 & 80 & 5.69\e{-4} secs & 5.81\e{-12} & 1.03\e{-16} & 1.76988\e{-5}\\
    \addlinespace[.25em]
     & $-10$  & 50 & 100 & 8.19\e{-4} secs & 1.07\e{-8} & 8.56\e{-24} & 8.01983\e{-16}\\
    \addlinespace[.25em]
     & $-13$  & 60 & 120 & 1.27\e{-3} secs & 1.61\e{-2} & 1.55\e{-48} & 9.63884\e{-47}\\
    \end{tabular}
    \caption{
    {\bf The evaluation of the probability density functions}. The actual
    values of $\frac{d}{ds}F_2(k;s)$ are reported to 6 significant digits.
    Note that the relative accuracy degenerates when one evaluates
    $\frac{d}{ds}F_2(k;s)$ in the left tails of the distributions (see
    Section \ref{sec:klevel}).
    }
    \label{pdfexp}
    \end{center}
\end{table}

\begin{table}[h!!]
    \begin{center}
    \begin{tabular}{cccccccc}
    \multirow{2}{*}{$k$} &\multirow{2}{*}{$s$}&  \multirow{2}{*}{$n$} & \multirow{2}{*}{$N$} & \multirow{2}{*}{Time} & Relative & Absolute & \multirow{2}{*}{$F_2(k;s)$} \\
    & & &   &  & error & error&\\
    \midrule
    \addlinespace[.5em]
    $1$ & $50$  & 30 & 50 & 1.70\e{-4} secs & $<$1.00\e{-16} & $<$1.00\e{-16} & 1.00000\e{0} \\
    \addlinespace[.25em]
    & $25$  & 20  & 40  & 1.20\e{-4} secs & $<$1.00\e{-16} & $<$1.00\e{-16} & 1.00000\e{0} \\
    \addlinespace[.25em]
    & $10$  & 20  & 40 & 1.25\e{-4} secs & $<$1.00\e{-16} & $<$1.00\e{-16} & 1.00000\e{0} \\
    \addlinespace[.25em]
    & $5$  & 20  & 40 & 1.36\e{-4} secs & $<$1.00\e{-16} & $<$1.00\e{-16}& 1.00000\e{0} \\
    \addlinespace[.25em]
    & $2$  & 20  & 40 & 1.44\e{-4} secs & 1.11\e{-16} & 1.11\e{-16} & 9.99888\e{-1} \\
    \addlinespace[.25em]
    & $0$  & 20  & 40 & 1.50\e{-4} secs & 1.15\e{-16} & 1.11\e{-16} & 9.69373\e{-1}\\
    \addlinespace[.25em]
    & $-2$  & 20  & 40 & 1.49\e{-4} secs & 4.03\e{-16} & 1.67\e{-16} & 4.41322\e{-1}\\
    \addlinespace[.25em]
    & $-5$  & 40 & 60  & 2.96\e{-4} secs & 1.39\e{-12} & 2.96\e{-17} & 2.13600\e{-5} \\
    \addlinespace[.25em]
    & $-10$  & 50 & 80  & 4.39\e{-4} secs & 3.07\e{-4} & 1.29\e{-40} & 4.21226\e{-37} \\
    \addlinespace[.25em]
    & $-20$  & 70  & 120 & 9.45\e{-4} secs & 1.80\e{88} & 3.19\e{-202} & 1.77182\e{-290}\\
    \addlinespace[.25em]
    2 & $30$  & 30 & 50 & 1.79\e{-4} secs & $<$1.00\e{-16} & $<$1.00\e{-16} & 1.00000\e{0}\\
    \addlinespace[.25em]
     & $0$  & 20 & 40 & 1.50\e{-4} secs & 1.11\e{-16} & 1.11\e{-16} & 9.99998\e{-1}\\
    \addlinespace[.25em]
     & $-4$  & 30 & 50 & 2.23\e{-4} secs & 1.04\e{-14} & 3.50\e{-15} & 3.35602\e{-1}\\
    \addlinespace[.25em]
     & $-6$  & 50 & 80 & 4.88\e{-4} secs & 2.99\e{-13} & 1.10\e{-16} & 3.69221\e{-4}\\
    \addlinespace[.25em]
     & $-10$  & 50 & 100 & 7.38\e{-4} secs & 1.70\e{-6} & 1.38\e{-31} & 8.14202\e{-26}\\
    \addlinespace[.25em]
     & $-12$  & 60 & 120 & 1.07\e{-3}  secs & 3.30\e{-2} & 1.21\e{-49} & 3.65917\e{-48}\\
    \addlinespace[.25em]
    3 & $15$  & 30 & 50 & 2.46\e{-4}secs & $<$1.00\e{-16} & $<$1.00\e{-16} & 1.00000\e{0}\\
    \addlinespace[.25em]
     & $4$  & 20 & 40 & 2.11\e{-4} secs & $<$1.00\e{-16} & $<$1.00\e{-16} & 1.00000\e{0}\\
    \addlinespace[.25em]
     & $-4$  & 30 & 50 & 3.03\e{-4} secs & $<$1.00\e{-16} & $<$1.00\e{-16} & 9.59838\e{-1}\\
    \addlinespace[.25em]
     & $-8$  & 50 & 80 & 5.69\e{-4} secs & 9.70\e{-12} & 2.03\e{-17} & 2.09567\e{-6}\\
    \addlinespace[.25em]
     & $-10$  & 50 & 100 & 8.19\e{-4} secs & 1.42\e{-8} & 6.93\e{-25} & 4.89120\e{-17}\\
    \addlinespace[.25em]
     & $-13$  & 60 & 120 & 1.27\e{-3} secs & 1.89\e{-2} & 5.63\e{-50} & 2.98361\e{-48}
    \end{tabular}
    \caption{
    {\bf The evaluation of the cumulative distribution functions}. The
    actual values of $F_2(k;s)$ are reported to 6 significant digits.
    Note that the relative accuracy degenerates when one evaluates
    $F_2(k;s)$ in the left tails of the distributions (see
    Section \ref{sec:klevel}).
    }
    \label{cdfexp}
    \end{center}
\end{table}
\pagebreak
\begin{observation}
The computation times of $\frac{d}{ds}F_2(k;s)$ and $F_2(k;s)$ are dominated
by the calculation of the eigendecomposition of the Airy integral operator
$\T_s$ (see Observation \ref{fasttw}). As a consequence, the computation times of
$\frac{d}{ds}F_2(k;s)$ and $F_2(k;s)$ are almost identical for any fixed
$k$ and $s$ (see Tables \ref{pdfexp}, \ref{cdfexp}).
\end{observation}
\begin{observation}
Once the eigendecomposition of the Airy integral operator $\T_s$ is
computed, the time cost of evaluating $F_2(k;s)$ and $\frac{d}{ds}F_2(k;s)$
via formulas (\ref{f2k}) and (\ref{df2k}) for different $k$ is relatively
negligible (see Observation \ref{fasttw}). We note that we only consider the
evaluation of $\frac{d}{ds}F_2(k;s)$ and $F_2(k;s)$ for a single $k$ in our
experiments, which means that the reported times include the time required
for the computation of the eigendecomposition. 
\end{observation}
\begin{figure}[h]
    \begin{subfigure}{0.49\textwidth}
      \centering
      \hspace*{-2em}\includegraphics[width=\textwidth]{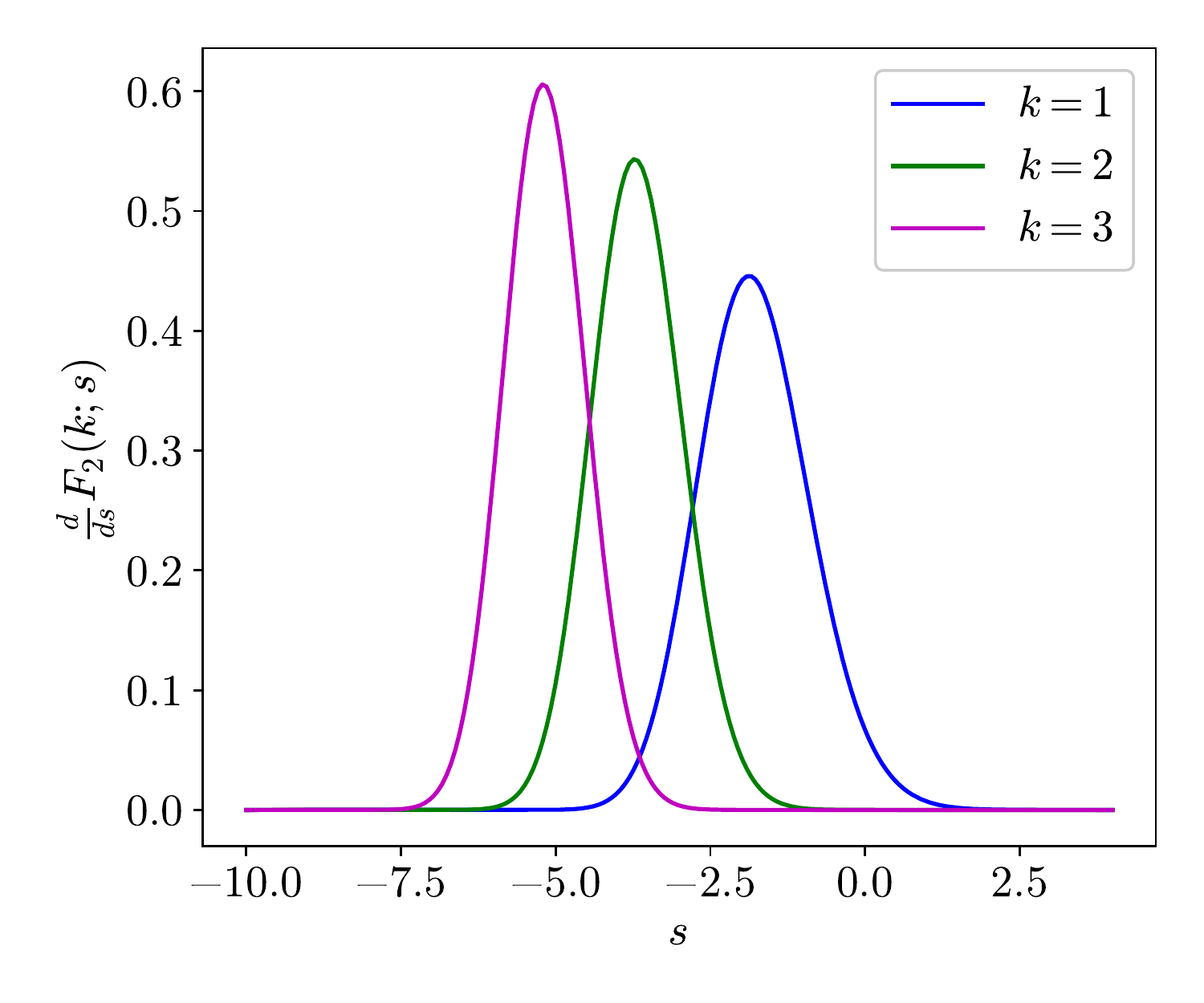}
    \end{subfigure}
    \begin{subfigure}{0.49\textwidth}
      \centering
      \hspace*{-2em}\includegraphics[width=\textwidth]{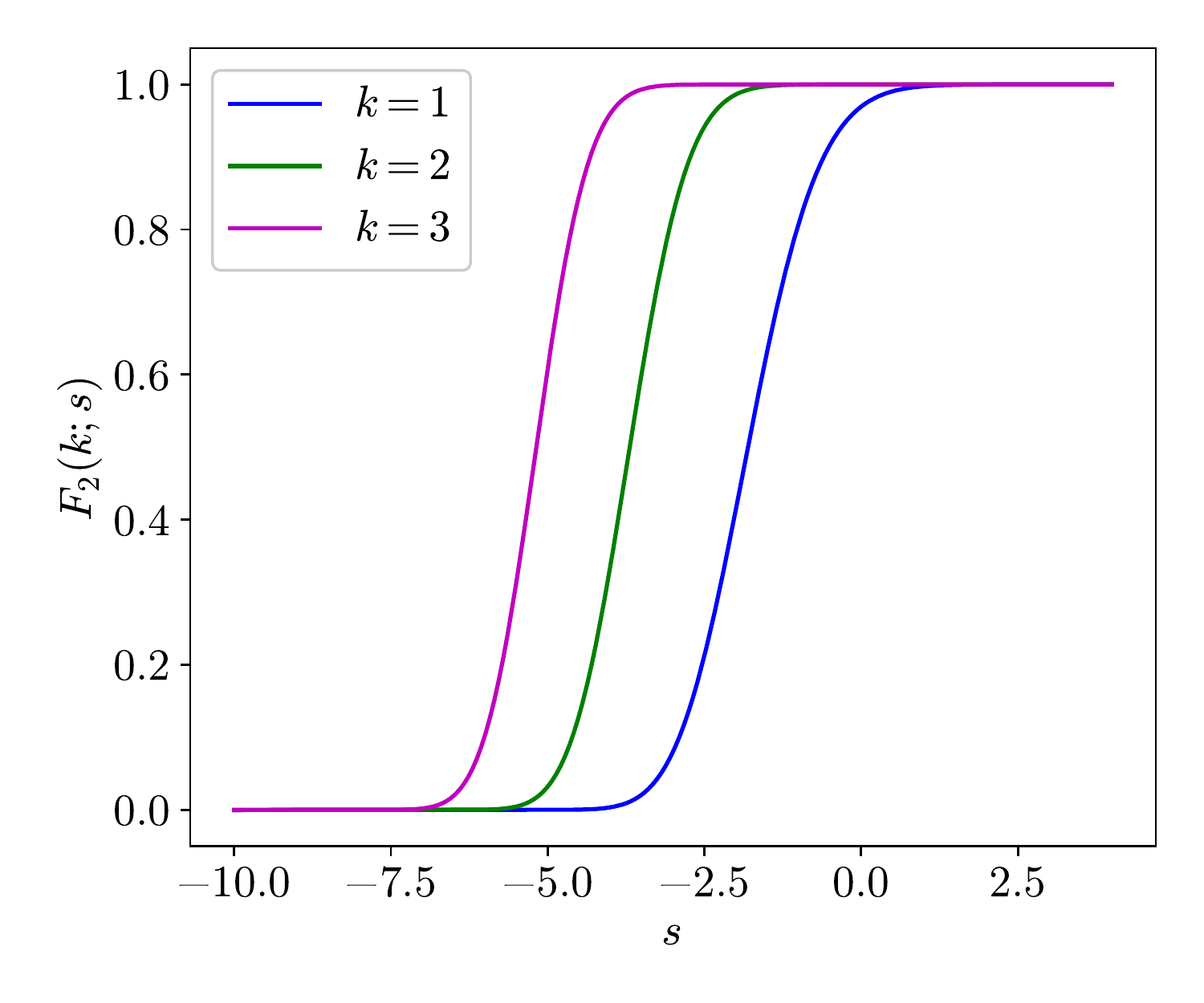}
    \end{subfigure}
  \caption{{\bf $\frac{d}{ds}F_2(k;s)$ and $F_2(k;s)$ for $k=1,2,3$}.}
  \label{twplot}
\end{figure}
\begin{observation}
From Tables \ref{pdfexp}, \ref{cdfexp} and Figure \ref{twplot}, it's clear
that our algorithm evaluates the distributions $\frac{d}{ds}F_2(k;s)$ and
$F_2(k;s)$ to relative accuracy everywhere, except in the left tail. The
algorithm only evaluates the left tail of the distributions to absolute
precision, since the leading eigenvalues of the Airy integral operator
$\T_s$ converge to $1$ as $s\to -\infty$ (see also Theorem \ref{thm:eig1}),
which leads to catastrophic cancellation in the computation of the
distributions (see formulas (\ref{f2k}), (\ref{df2k})). 
\end{observation}
\clearpage

\subsection{Computation of finite-energy Airy beams}

In this section, we compute the beam intensities for both the finite Airy
beams and the Airy eigenfunction beams constructed from
$\psi_{0,c}$, described in Section~\ref{optics}. In our experiments, we 
construct finite Airy beams and Airy eigenfunction beams 
with unit total energy, and with roughly the same intensity in their main lobes.
We demonstrate that the eigenfunction beams are more non-diffracting
than the finite Airy beams (see Figures \ref{fig:cm2_a0202},
\ref{fig:cm1_a0108}). We also plot the densities and beam intensities of
the Airy eigenfunction beams at $\xi=0$ for various values of the parameter
$c$ in Figure~\ref{fig:sigphieig}.

\begin{figure}[h]
    \centering
    \begin{subfigure}{0.6\textwidth}
      \centering
      \includegraphics[width=\textwidth]{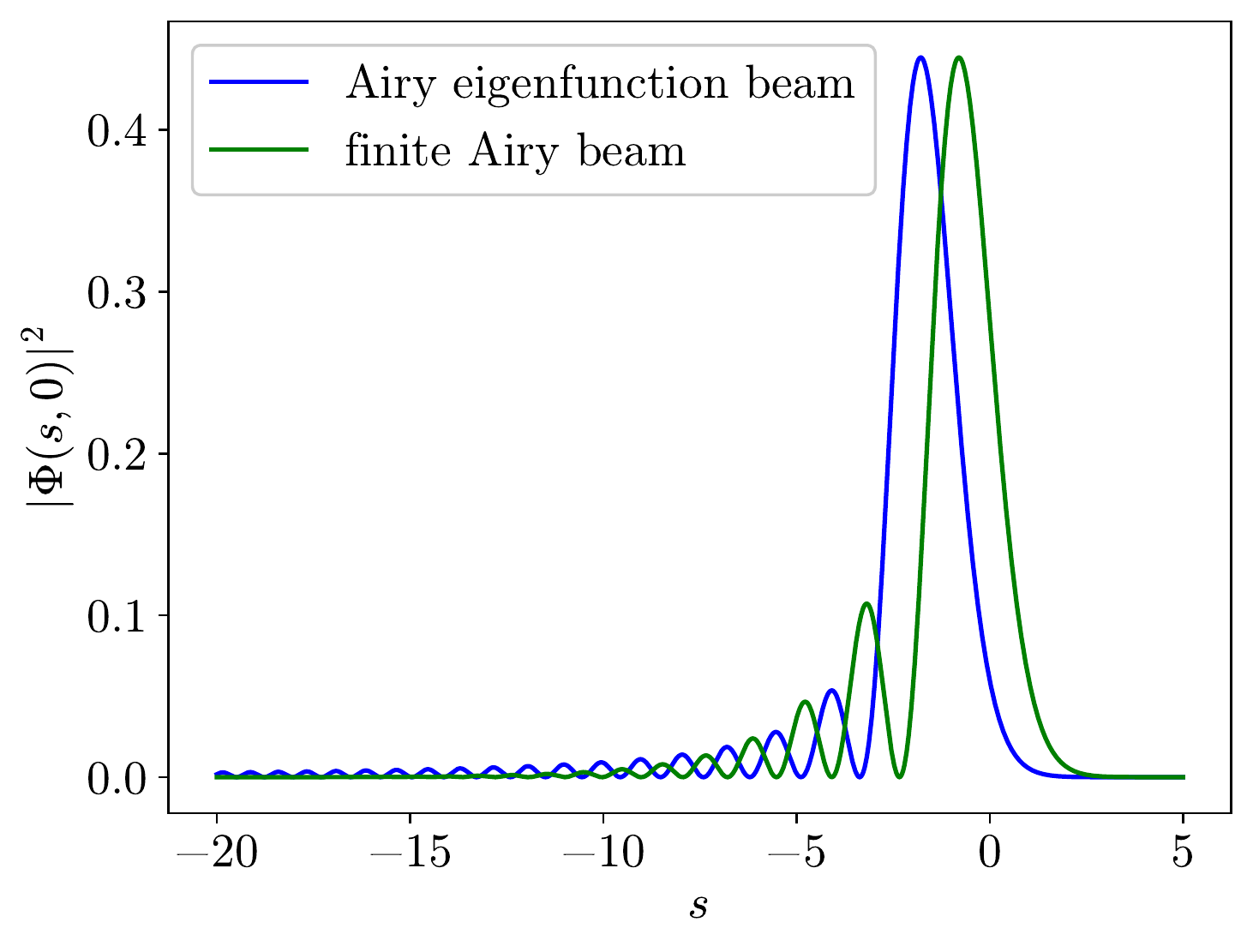}
      \caption{\label{fig:cm2_a0202_init}}
    \end{subfigure}
    \begin{subfigure}{0.99\textwidth}
      \centering
      \includegraphics[width=\textwidth]{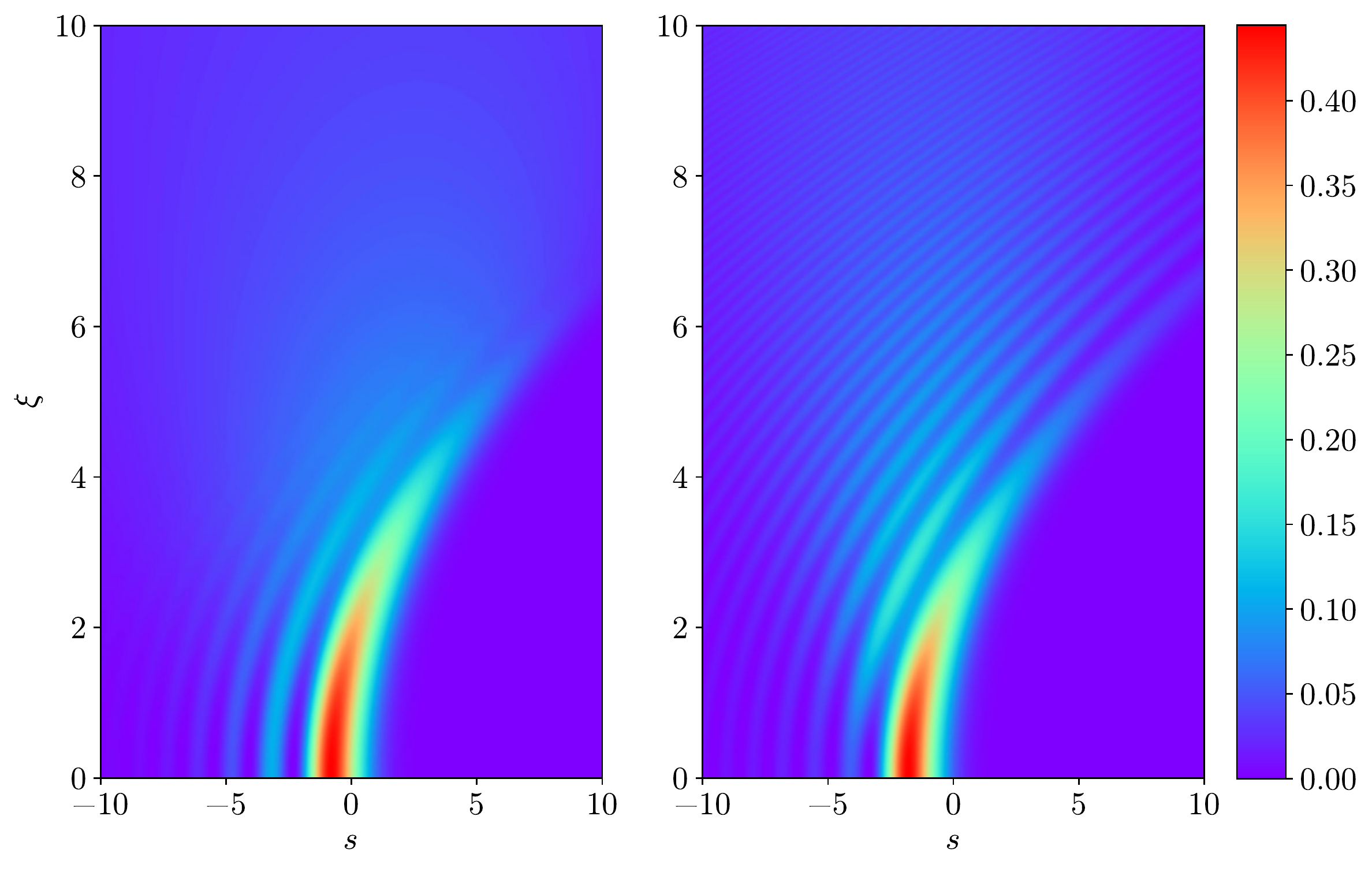}
      \caption{\label{fig:cm2_a0202_beam}}
    \end{subfigure}

  \caption{{\bf The intensity profiles of the finite Airy beam with
  parameter $\alpha=0.202$ and the Airy eigenfunction beam with
  parameter $c=-2$.}  The initial intensity profiles are shown in Figure
  (\subref{fig:cm2_a0202_init}). The beam profiles of the finite Airy beam and
  the Airy eigenfunction beam are shown on the left and right of Figure
  (\subref{fig:cm2_a0202_beam}), respectively.}
  \label{fig:cm2_a0202}

\end{figure}

\begin{figure}[h]
    \centering
    \begin{subfigure}{0.6\textwidth}
      \centering
      \includegraphics[width=\textwidth]{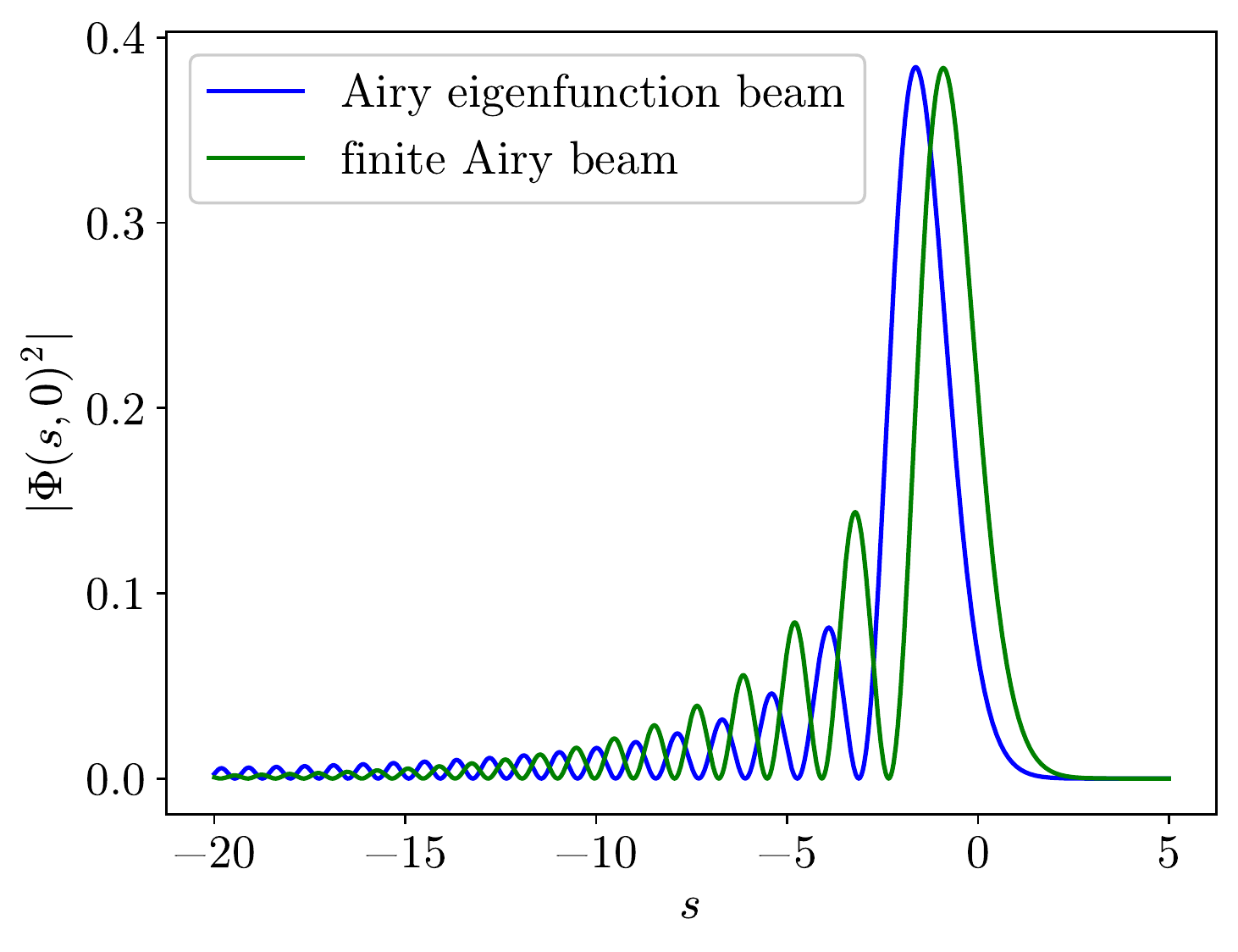}
      \caption{\label{fig:cm1_a0108_init}}
    \end{subfigure}
    \begin{subfigure}{0.99\textwidth}
      \centering
      \includegraphics[width=\textwidth]{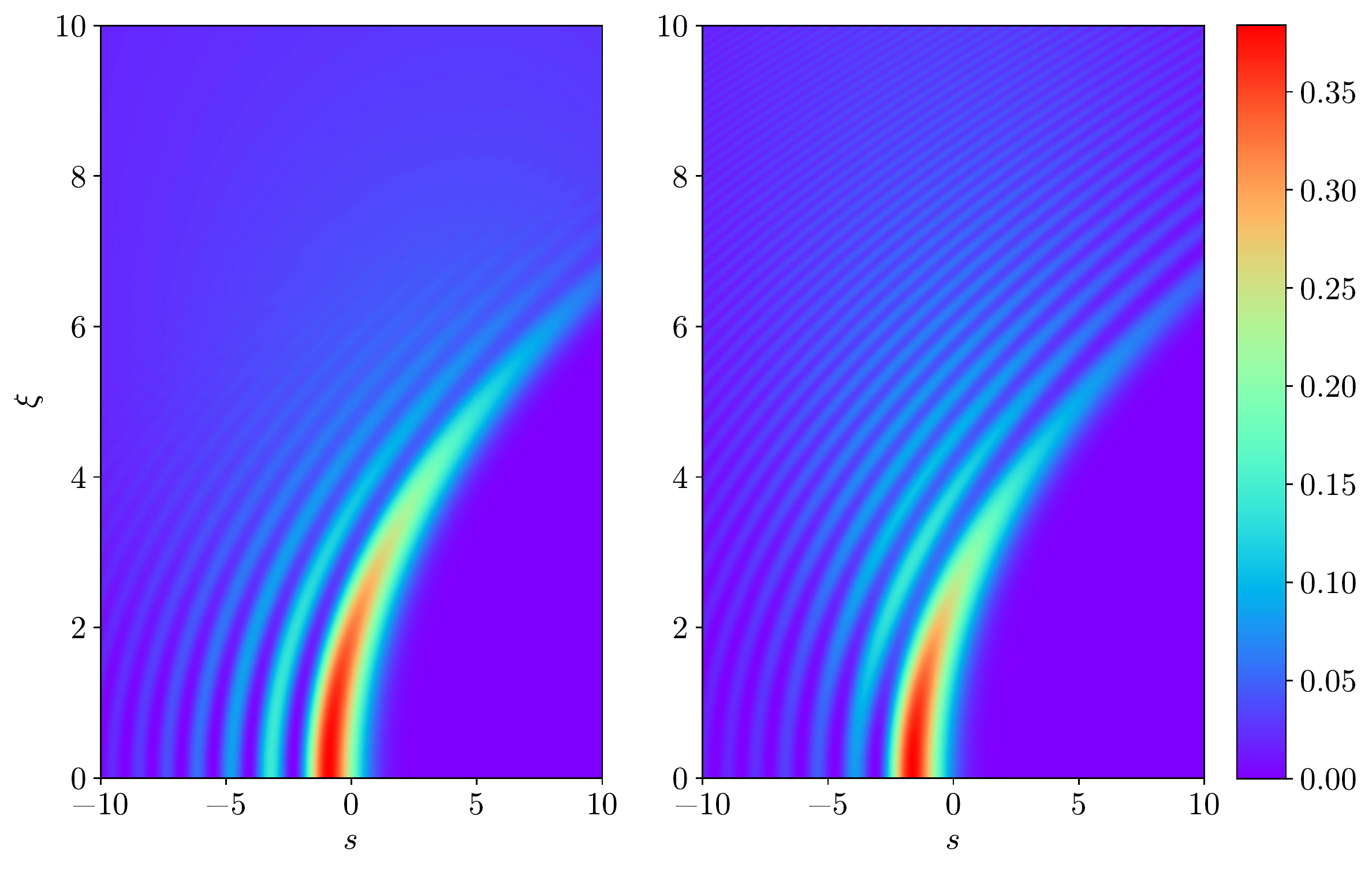}
      \caption{\label{fig:cm1_a0108_beam}}
    \end{subfigure}
  \caption{{\bf The intensity profiles of the finite Airy beam with
  parameter $\alpha=0.108$ and the optimal finite-energy Airy beam with
  parameter $c=-1$.} The initial intensity profiles are shown in 
  (\subref{fig:cm1_a0108_init}). The beam profiles of the finite Airy beam and
  the Airy eigenfunction beam are shown on the left and right of Figure
  (\subref{fig:cm1_a0108_beam}), respectively.}
  \label{fig:cm1_a0108}

\end{figure}

\begin{figure}[h]
    \centering
    \begin{subfigure}{0.49\textwidth}
      \centering
      \includegraphics[width=\textwidth]{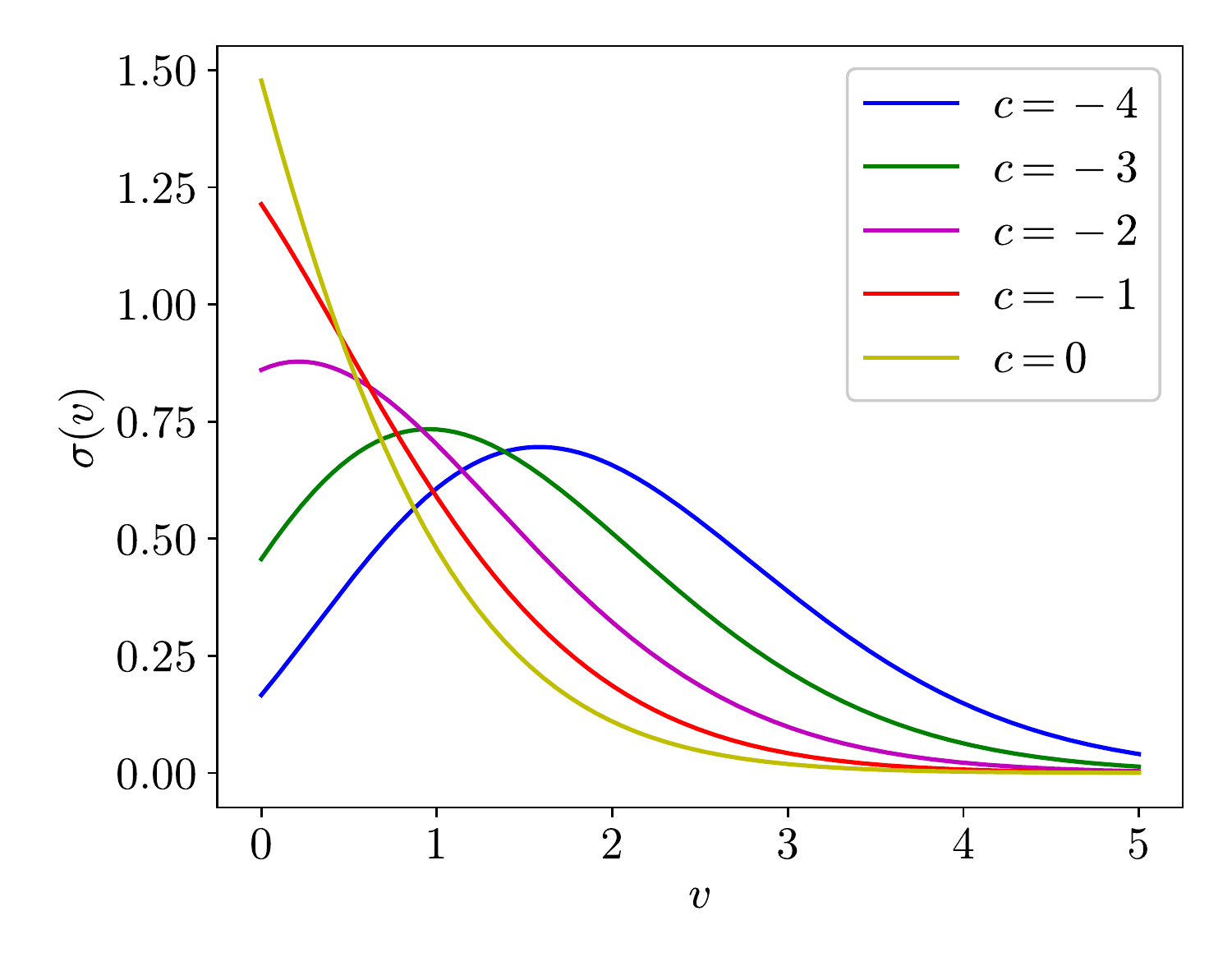}
      \caption{\label{fig:sigeig}}
    \end{subfigure}
    \begin{subfigure}{0.49\textwidth}
      \centering
      \includegraphics[width=\textwidth]{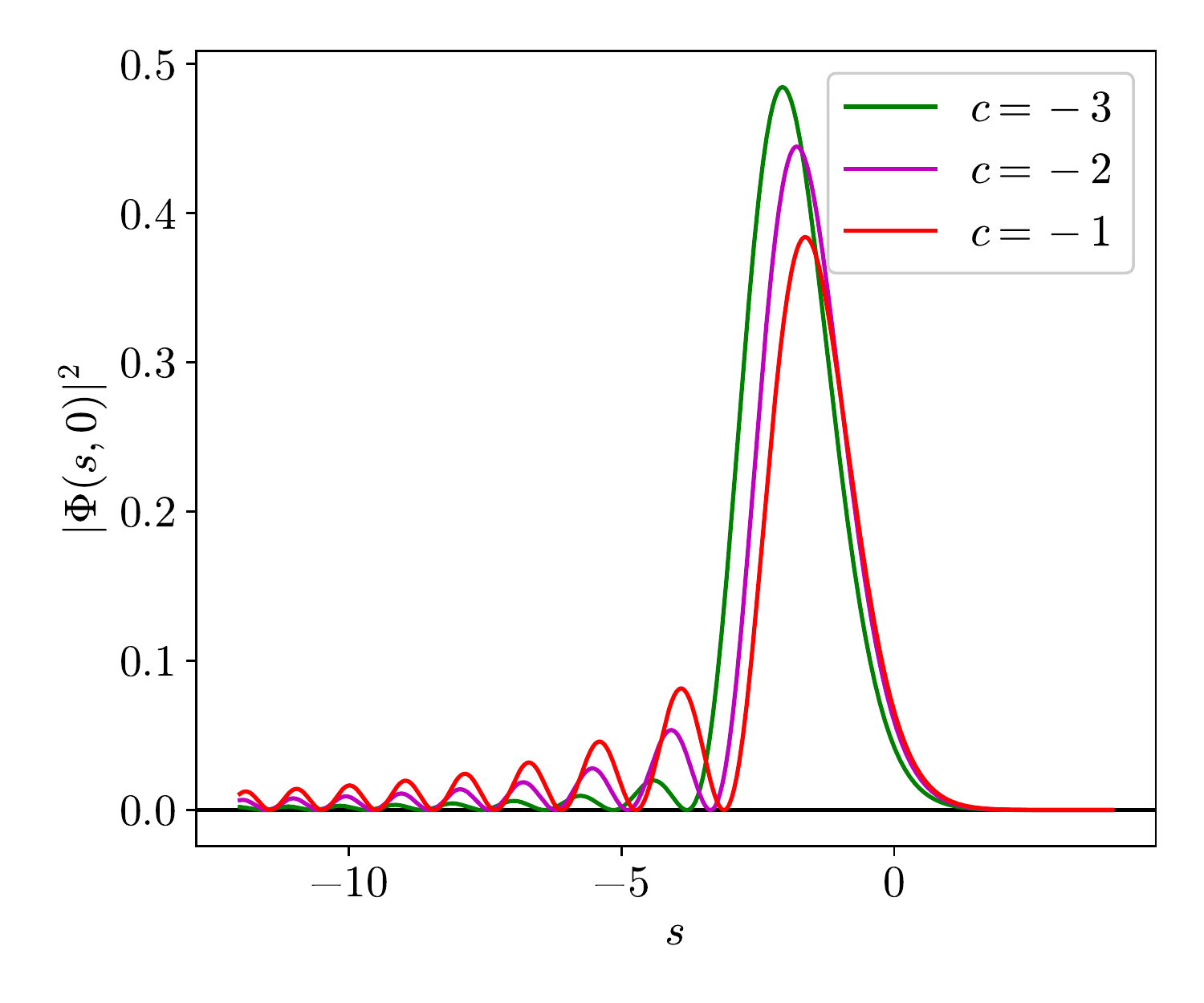}
      \caption{\label{fig:phieig}}
    \end{subfigure}
  \caption{{\bf The density functions and beam intensities of the Airy
  eigenfunction beams, constructed from $\psi_{0,c}$, for various values of
  the parameter $c$.} The density functions $\sigma(v)=\psi_{0,c}(v)$ and
  their corresponding beam intensities $\abs{\Phi(s,0)}^2$, defined by
  formula~(\ref{for:initpro}), are shown in Figures~(\subref{fig:sigeig})
  and~(\subref{fig:phieig}), respectively, for various values of the
  parameter $c$.  }
  \label{fig:sigphieig}

\end{figure}

\begin{observation}
From Figures \ref{fig:cm2_a0202} and \ref{fig:cm1_a0108}, it's clear
that the Airy eigenfunction beams exhibit the key characteristics of the Airy beam.
Moreover, the Airy eigenfunction beams do a better job in preserving 
fine structure than the finite Airy beams, which implies that
the Airy eigenfunction beams have both a better self-healing ability, and
stronger gradient forces.
\end{observation}
\clearpage

\section{Conclusions}

In this paper, we present a numerical algorithm for rapidly evaluating the
eigendecomposition of the Airy integral operator $\T_c$,
defined in~(\ref{Tc}). Our method computes the eigenvalues $\lambda_{j,c}$
of $\T_c$ to full relative accuracy, and computes the eigenfunctions
$\psi_{j,c}$ of $\T_c$ and $\L_c$ in the form of an expansion~(\ref{exp}) in
scaled Laguerre functions, where the expansion coefficients are also
computed to full relative accuracy. 
In addition, we characterize the previously unstudied eigenfunctions of the
Airy integral operator, and describe their extremal properties in relation
to an uncertainty principle involving the Airy transform. 

We also describe two applications. First, we show that this algorithm can be
used to rapidly evaluate the distributions of the $k$-th largest level at
the soft edge scaling limit of Gaussian ensembles to full relative precision
rapidly everywhere, except in the left tail (the left tail is computed to
absolute precision).  
Second, we show that the eigenfunctions of the Airy integral operator can be
used to construct finite-energy Airy beams that achieve the longest possible
diffration-free distances by spreading their energy as evenly as possible in
their side lobes, while also concentrating energy near the main lobes of
their initial profiles.

\section{Acknowledgements}
We sincerely thank Jeremy Quastel for his helpful advice and for our
informative conversations. The first author would like to thank Shen Zhenkun
and Gu Qiaoling for their endless support, and he is fortunate, grateful,
and proud of being their grandson.

\appendix

\section{Appendix: Miscellaneous Properties of the Airy integral operator and its
commuting differential operator}\label{sec:misc}

In this section, we describe miscellaneous properties of the eigenfunctions
$\psi_{n,c}$ of the operators $\T_c$ and $\L_c$, as well as properties of
the eigenvalues $\chi_{n,c}$ of the commuting differential operator $\L_c$, and
$\lambda_{n,c}$ of the Airy integral operators $\T_c$.

\subsection{Derivative of $\chi_{n,c}$ with respect to $c$}
\begin{theorem}
For all real number $c$ and non-negative integers $n$, 
    \begin{align}
\pp{\chi_{n,c}}{c}=\int_0^\infty x\bigl(\psi_{n,c}(x)\bigr)^2 \d x.
    \end{align}
\end{theorem}

\begin{proof}
By (\ref{eigode}), we have
    \begin{align}
\dd{}{x} \Bigl(x \dd{}{x}\psi_{n,c} \Bigr) -
(x^2+cx-\chi_n)\psi_{n,c}=0.\label{dchi2}
    \end{align}
With the infinitesimal change $c=c+h$, it follows that
$\chi_n=\chi_n+\epsilon$, $\psi_{n,c}(x)=\psi_{n,c}(x)+\delta(x)$.
Therefore, (\ref{dchi2}) becomes
    \begin{align}
\dd{}{x} \Bigl(x \dd{}{x}(\psi_{n,c}+\delta) \Bigr) -
\bigl(x^2+(c+h)x-(\chi_n+\epsilon)\bigr)(\psi_{n,c}+\delta)=0.\label{dchi3}
    \end{align}
After subtracting (\ref{dchi2}) from (\ref{dchi3}) and discarding
infinitesimals of second order or greater, (\ref{dchi3}) becomes
    \begin{align}
\L_c[\delta](x)-(hx-\epsilon)\psi_{n,c}(x)=0,\label{dchi4}
    \end{align}
where $\L_c$ is defined by (\ref{Lc}). Then, we multiply both sides of
(\ref{dchi4}) by $\frac{\psi_{n,c}(x)}{h}$ and integrate both sides over the
interval $[0,\infty)$, which gives us
    \begin{align}
\hspace{-2em}\frac{1}{h}\int_0^\infty\L_c[\delta](x)\psi_{n,c}(x)\d
x-\int_0^\infty x\bigl(\psi_{n,c}(x)\bigr)^2\d
x+\frac{\epsilon}{h}\int_0^\infty\bigl(\psi_{n,c}(x)\bigr)^2=0.\label{dchi5}
    \end{align}
Due to the self-adjointness of $\L_c$,
    \begin{align}
\frac{1}{h}\int_0^\infty\L_c[\delta](x)\psi_{n,c}(x)\d
x=\frac{1}{h}\int_0^\infty\delta(x)\L_c[\psi_{n,c}](x)\d x=0.\label{dchi7}
    \end{align}
By (\ref{dchi7}) and the fact that $\norm{\psi_{n,c}}_2=1$, in the
appropriate limit, (\ref{dchi5}) becomes
    \begin{align}
\pp{\chi_{n,c}}{c}=\int_0^\infty x\bigl(\psi_{n,c}(x)\bigr)^2 \d x.
    \end{align}
\end{proof}

\subsection{Recurrence relations involving the derivatives of eigenfunctions of different orders}
\begin{theorem}\label{thm:rec}
For all real numbers $c$, non-negative integers $n$, and $x\in[0,\infty)$,
    \begin{align}
    \label{rec1}
&\hspace{-2em}-k(k-1) \psi_{n,c}^{(k-2)}(x)  - k(c+2x) \psi_{n,c}
^{(k-1)}(x)+\left(\chi_{n,c}-c x -x^2\right) \psi_{n,c} ^{(k)}(x)\notag\\
&\hspace{-2em}+(k+1) \psi_{n,c} ^{(k+1)}(x)+x \psi_{n,c} ^{(k+2)}(x)=0,
    \end{align}
for all $k\geq 2$. Furthermore,
    \begin{align}
    \label{rec2}
\hspace{-2em}-(c+2x)\psi_{n,c}(x)+(\chi_{n,c}-cx-x^2)\psi_{n,c}'(x)+2\psi_{n,c}''(x)+x\psi_{n,c}^{(3)}(x)=0.
    \end{align}
In particular, for all positive real $c$, non-negative integers $n$,
    \begin{align}
&\hspace{-2em}\chi_{n,c}\psi_{n,c}(0)+\psi_{n,c}'(0)=0,\label{rec3}\\
&\hspace{-2em}\psi_{n,c}(0)\neq 0.\label{rec4}
    \end{align}
\end{theorem}

\begin{proof}
The identities (\ref{rec1}) and (\ref{rec2}) are immediately obtained by
repeated differentiation of (\ref{eigode}).
The identity (\ref{rec3}) is proved by substituting $x=0$ into (\ref{eigode}).
Finally, the identity (\ref{rec4}) can be easily verified via proof by contradiction.
\end{proof}

\begin{remark}
We can compute the initial conditions $\psi_{n,c}(x)$ and $\psi_{n,c}'(x)$ by
evaluating the truncated expansion (\ref{exptrun}) and its first derivative in
$\O(N)$ operations, where $N$ represents the number of the expansion
coefficients. The higher derivatives can then be calculated via identities
(\ref{eigode}), (\ref{rec2}) and (\ref{rec1}) in $\O(1)$ operations. This
theorem is useful for computing the Taylor expansion of $\psi_{n,c}$ at a
given point $x$.
\end{remark}

\begin{corollary}
For all positive real $c$, non-negative integers $m,n$,
    \begin{align}
\hspace{-2em}(\chi_{m,c}-\chi_{n,c})\psi_{m,c}(0)\psi_{n,c}(0)+\psi_{m,c}'(0)\psi_{n,c}(0)-\psi_{m,c}(0)\psi_{n,c}'(0)=0.
    \end{align}
\end{corollary}
\begin{proof}
The corollary follows directly from the identity (\ref{rec3}).
\end{proof}

\subsection{Expansions in eigenfunctions}
Given a real number $c$, the functions $\psi_{0,c},\psi_{1,c},\dots$ are a
complete orthonormal basis in $L^2[0,\infty)$.
Thus, every $f\in L^2[0,\infty)$ admits an expansion in the basis $\{\psi_{n,c}\}$.
In this subsection, we'll provide identities for the expansion coefficients of
$\psi_{n,c}', \psi_{n,c}'', x\psi_{n,c}$ and $\pp{\psi_n}{c}$, in the
basis $\{\psi_{n,c}\}$.

\begin{theorem}
\label{eseries}
For any real $c$, non-negative integers $m,n$, 
    \begin{align}
&\hspace{-5em}\int_0^\infty \psi_n'(x)\psi_m(x)\d
x=-\frac{\lambda_m}{\lambda_n+\lambda_m}\psi_n(0)\psi_m(0),\label{es1}
    \end{align}
and if $m\neq n$, then
    \begin{align}
&\hspace{-5em}\int_0^\infty \psi_n''(x)\psi_m(x)\d x=\frac{\lambda_m}{\lambda_n-\lambda_m}\Bigl(\psi_n'(0)\psi_m(0)-\psi_n(0)\psi_m'(0)\Bigr),\label{es2}\\
&\hspace{-5em}\int_0^\infty x\psi_n(x)\psi_m(x)\d x=\frac{\lambda_n\lambda_m}{\lambda_n^2-\lambda_m^2}\Bigl(\psi_n'(0)\psi_m(0)-\psi_n(0)\psi_m'(0)\Bigr),\label{es3}\\
&\hspace{-5em}\int_0^\infty \pp{\psi_n}{c}(x)\psi_{m}(x)\d x=\frac{\lambda_n\lambda_m}{\lambda_m^2-\lambda_n^2}\psi_m(0)\psi_n(0),\label{es4}
    \end{align}
where $\psi_m,\psi_n,\lambda_m,\lambda_n$ denote the eigenfunctions and
eigenvalues of the Airy integral operator with parameter $c$.
\end{theorem}

\begin{proof}
To prove (\ref{es1}), we start with the identity
    \begin{align}
\hspace{-3em}\lambda_n\int_0^\infty \psi_n'(x)\psi_m(x)\d
x=\int_0^\infty\Bigr(\int_0^\infty\dd{}{x}\ai(x+y+c)\psi_n(y)\d
y\Bigl)\psi_m(x)\d x.\label{es5}
    \end{align}
Note that 
    \begin{align}
\dd{}{x}\ai(x+y+c)=\dd{}{y}\ai(x+y+c).\label{es5.4}
    \end{align}
Therefore, the above calculation (\ref{es5}) can be repeated with $m$ and
$n$ exchanged, yielding the identity
    \begin{align}
\hspace{-3em}\lambda_m\int_0^\infty \psi_m'(x)\psi_n(x)\d
x=&\int_0^\infty\Bigr(\int_0^\infty\dd{}{x}\ai(x+y+c)\psi_m(y)\d
y\Bigl)\psi_n(x)\d x\notag\\
 \hspace{-3em}=&\int_0^\infty\Bigr(\int_0^\infty\dd{}{y}\ai(y+x+c)\psi_n(x)\d
x\Bigl)\psi_m(y)\d y.\label{es5.5}
    \end{align}
By combining (\ref{es5}) and (\ref{es5.5}), we get
    \begin{align}
\int_0^\infty \psi_n'(x)\psi_m(x)\d
x=\frac{\lambda_m}{\lambda_n}\int_0^\infty \psi_m'(x)\psi_n(x)\d
x.\label{es7}
    \end{align}
On the other hand, integrating the right side of (\ref{es7}) by parts and
rearranging the terms gives (\ref{es1}).

In the following, we assume that $m\neq n$. 

To prove formula (\ref{es2}), we first combine (\ref{aiode}),
(\ref{Teigfun}), and derive the following identity:
    \begin{align}
\lambda_n\psi_n''(x)=\int_0^\infty(x+y+c)\ai(x+y+c)\psi_n(y)\d y.\label{es9}
    \end{align}
By repeating the same procedure (\ref{es5})-(\ref{es7}), we get
    \begin{align}
\int_0^\infty \psi_n''(x)\psi_m(x)\d
x=\frac{\lambda_m}{\lambda_n}\int_0^\infty \psi_m''(x)\psi_n(x)\d
x.\label{es10}
    \end{align}
Integrating the right side of (\ref{es10}) by parts twice and rearranging
the terms gives (\ref{es2}).

To prove (\ref{es3}), first note that by combining (\ref{es9}) and
(\ref{Teigfun}), we get
    \begin{align}
\lambda_n\Bigl(\psi_n''(x)-(x+c)\psi_n(x)\Bigr) = \int_0^\infty
y\ai(x+y+c)\psi_n(y)\d y.\label{es12}
    \end{align}
Taking the inner product of both sides of (\ref{es12}) with $\psi_m(x)$, we have
    \begin{align}
\hspace*{-6em}\lambda_n\int_0^\infty\Bigl(\psi_n''(x)-(x+c)\psi_n(x)\Bigr)\psi_m(x)\d
x =&\int_0^\infty\int_0^\infty y\ai(x+y+c)\psi_n(y)\d y\, \psi_m(x)\d
x\notag\\
\hspace*{-6em}=&\int_0^\infty y\psi_n(y)\int_0^\infty \ai(y+x+c)\psi_m(x)\d
x\, \d y\notag\\
\hspace*{-6em}=&\,\lambda_m\int_0^\infty y\psi_n(y)\psi_m(y) \d y.\label{es14}
    \end{align}
Therefore, (\ref{es14}) becomes
    \begin{align}
\hspace*{-4em}(\lambda_m+\lambda_n)\int_0^\infty x\psi_n(x)\psi_m(x)\d x=
\lambda_n\int_0^\infty\Bigl(\psi_n''(x)-c\psi_n(x)\Bigr)\psi_m(x)\d x.
    \end{align}
By combining the orthogonality of $\psi_n$ and (\ref{es2}), we prove (\ref{es3}).

To prove (\ref{es4}), we take the derivative with respect to $c$ of both
sides of (\ref{Teigfun}), yielding the identity
    \begin{align}
\hspace*{-5em}\pp{\lambda_n}{c}\psi_n(x)+\lambda_n\pp{\psi_n}{c}(x)=\int_0^\infty
\Bigl(\frac{d}{dc}\ai(x+y+c)\psi_n(y)+\ai(x+y+c)\pp{\psi_n}{c}(y)\Bigr)\d
y.\label{es17}
    \end{align}
Taking the inner product of both sides of (\ref{es17}) with $\psi_m(x)$, by
(\ref{Teigfun}), we get
    \begin{align}
&\hspace*{-6em}\lambda_n\int_0^\infty \pp{\psi_n}{c}(x)\psi_m(x)\d x\notag\\
&\hspace*{-6em}=\int_0^\infty\Bigl(\int_0^\infty
\frac{d}{dc}\ai(x+y+c)\psi_n(y)\d y\Bigr)\psi_m(x)\d x
+\lambda_m\int_0^\infty\pp{\psi_n}{c}(y)\psi_m(y)\d y.\label{es19}
    \end{align}
Since
    \begin{align}
\dd{}{x}\ai(x+y+c)=\dd{}{c}\ai(x+y+c),\label{es20}
    \end{align}
we have that
    \begin{align}
\lambda_n\psi_n'(x)=\int_0^\infty\dd{}{c}\ai(x+y+c)\psi_n(y)\d y.\label{es22}
    \end{align}
Therefore, 
    \begin{align}
\hspace*{-3em}\int_0^\infty\Bigl(\int_0^\infty
\frac{d}{dc}\ai(x+y+c)\psi_n(y)\d y\Bigr)\psi_m(x)\d
x=\lambda_n\int_0^\infty \psi_n'(x)\psi_m(x)\d x.\label{es25}
    \end{align}
Finally, by combining (\ref{es1}), (\ref{es19}) and (\ref{es25}), we prove
(\ref{es4}).
\end{proof}

\subsection{Behavior of the eigenfunction $\psi_{n,c}$ as $c\to\infty$}
In this section, we show that the eigenfunction $\psi_{n,c}$ of the Airy
integral operator $\T_c$ converges to a scaled Laguerre function
in the limit as $c\to\infty$. 

\begin{theorem}\label{thm:lag}
As $c\to \infty$, 
  \begin{align}
\psi_{n,c}(x)\to h_n^{2\sqrt{c}}(x),
  \end{align}
where $h_n^{2\sqrt{c}}$ is the scaled Laguerre function defined in
(\ref{slag}) with parameter $a=2\sqrt{c}$. 
\end{theorem}

\begin{proof}
As $c\to\infty$, $\psi_{n,c}$ converges to the solution of 
  \begin{align}
\dd{}{x}(x\dd{}{x} f)-cxf=-\chi_{n,c}f,\label{newlode}
  \end{align}
by formula (\ref{eigode}) and the fact that $\psi_{n,c}$ becomes almost
compactly supported in the limit as $c\to\infty$ (see Theorem
\ref{thm:tpsiineq}). By comparing (\ref{newlode}) and (\ref{slagode}), we
conclude that
  \begin{align}
\lim_{c\to \infty}\psi_{n,c}(x)=h_{n}^{2\sqrt{c}}(x).
  \end{align}
\end{proof}

\begin{corollary}
$\lim_{c\to \infty}\chi_{n,c}=
(2n+1)\sqrt{c}.$
\end{corollary}

\begin{proof}
This corollary is an immediate consequence of formulas
(\ref{newlode}) and (\ref{slagode}).
\end{proof}

\subsection{Behavior of the eigenfunction $\psi_{n,c}$ as $c\to-\infty$}
In this section, we show that the eigenfunction $\psi_{n,c}$ of the Airy
integral operator $\T_c$ converges to a scaled and shifted Hermite function
in the limit as $c\to-\infty$. We first introduce the mathematical preliminaries.

The Hermite polynomials, denoted by $H_n\colon\R\to \R$, are defined
by the following three-term recurrence relation for any $k\geq 1$ (see
\cite{stegun}):
    \begin{align}
H_{n+1}(x)=2xH_n(x)-2nH_{n-1}(x),\label{hrec}
    \end{align}
with the initial conditions
    \begin{align}
H_0(x)=1,\quad H_1(x)=1-x.\label{hini}
    \end{align}
The polynomials defined by the formulas (\ref{hrec}) and (\ref{hini}) are an
orthogonal basis in the Hilbert space induced by the inner product
$\inner{f}{g}=\int_0^\infty e^{-x^2} f(x)g(x) \d x$, i.e.,
    \begin{align}
\inner{H_n}{H_m}=\int_0^\infty e^{-x^2}H_n(x)H_m(x) \d
x=\sqrt{\pi}2^nn!\delta_{n,m}.\label{hort}
    \end{align}
In addition, Hermite polynomials are solutions of Hermite's equation:
    \begin{align}
f''-2xf'+2nf=0.    
    \end{align}

We find it useful to use the scaled Hermite functions defined below.
\begin{definition}
Given a positive real number $a$, the scaled Hermite functions, denoted by
$\phi_n^a\colon\R\to \R$, are defined by
    \begin{align}
\phi_n^a(x)=\frac{\sqrt{a}}{\pi^{\frac{1}{4}}2^{\frac{n}{2}}(n!)^{\frac{1}{2}}}
e^{-a^2x^2/2}H_n(ax).\label{sherm}
    \end{align}
\end{definition}
The scaled Hermite functions satisfy the differential equation
    \begin{align}
-\frac{d^2}{dx^2}\phi_n^a  +a^4x^2\phi_n^a=a^2(2n+1)\phi_n^a.\label{shode}
    \end{align}

\begin{lemma}
Given a negative real number $c$, define
$g_n(u)=\psi_{n,c}((-\frac{c}{2})^{\frac{1}{4}}u-\frac{c}{2})$,
where $\psi_{n,c}$ is the $(n+1)$-th eigenfunction of the Airy integral
operator $\T_c$. Then, $g_n$ is the solution of the ODE
  \begin{align}
\hspace*{-2em}-\frac{d^2}{du^2} f - \bigl(-\frac{c}{2}\bigr)^{-\frac{3}{4}}\Bigl(u\frac{d^2}{du^2} f
+\frac{d}{du} f\Bigr) + u^2 f =
\bigl(\bigl(-\frac{c}{2}\bigr)^{\frac{3}{2}}+\bigl(-\frac{c}{2}\bigr)^{-\frac{1}{2}}\chi_{n,c}\bigr)f,
  \end{align}
where $\chi_{n,c}$ is the $(n+1)$-th eigenvalue of the commuting differential 
operator $\L_c$.

\end{lemma}

\begin{proof}
The lemma directly follows from the definition of $g_n$ and the differential
equation satisfied by $\psi_{n,c}$ (see formula (\ref{eigode})).
\end{proof}
\begin{theorem}\label{thm:herm}
As $c\to -\infty$, 
  \begin{align}
\psi_{n,c}(x)\to \frac{1}{\sqrt{a}}\phi_n^{a}\bigl(x+\frac{c}{2}\bigr),
  \end{align}
where $a=(-\frac{c}{2})^{-\frac{1}{4}}$. In other words, $\psi_{n,c}(x)$
converges a Hermite function that is translated by $-\frac{c}{2}$, and
scaled by scaling parameter $(-\frac{c}{2})^{-\frac{1}{4}}$.
\end{theorem}

\begin{proof}
As $c\to-\infty$, $g_n$ converges to the solution of 
  \begin{align}
-\frac{d^2}{du^2} f + u^2 f =
\bigl(\bigl(-\frac{c}{2}\bigr)^{\frac{3}{2}}+\bigl(-\frac{c}{2}\bigr)^{-\frac{1}{2}}\chi_{n,c}\bigr)f,
\label{gnode}
  \end{align}
since $\lim_{c\to -\infty} (-\frac{c}{2})^{-\frac{3}{4}} =0$. By comparing 
(\ref{gnode}) and $(\ref{shode})$, we conclude that
  \begin{align}
\lim_{c\to -\infty}g_n(u)=\phi_n^1(u).
  \end{align}
Therefore, by definition,
  \begin{align}
\lim_{c\to
-\infty}\psi_{n,c}(x)=\phi_n^1\bigl(\bigl(-\frac{c}{2}\bigr)^{-\frac{1}{4}}\bigl(x+\frac{c}{2}\bigr)\bigr)
=\frac{1}{\sqrt{a}}\phi_n^a\bigl(x+\frac{c}{2}\bigr),
  \end{align}
where $a=(-\frac{c}{2})^{-\frac{1}{4}}$.
\end{proof}

\begin{corollary}

$\lim_{c\to -\infty}\chi_{n,c}=
(2n+1)(-\frac{c}{2})^{\frac{1}{2}}-\frac{c^2}{4}.$
\end{corollary}

\begin{proof}
This corollary is an immediate consequence of formulas
(\ref{gnode}) and (\ref{shode}).
\end{proof}

\end{document}